\let\oldvec\vec% Store \vec in \oldvec
\documentclass[11pt]{llncs}

\let\vec\oldvec% Restore \vec from \oldvec
\usepackage{silence} %to suppress warnings
\usepackage{amsmath,amsfonts,amssymb,color,anysize,enumitem,graphicx,epstopdf,multirow,float,subcaption,todonotes,esvect}
\usepackage[margin=0.84in]{geometry}
\usepackage[pagebackref=true]{hyperref}

\newcommand{\pa}{\mathsf{par}}
\newcommand{\dgr}{\mathsf{deg}}
\newcommand{\dist}{\mathsf{dist}}

\DeclareMathOperator{\diam}{diam}
\DeclareMathOperator{\rc}{rc}

\DeclareMathOperator{\f}{f}

\DeclareMathOperator{\src}{src}
\DeclareMathOperator{\fvs}{fvs}

\newtheorem{rul}{Coloring Rule}
\newtheorem{prul}{Path Rule}
\newtheorem{observation}{Observation}
\newtheorem{invariant}{Invariant}

\newcommand{\calF}{\mathcal{F}}

\newcommand{\calT}{\mathcal{T}}

\newcommand{\one}[1]{(#1)_1}
\newcommand{\two}[1]{(#1)_2}
\newcommand{\dir}[1]{\vv{#1}}
\newcommand{\st}[2]{\mathtt{ST}(#1,#2)}

\newcommand{\ct}[2]{\mathtt{CT}(#1,#2)}
\newcommand{\und}[1]{\tilde{#1}}
\newcommand{\nth}{^\text{th}}

\let\doendproof\endproof
\renewcommand\endproof{~\hfill\qed\doendproof}

\graphicspath{{figs/}}
\title{Upper Bounding Rainbow Connection Number by Forest Number}
\author{L. Sunil Chandran \inst{1}\thanks{The major part of the work was done when this author was on a long-term research visit at Max Planck Institute for Informatics, Saarbr\"{u}cken, Germany. The visit was funded by the Alexander von Humboldt fellowship.} \and Davis Issac \inst{2}\thanks{The major part of the work was done when this author was a PhD student at Max Planck Institute for Informatics, Saarbr\"{u}cken, Germany.} \and Juho Lauri \thanks{The major part of the work was done when this author was at Bell Labs, Dublin, Ireland.}\and Erik Jan van Leeuwen \inst{3}}
\institute{
	Indian Institute of Science, Bangalore, India. \and
	Hasso Plattner Institute, Potsdam, Germany. \and
	Utrecht University, The Netherlands.
}
\begin{document}
\maketitle
\begin{abstract}
A path in an edge-colored graph is \emph{rainbow} if no two edges of it are colored the same, and the
graph is \emph{rainbow-connected} if there is a rainbow path between each pair of its vertices.
The minimum number of colors needed to rainbow-connect a graph $G$ is the \emph{rainbow connection number} of $G$, denoted by $\rc(G)$.

\hspace{24pt} A simple way to rainbow-connect a graph $G$ is to color the edges of a \emph{spanning tree} with  distinct colors  and
then re-use any of these colors to color the remaining edges of $G$. This proves that $\rc(G) \le  |V(G)|-1$. We ask whether there is
a stronger connection between tree-like structures and rainbow coloring than that is implied by the above trivial argument. 
For instance,  is it possible to find an upper bound of $t(G) -1$ for $\rc(G)$, where $t(G)$ is the number of vertices in the largest induced tree  of $G$?  The answer turns out to be negative, as there are counter-examples that show that even $c\cdot t(G)$ is not an upper bound for $\rc(G)$  for any given constant $c$.  

\hspace{24pt} In this work we show that if we consider the \emph{forest number} $\f(G)$, the number of vertices in  a maximum induced forest of $G$, instead of  $t(G)$, then surprisingly we do get an upper bound. More specifically, we prove
that ${\rc(G) \leq \f(G) + 2}$.
Our result indicates a stronger connection between rainbow connection and tree-like structures than that was suggested by the simple spanning tree based upper bound.
\end{abstract}
\section{Introduction}
Let $G$ be a connected, simple and finite graph.
Consider any edge-coloring of $G$.
A path in $G$ is said to be \emph{rainbow} if no two edges of it are colored the same.
The graph $G$ is {\bf rainbow-connected} if there is a rainbow path between each pair of its vertices.
If there is a rainbow \emph{shortest} path between every pair of its vertices, we say that $G$ is {\bf strongly rainbow-connected}.
The minimum number of colors required to rainbow-connect $G$ is known as the {
\bf rainbow connection number} of $G$, and denoted as $\rc(G)$.
Similarly, the minimum number of colors needed to strongly rainbow-connect $G$ is the {
\bf strong rainbow connection number} of $G$, denoted as $\src(G)$.
These measures of rainbow connectivity were introduced by Chartrand~{et al.}~\cite{Chartrand2008} in 2008.
The concept has gathered significant attention from both combinatorial and algorithmic perspectives. 
Indeed, the work of Chartrand~{et al.}~\cite{Chartrand2008} has already amassed more than 400 citations.
In addition to being a theoretically interesting way of strengthening the usual notion of connectivity, rainbow connectivity has potential applications in networking~\cite{Chakraborty2009}, layered encryption~\cite{Dorbec2014}, and broadcast scheduling~\cite{Joseph2013}.

While introducing the parameters, Chartrand~{et al.}~\cite{Chartrand2008} established basic bounds along with exact values of the parameters for some structured graphs.
To repeat their results, recall that the \emph{diameter} of $G$, denoted by $\diam(G)$, is the length of a longest shortest path in $G$.
Now, it is straightforward to verify that $\diam(G) \leq \rc(G) \leq \src(G) \leq m$, where $m$ is the number of edges of $G$.
In other words, both $\rc(G)$ and $\src(G)$ are always sandwiched between one and $m$.
The extremal cases are not difficult to see: $\rc(G) = \src(G) = 1$ if and only if $G$ is complete;
 $\rc(G) = \src(G) = m$ if and only if $G$ is a tree.
The authors also determined the exact rainbow connection numbers for cycle graphs, wheel graphs, and complete multipartite graphs.

Much of the research on rainbow connectivity has focused on finding bounds on the parameters, either in terms of the number of vertices $n$ or some other well-known parameters.
It follows that $\rc(G)\le n-1$ by taking a spanning tree and coloring its edges with distinct colors, and repeating an already used color for the other edges.
For 2-connected graphs, Ekstein~{et al.}~\cite{Ekstein2013} showed that $\rc(G) \leq \lceil n/2 \rceil$, and this is tight as witnessed by e.g., odd cycles. 
Further, it has turned out that domination is a useful concept when deriving upper bounds on $\rc(G)$ (see e.g.,~\cite{Caro2008:ejc,Krivelevich2010,Chandran2011}). 
Specifically, Krivelevich and Yuster~\cite{Krivelevich2010} showed that $\rc(G)\leq \tfrac{20n}{\delta}$, later improved by 
Chandran~{et al.}~\cite{Chandran2011} to $\rc(G) \leq \tfrac{3n}{\delta+1}+3$, where $\delta$ denotes the minimum degree of~$G$. 
Moreover, the latter authors derived that when $\delta \geq 2$, then $\rc(G) \leq \gamma_c(G) + 2$, where $\gamma_c(G)$ is the connected domination number. 
For some structured graph classes, this leads to upper bounds of the form $\rc(G) \leq \diam(G) + c$, where $c$ is a small constant. 
For instance, it follows that $\rc(G) \leq \diam(G) + 1$ when $G$ is an interval graph and $\rc(G) \leq \diam(G) + 3$, when $G$ is an AT-free graph, both bounds holding when $\delta \geq 2$.
Basavaraju et. al.~\cite{basavaraju2014rainbow} show that for every bridgeless graph $G$ with radius $r$, $\rc(G) \le r(r + 2)$,
and for a bridgelss graph with radius $r$ and chordality (length of a largest induced cycle) $k$, $\rc(G)\le  rk$.

In addition to domination, various authors (see e.g.,~\cite{Caro2008:ejc}) have noted trees to be useful in bounding $\rc(G)$.
As mentioned earlier, $\rc(G) \leq n - 1$ follows by coloring the edges of a spanning tree of $G$ with distinct colors.
Moreover, Kam{\v{c}}ev~{et al.}~\cite{Kamcev2015} proved that $\rc(G) \leq \diam(G_1) + \diam(G_2) + c$, where $G_1=(V,E_1)$ and $G_2=(V,E_2)$ are connected spanning subgraphs of $G$ and $c \leq | E_1 \cap E_2 |$.
For a more comprehensive treatment, we refer the curious reader to the books~\cite{Chartrand-book,Li2012-monograph} and the surveys~\cite{Li2012-survey,Li2017-survey} on rainbow connectivity.

In light of the above results, 
it makes sense to search for bounds on $\rc(G)$ in terms of other graph parameters, that possibly arise from ``tree-related'' and ``dominating" graph structures.
Intuitively, a graph structure that has both characteristics is a maximum induced forest of a graph.
Hence, the question arises whether one can bound $\rc(G)$ in terms of its {\bf forest number} $\f(G)$, the number of vertices in the largest induced forest in the graph.
We answer this in the affirmative by proving the following theorem.
\begin{theorem}
	\label{thm:rc-fc}
A connected graph $G$ with forest number $\f(G)$ has $\rc(G) \leq \f(G) + 2$.
\end{theorem}
Observe that the bound is tight up to an additive factor of $3$ due to trees that have $\rc(G)=n-1=\f(G)-1$. 
Our bound improves the upper bound of $n-1$ obtained by coloring the edges of a spanning tree in distinct colors, except when $\f(G)\ge n-2$. We leave as an open problem the question of whether the stronger upper bound of $\f(G)-1$ is true.

One might be tempted to conjecture a strengthening of our bound, namely that $\rc(G)$ is at most $t(G)$, the number of vertices in the largest induced \emph{tree} in the graph. 
However, this turns out to be not true.
To see this, one can consider a graph $G$ obtained by taking a $K_k$ for any $k \geq 3$ with a pendant vertex attached to each of its vertices. Then, we have that $\rc(G)=k$ whereas $t(G)=4$.

Finally, we note that the complement of an induced forest is a \emph{feedback vertex set}. The \emph{feedback vertex set number} is the size of the smallest set of vertices in a graph whose removal leaves an induced forest. Hence, Theorem~\ref{thm:rc-fc} directly implies the following.
\begin{corollary}
A connected graph $G$ with feedback vertex set number $\fvs(G)$ has $\rc(G) \leq |V(G)| - \fvs(G) + 2$.
\end{corollary}

\subsection{Overview of our Techniques}
\label{sub:overview_of_our_techniques}
We give the proof of the upper bound in three takes. In Take~1 in Section~\ref{sub:easy_weaker_bound}, we give a short proof that $\rc(G)\le3\f(G)+1$. For this, we first observe that given a connected dominating set $D$ such that each of the remaining vertices have at least two neighbors in $D$, it is easy to find a rainbow coloring with $|D|+1$ colors. Then we observe that any maximal induced forest $F$ can be turned into such a connected dominating set by adding at most $2|F|$ more vertices. The bound of $3\f(G)+1$ follows from these two observations.

In Take~2 in Section~\ref{sub:take2}, we strengthen our bound from $3\f(G)+1$ to $2\f(G)+2$. 
In this section, we already introduce the main structures and insights used for our final bound in Take 3. 
We fix a maximum induced forest $F$
and define $H$ to be the graph obtained from $G$ by contracting each connected component of $G-F$, each of which is a tree, of $G$ into a single vertex.
Thus $H$ consists of \emph{tree vertices} and \emph{non-tree vertices}.
An edge from a non-tree vertex $u$ to a tree-vertex $x_T$ is classified as a $2$-edge if $u$ has at least two edges to the tree $T$ (the tree that was contracted into the tree-vertex $x_T$), and as a $1$-edge otherwise.
We fix a carefully chosen spanning tree of $H$, root it at some (contracted) tree vertex, and direct all the edges towards root. We call this the \emph{skeleton} $B$.
The skeleton is chosen so that the number of $2$-edges in it is maximized.
The \emph{inner skeleton} $B_1$ is defined to be $B$ minus the leaves of $B$ that are non-tree vertices.

Our idea is to color all the edges of the forest $F$ with distinct colors and then associate each tree of $F$ with two additional colors called \emph{surplus colors}, and also keep aside two \emph{global surplus colors}. Note that this makes the total number of colors $2\f(G)+2$ as required. 
Then, we show that the $2$ surplus colors per tree are sufficient to color the edges of the inner skeleton $B_1$ so that there is a rainbow path between every pair of vertices in $B_1$, also giving a corresponding rainbow path in $G$ between every pair of vertices in $B_1$.
We show that if each tree $T$ can take care of the first three edges while following the outward path to the root, then we cover all edges of $B_1$. 
Although we only have two surplus colors, the three edges can indeed be taken care of.
For this, we show that the third edge is not taken care of by any other tree only when the first outgoing edge is a $2$-edge.
In this case, this out-going $2$-edge, say $x_Tu$ can be colored with the same color as some edge in the path between the $2$ neighbors of $u$ in $T$, while still maintaining a rainbow connection in $G$ between every pair of vertices in $B_1$.

After rainbow-connecting $B_1$, we connect the vertices outside of the inner skeleton to the inner skeleton using the two global surplus colors. From the choice of the skeleton, we have that each outer vertex has a $2$-edge to at least one tree-vertex. Thus, the outer vertices have at least two edges to one tree, that we color with the two different global surplus colors. Thus to get a rainbow path between two outer vertices $x$ and $y$, we can travel from $x$ into the inner skeleton using global surplus color 1, then use the path inside the inner skeleton, and then go to $y$ using global surplus color 2.

In Take $3$ in Section~\ref{sec:take3}, we improve the bound from $2\f(G)+2$ to $ \f(G)+2$. The improvement comes from the fact that we use only one surplus color per tree instead of two. In order to make the coloring of $B_1$ work with one surplus color per tree, we do a case analysis to color the edges around a vertex in $B_1$. The cases are differentiated mainly on the basis of the number of edges and the number of $2$-edges incident on a vertex.

%end of  subsection overview_of_our_techniques 	

\subsection{Preliminaries} 
\label{sec:notation}

%end of  section notation 
For a graph $G$, a subgraph $H$ of $G$, and any $E'\subseteq E(G)$, we use $E'(H)$ to denote $E'\cap E(H)$.
For a vertex $v$ of (di)graph $G$, we use $\deg_G(v)$ to denote the degree of $v$ in $G$.
We use $\dist_G(u,v)$ to denote number of vertices in any shortest path between $u$ and $v$ in $G$. 
For graph $G$ and $S\subseteq V(G)$, we define $G\setminus S:= G[V(G)\setminus S]$.
We use $uv$ for an edge between $u$ to $v$ and for a directed edge from $u$ to $v$, we use $\dir{uv}$. 
For the latter, we may omit the arrow, when the direction is not relevant.
For a directed graph $G$, we denote by $\und{G}$,
%\todo[inline]{Juho: I am strongly against the current form that places the tilde so high up on $G$ that it widens the lines. Can we change this to a more neutral notation?} 
the underlying undirected graph of it.
Since, for a forest $\calF$, each connected component is a tree, we will use the phrases ``tree of $\calF$" and ``connected component of $\calF$" analogously.
An \emph{in-arborescence} is a directed graph with a special root vertex such that all vertices have a unique directed path to the root vertex.
For a tree $T$ and vertices $u$ and $v$ in $T$,
we use $T_{uv}$ to denote the unique path in $T$ between $u$ and $v$.
\section{Proof of Theorem~\ref{thm:rc-fc}} 
\label{sec:proof_of_main_theorem}
Let $G=(V,E)$ be a connected graph.
Our goal is to prove that $\rc(G) \le \f(G)+2$.
In what is to follow, we proceed incrementally by starting from a weaker bound, namely that $\rc(G) \leq 3\f(G)-1$ in Section~\ref{sub:easy_weaker_bound}, and then improving upon this by proving that $\rc(G) \leq 2\f(G) + 2$ in Section~\ref{sub:take2}.
Finally, in Section~\ref{sec:take3}, to prove our final result, we use the ideas presented in these proofs in a more detailed and technically involved manner.
\subsection{Take 1: 
\texorpdfstring{$\rc(G) \le 3\f(G)-1$}{}}
\label{sub:easy_weaker_bound}
In this subsection, as a starting point, we prove that $\rc(G) \leq 3\f(G)-1$ which is considerably weaker but significantly less involved than our main result.
We begin with the following simple construction of a rainbow coloring.
\begin{lemma}
	\label{lem:dom}
	If there is a set $D\subseteq V$ such that $G[D]$ is connected and every vertex in $V\setminus D$ has at least two neighbors in $D$, then $\rc(G) \le |D|+1$.
\end{lemma}
\begin{proof}
	Since $G[D]$ is connected, it has at least one spanning tree $T$.
	Pick any spanning tree $T$ of $G[D]$ and color its edges with distinct colors from $1$ to $|D|-1$.
	By the second precondition of the lemma, every vertex $v \in V \setminus D$ has at least two neighbors in $D$; call them $d_1(v)$ and $d_2(v)$.
	Color the edge $vd_1(v)$ with color $|D|$ and $vd_2(v)$ with color $|D|+1$.
	%Color the remaining edges of $G$ with any arbitrary color in $[|D|+1]$.
	We will now prove that there is a rainbow path between any $u,v\in V(G)$.
	
	When $u,v \in V(D)$, the proof is trivial as $G[D]$ is rainbow-connected using only edges from $T$,
	all of which we colored distinctly.
	Similarly, when $u \in D$ and $v \in V \setminus D$, a rainbow path between the two uses only edges of $T$ plus either $vd_1(v)$ or $vd_2(v)$, both colored with a color not appearing on any edge in $T$.
	%By symmetry, the same argument applies when $u \in V \setminus D$ and $v \in D$.
	Finally, suppose both $u$ and $v$ are in $V \setminus D$.
	Observe that there is a rainbow path between $d_1(u)$ and $d_2(v)$ using only edges from $T$, particularly not using colors $|D|$ or $|D|+1$.
	Since $ud_1(u)$ is colored with color $|D|$ and $vd_2(v)$ is colored with color $|D|+1$, we have a rainbow path between $u$ and $v$, completing the proof.
\end{proof}
To proceed, let $\calF$ be a maximum induced forest of $G$.
Let $F = V(\calF)$, i.e., the set of vertices in $\calF$ and let $\calT$ denote the set of connected components of $\calF$.
The following structural observation will be useful for us.
\begin{lemma}
	\label{lem:tree2edge}
	For any maximum induced forest $\calF$ of $G$ and
	for any $v\in V\setminus F$,
	there exists a connected component $T$ of $\calF$ such that
	$v$ has at least two neighbors in $T$.
\end{lemma}
\begin{proof}
	Otherwise, $G[F \cup \{v\}]$ is a forest, contradicting the maximality of $\calF$.
\end{proof}
%Let us divide the set of trees $\calT$ into two: let $\calT_1$ be the set of trees that contain only $1$ vertex, and $\calT_2$ be the remaining trees.
%Now, we divide $F$ into $F_1$ and $F_2$: $F_1$ is the set of vertices that are contained in the trees in $\calT_1$, and $F_2$ is the set of vertices that are contained in the trees in $\calT_2$.
%Let $f_1:=|F_1|$ and $f_2:=|F_2|$.
The proof of the following lemma is folklore after observing that $F$ is a dominating set of $G$ (by the maximality of $\calF$), and proved here only for the sake of completeness.
\begin{lemma}
	\label{lemma:forest-connected}
	There exists a vertex set $A\subseteq V\setminus F$ of size at most $2|F|-2$ such that $G[F\cup A]$ is connected.
\end{lemma}
\begin{proof}
	Let $B$ be a smallest subset of $V\setminus F$ such that $G[F\cup B]$ has a smaller
	number of connected components than $G[F]$.
	We show that $|B|\le 2$.
	%Assume for the sake of contradiction that $|B|\ge 3$.
	From the minimality of $B$, it follows that $B$ induces a path in $G$ that connects two connected components of $\calF$.
	Let $b_1,b_2,\ldots, b_k$ be the vertices in this path, where
	 $b_1$ and $b_k$ are adjacent to some vertex in distinct connected components $T_1$ and $T_2$ respectively.
	Due to Lemma~\ref{lem:tree2edge}, $b_2$ has an edge to some connected component $T\in \calT$.
	If $T=T_1$, then $B\setminus \left\{ b_1 \right\}$ would also connect $T_1$ and $T_2$.
	Hence $T\neq T_1$.
	But then the vertex set $\left\{ b_1,b_2 \right\}$ is sufficient to connect $T$ and $T_1$ and thereby reduce the number of components.
	This implies $k=2$ and hence
	$|B|\le 2$.
	Now, we add $B$ to $F$ to reduce the number of connected components by at least one.
	Then we can repeat the process again until we get a single connected component.
	At each stage, the currently existing connected components take the role of the $T_i$'s.
	We only need to repeat the process at most $|F|-1$ times until we get a single connected component.
	Since each repetition adds at most $2$ vertices, the total number of vertices that we add is at most $2|F|-2$.
\end{proof}
We then arrive at the following conclusion.
\begin{theorem}
	Any connected graph $G$ has $\rc(G) \leq 3\f(G) - 1$.
\end{theorem}
\begin{proof}
%Lemmas~\ref{lem:dom}, \ref{lem:tree2edge} and Lemma~\ref{lemma:forest-connected} together imply that $\rc\le 3f-1$ as follows.
By Lemma~\ref{lemma:forest-connected}, we have a set $A$ such that $G[F\cup A]$ is connected and $|A|\le 2|F|-2$.
Let $D=F\cup A$.
By Lemma~\ref{lem:tree2edge}, we have that each vertex in $V\setminus F$ has at least two neighbors in $F$. 
This implies that each vertex in $V\setminus D$ has at least two neighbors in $D$. 
Thus, $D$ satisfies both preconditions of Lemma~\ref{lem:dom} and hence 	
$$\rc(G) \le |D|+1\le |F|+|A|+1\le |F|+(2|F|-2)+1=3\f(G)-1.$$
This completes the proof.
\end{proof}

\subsection{Take 2: 
\texorpdfstring{$\rc(G) \le 2\f(G)+2$}{}}
\label{sub:take2}
In this subsection we strengthen the previous upper bound to $2\f(G)+2$.
Many of the concepts and techniques that we use in the final proof are already introduced here.

As in the previous subsection, let $\calF$ be a maximum induced forest of $G$. 
Let $F = V(\cal F)$ be the set of vertices in $\cal F$.
Let $\calT$ be the set of connected components (trees) of $\calF$ and let $t=|\calT|$. 
Let $S:= V\setminus F$. 
Also, let $f=|V(F)|=\f(G)$.
We call an edge $uv$ of G a {\bf tree-edge} if both $u$ and $v$ belong to the same tree in $\mathcal{T}$; otherwise, the edge is called a {\bf non-tree edge}.

\begin{figure}[t]
    \centering
    \begin{subfigure}{0.49\textwidth}
    \centering
        \includegraphics[scale=0.72,keepaspectratio]{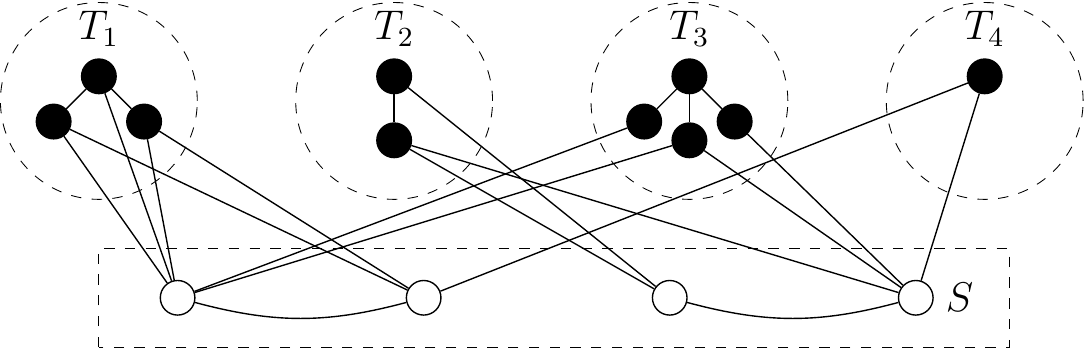}
        \caption{}
        %\label{fig:forest}
    \end{subfigure}%
    \hfill
    \begin{subfigure}{0.49\textwidth}
    \centering
		\includegraphics[scale=0.72,keepaspectratio]{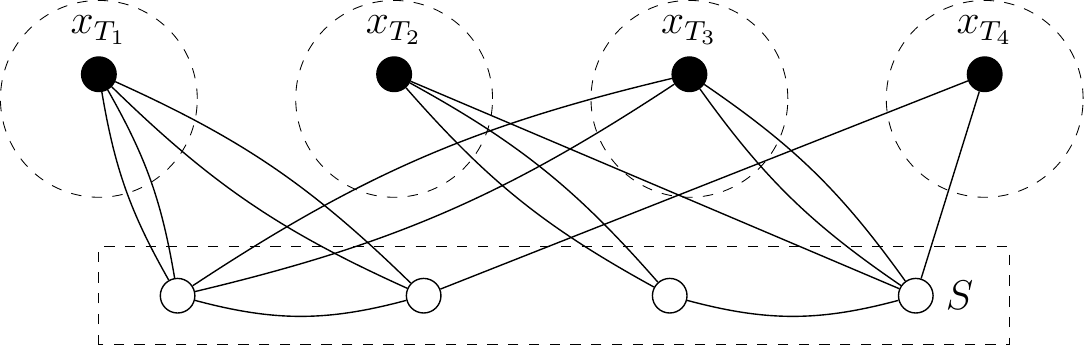}
		\caption{}
        %\label{fig:forest-contracted}
    \end{subfigure}
	\caption{\textbf{(a)} A graph $G$ is partitioned into a maximum induced forest $\calF$ and $S = V(G) \setminus V(\calF)$. The connected components (trees) of $\calF$ are $T_1,T_2,T_3$ and $T_4$. The edges between two black vertices, corresponding to vertices in $V(\calF)$, are tree-edges. \textbf{(b)} The graph $H$ obtained after contracting the connected components of $\calF$. We draw a 2-edge with $2$ lines and a 1-edge with a single line.}
	\label{fig:take2}
\end{figure}

\begin{figure}[b]
	\centering
	\begin{subfigure}{0.49\textwidth}
    \centering
	\includegraphics[scale=0.8,keepaspectratio]{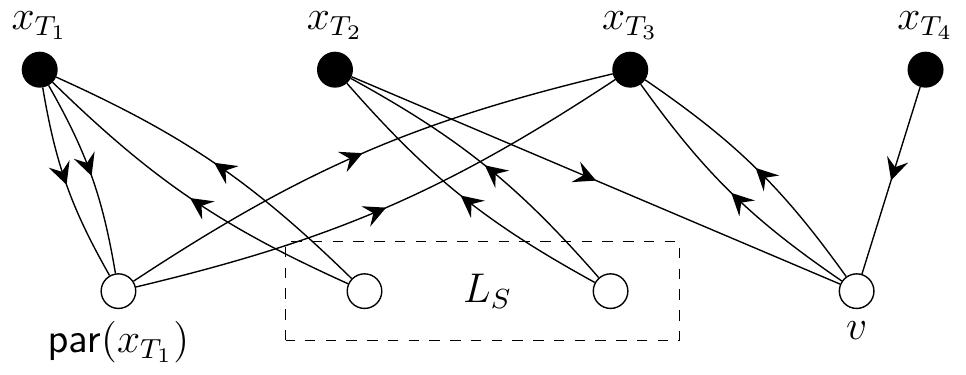}
		\caption{}
        %\label{fig:skeleton}
    \end{subfigure}%
    \hfill
	\begin{subfigure}{0.49\textwidth}
    \centering
		\includegraphics[scale=0.8,keepaspectratio]{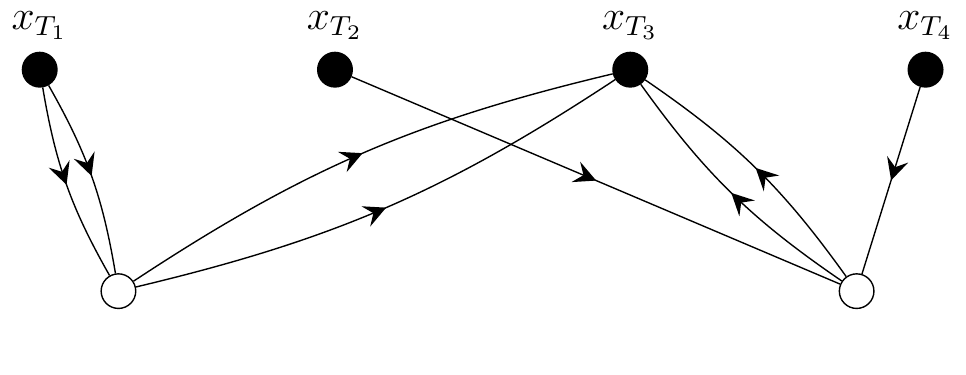}
        \caption{}
        %\label{fig:skeleton-b1}
    \end{subfigure}
	\caption{\textbf{(a)} A skeleton $B$, where $x_{T_3}$ is the root. For vertex $v$, $x_{T_2}$ and $x_{T_4}$ are the children. \textbf{(b)} The inner skeleton $B_1 := B[V(B) \setminus L_S]$.}
    \label{fig:skeleton}
\end{figure}
\begin{figure}[t]
\centering
\includegraphics[scale=1.25,keepaspectratio]{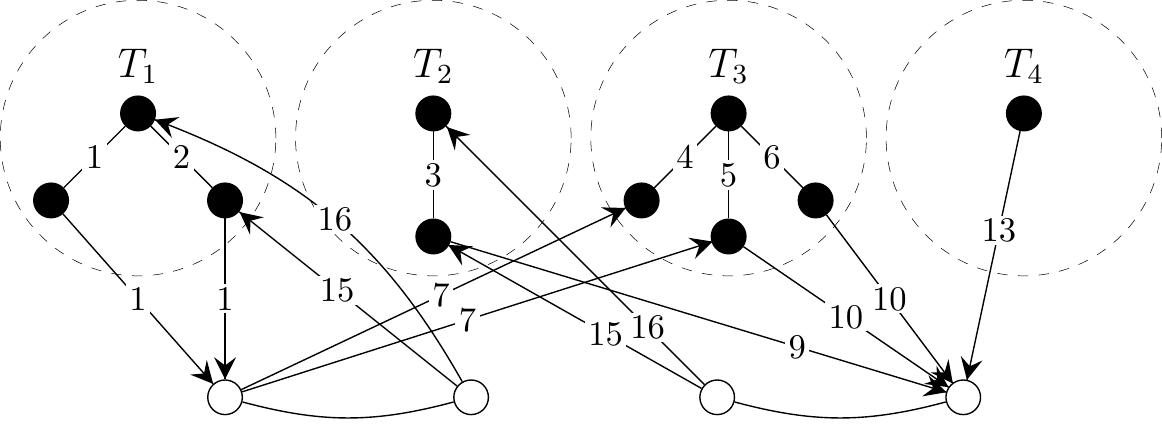}
\caption{The rainbow coloring from Take~2 for the graph $G$ in Figure~\ref{fig:take2}. As an example, the color-giving edge of $T_1$ is the edge colored with~1. Only relevant edges are shown. The direction of edges across $S$ (the white vertices) and $\calF$ (the black vertices) drawn according to the direction in $B_1$. The surplus colors are $s_1(T_1)=7$, $s_2(T_1)=8$, $s_1(T_2)=9$, $s_2(T_2)=10$, $s_1(T_3)=11$, $s_2(T_3)=12$, $s_1(T_4)=13$, and $s_2(T_4)=14$ while the global surplus colors are $g_1=15$ and $g_2=16$.}
\label{fig:take2-coloring}
\end{figure}

\begin{figure}[b]
\centering
\includegraphics[scale=1.0,keepaspectratio]{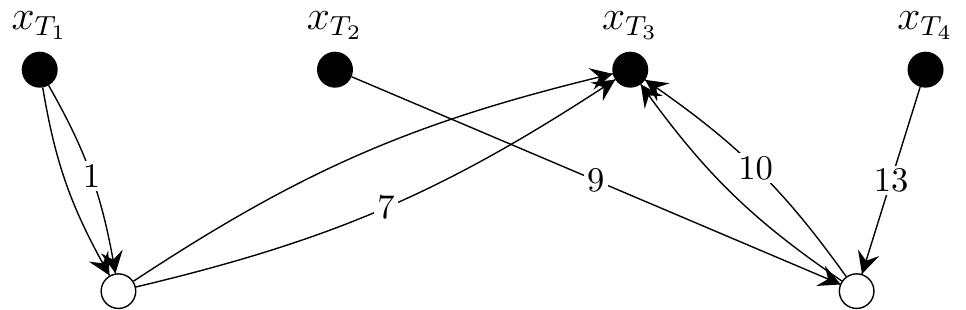}
\caption{Coloring of $B_1$ according to the coloring procedure in Take~2.}
\label{fig:skeleton-b1-coloring-take2}
\end{figure}

Let $H$ be the graph obtained from $G$ by contracting each connected component of $\calF$ to a single vertex
(see Figure~\ref{fig:take2}).
Formally, we define $H$ as:
\begin{align*}
	V(H):= &\; V_{\calT} \cup  S, \;\text{where}\\
	V_{\calT}:= &\left\{ x_T:T\in \calT \right\}\text{and} \\
	E(H):= 
	&\; E(G[S])\cup \left\{ ux_T: u\in S,\; T\in \calT,\; u \text{ has at least one edge to } V(T) \text{ in } G\right\}.
\end{align*}
We call the vertices in $V_{\calT}$ the {\bf tree vertices} and the vertices in $S$ the \emph{\bf non-tree vertices} of $H$.
Notice that $V_{\calT}$ is an independent set in $H$, because there are no edges in $G$ between any two distinct connected components of $\calF$. 
We partition the edges of $H$ into the following two sets:
\begin{align*}
	E_1:= &\; E(G[S])\cup \left\{ ux_T: u\in S,\; T\in \calT,\; u \text{ has exactly one edge to } V(T) \text{ in } G\right\}\text{and}\\
	E_2:= &\; \left\{ ux_T: u\in S,\; T\in \calT,\; u \text{ has at least two edges to } V(T) \text{ in } G\right\}.
\end{align*}
The edges in $E_1$ are called {\bf 1-edges} while those in $E_2$ are \emph{\bf 2-edges.}
%From now on, whenever we say a 1-edge or a 2-edge, it is implicitly assumed that we are talking about an edge in $H$.
See Figure~\ref{fig:take2} for an illustration of the above definitions.
We define a function $f_{\calT}:V_{\calT}\rightarrow \calT$ that maps a tree vertex to its corresponding tree, i.e., $f_\calT(x_T)=T$. 
For each edge in $H$, we define its {\bf representatives} in $G$ as follows.
Consider first a 2-edge $e$ between $u\in S$ and $x_T\in V_{\calT}$. By definition of a 2-edge, $u$ has at least two edges to $V(T)$ in $G$.
We arbitrarily choose two of these edges as the {representatives} in $G$ of the 2-edge $e$ in $G$ and denote them by $\one{e}$ and $\two{e}$.
%For a $1$-edge $\{u,T\}$, there is single edge between $u$ and $V(T)$ in $G$, which we denote by $\{u,T\}_1$.
For a $1$-edge $e$ between $u\in S$ and $x_T\in V_{\calT}$, 
there is a unique edge between $u$ and $V(T)$ in $G$, 
by the definition of a $1$-edge.
We call this edge the representative of $ux_T$ in $G$, and denote it $\one{e}$.
For a $1$-edge $e$ between $u\in S$ and $v\in S$, we call $uv$ its own representative in $G$.
For simplicity, we might simply say \emph{representatives} instead of representatives in $G$.
Whenever we say a representative, it is implicitly assumed that we are talking about an edge in $G$.
For a 2-edge $ux_T$ with representatives $uv_1$ and $uv_2$, we call the vertices $v_1$ and $v_2$, the {\bf foots} of $ux_T$.
The unique path between $v_1$ and $v_2$ in $T$ is called the {\bf foot-path} of $ux_T$.

A {\bf skeleton} is an in-arborescence obtained by taking a spanning tree of $H$ with an arbitrary node of $V_{\calT}$ fixed as its root with all edges directed towards the root. 
%A {\bf skeleton} is defined as a spanning tree of $H$ with one of the nodes in $V_{\calT}$ fixed as its root and all the edges of it directed towards the root.
Let $B$ be a skeleton such that the number of 2-edges in $B$ is as large as possible (or equivalently, the number of $1$-edges is as small as possible, as the total number of edges in a skeleton is always $|V(H)|-1$).
The {\bf parent} of a non-root vertex $v$ in $B$, denoted by $\pa(v)$, is the unique out-neighbor of $v$ in $B$.
The {\bf children} of $v$ are the in-neighbors of $v$ in $B$.
Whenever we say the parent (or child), we mean the parent (or child) in $B$, even if $B$ is not mentioned explicitly.
We call a directed edge $\dir{uv}$ in $B$, a $1$-edge (or $2$-edge respectively),
if $uv$ is a $1$-edge (or $2$-edge respectively) in $H$.
Let $L_S$ be the set of vertices of $S$ that are leaves of $B$ and let $B_1$ be the sub-arborescence of $B$ defined as $B_1:=B[V(B)\setminus L_S]$.
We call $B_1$ the {\bf inner skeleton}. Let $\und{B_1}$ be the underlying undirected tree of $B_1$.
These concepts are illustrated in Figure~\ref{fig:skeleton}.

%Notice that all leaves of $B_1$ are from $V_{\calT}$.
We now prove a lemma that is useful for our coloring procedure.
\begin{lemma}
	\label{lem:vertex2edge}
Every vertex in $S$ has at least one 2-edge incident on it in $B$.	
\end{lemma}
\begin{proof}
	Suppose for the sake of contradiction that $v$ is a vertex in $S$ that has only $1$-edges incident on it in $B$.
	By Lemma~\ref{lem:tree2edge}, there exist a $T\in \calT$ such that $v$ has at least two edges to $T$ in $G$.
	Therefore, $vx_T$ is a $2$-edge in $H$.
	Let $C$ be the connected component of $B\setminus v$ that contains the vertex $x_T$.
	Let $e$ be the unique edge in $B$ between $v$ and $C$.
	Removing $e$ from $B$ and adding 2-edge $vx_T$ gives a skeleton with higher number of 2-edges than $B$.
	This is a contradiction to the choice of $B$.
\end{proof}
The above lemma has the following corollaries.
\begin{corollary}
	\label{cor:L2edge}
	For every vertex in $L_S$, the unique edge incident on it in $B$ is a 2-edge.
\end{corollary}
\begin{corollary}
	\label{cor:leafB1}
	Every leaf of $B_1$ is a tree vertex.
\end{corollary}
\begin{proof}
	Suppose for the sake of contradiction that there is a leaf $v$ of $B_{1}$ that is a non-tree vertex.
	Clearly, $v\notin L_S$ by the definition of $B_1$. 
	Hence, $v$ is not a leaf of $B$.
	Then, there must be a vertex $u$ in $L_S$ that has an edge to $v$ in $B$.
	Since both $u$ and $v$ are in $S$, the edge $uv$ is a $1$-edge.
	This is a contradiction to Corollary~\ref{cor:L2edge}.
\end{proof}
\begin{corollary}
	\label{cor:1edge-child}
	For each $1$-edge $\dir{uv}$ in $B_1$, either $u$ is a tree vertex, or a child $u'$ of $u$ in $B_1$ is a tree vertex with $u'u$ being a 2-edge.
\end{corollary}
\begin{proof}
Suppose that $u$ is not a tree vertex.
Then there is an incoming 2-edge on $u$ in $B$, because its outgoing edge is a $1$-edge and there has to be at least one 2-edge incident on it due to Lemma~\ref{lem:vertex2edge}.
Let the other endpoint of this edge be $u'$.
Since at least one of the endpoints of a 2-edge has to be a tree vertex, $u'$ is a tree vertex.
Since $u'$ is a tree vertex, it has to be in $B_1$.
%Thus $w$ is a child of $u$ that is a tree vertex.
\end{proof}
\noindent
{\bf Coloring procedure:} We now give a rainbow coloring of $G$ using $f+t+2\le 2f+2$ colors (recall that $t=|\calT|$ is the number of connected components in $\calF$ and $f=\f(G)=|V(F)|$). 
The coloring procedure is illustrated in Figure~\ref{fig:take2-coloring}.
Since $\calF$ is a forest with $t$ connected components, the number of edges in $\calF$ is $f-t$.
Color all the edges in $\calF$ with distinct colors $1,2,\ldots, f-t$. 
We call colors $g_1:=f+t+1$ and $g_2:=f+t+2$, the {\bf global surplus colors}. 
We use the global surplus colors to color the representatives of those edges of $B$ that are incident on the vertices in $L_S$. 
Each vertex in $L_S$ has only one edge incident on it in $B$, and this edge is a 2-edge due to Corollary~\ref{cor:L2edge}.
Color one of the representatives of this 2-edge with $g_1$ and the other with $g_2$.
Now there are $2t$ unused colors, which are the colors $f-t+1$ to $f+t$. 
We allocate each tree in $\calT$ (i.e., each connected component of $\calF$), two of these colors as its {\bf surplus colors}.
That is, the $i\nth$ tree $T$ in $\calT$ is allocated colors
$f-t+2(i-1)+1$ and $f-t+2(i-1)+2$ as its surplus colors.
We denote the two surplus colors of $T$ by $s_1(T)$ and $s_2(T)$.
%Now pick an edge in $B$ whose representatives in $G$ are not colored so far. 
%Let $(u,v)$ be the picked directed edge.\\
%\underline{Case 1: If $u$ is a tree-vertex (i.e. $u\in \calT$) and $uv$ is a $2$-edge}:
%Let $T=u$.
%Let $v_1$ and $v_2$ be the foots of $uv=Tv$ in $T$.
%Pick any edge $e$ in the path between $v_1$ and $v_2$ in $T$.
%Note that $e$ is already colored in $G$ as we have already colored all edges in $F$. 
%Give $uv_1$ and $uv_2$ the same color as that of $e$.\\
%\underline{Case 2: If $u\in S$ and $uv$ is a $1$-edge} 
%If the out-edge in $B$ of $T_i$ is a $2$-edge then color both edges $(T_i\par(T_i))_1$ and $(T_i\par(T_i))_2$ with color  

Finally, we give a coloring of the edges of $B_1$ and then extend the coloring to their representatives in $G$ (see Figure~\ref{fig:skeleton-b1-coloring-take2}).
Whenever we color an edge $e$ in $B_1$ with color $c$, we also color the representatives of $e$ in $G$ also with $c$, though we may not mention this explicitly.
Pick each $T\in \calT$ such that $x_T$ is not the root of $B$ and do the following:
Let $\dir{x_Tv}$ be the outgoing edge of vertex $x_T$ in $B$, i.e., $v=\pa(x_T)$; 
let $w=\pa(v)$ (if parent of $v$ exists, i.e., if $v$ is not the root) and
let $z=\pa(w)$ (if $w$ exists and parent of $w$ exists, i.e., if $w$ is not the root).
	
\noindent {\bf Case 1.}\ $\dir{x_Tv}$ is a 2-edge.	

We fix an arbitrary edge in the foot-path of $x_Tv$ as the
{\bf color-giving edge} of $T$.
If $x_Tv$ is uncolored, then we color it with the same color as that of the color-giving edge of $T$. 
 Note that the color-giving edge is already colored as we have already colored all the edges inside $\calF$.
 If $w$ exists and $\dir{vw}$ is uncolored, color $\dir{vw}$ with $s_1(T)$, the first surplus color of $T$.
 If $z$ exists and $\dir{wz}$ is uncolored, color $\dir{wz}$ with $s_2(T)$, the second surplus color of $T$.

 \noindent {\bf Case 2.}\ $\dir{x_Tv}$ is a $1$-edge.	

If $\dir{x_Tv}$ is uncolored, then color $\dir{x_Tv}$ with $s_1(T)$.
If $w$ exists and $\dir{vw}$ is uncolored, then color $\dir{vw}$ with $s_2(T)$.

We will prove in Lemma~\ref{lem:b1} that the above procedure in fact colors all the edges of $B_1$, and moreover does so with distinct colors.
We extend this coloring of the edges of $B_1$ to their representatives in $G$ as mentioned before: for each edge $e$ of $B_1$ color their representatives (both representatives in case of 2-edges and the only representative in case of $1$-edge) with the color of $e$.

At this point, there might be some edges in $G$ that are still uncolored. 
We call them \emph{irrelevant edges}.
Indeed, when we exhibit rainbow paths between pairs of vertices later, these edges will not be used.
The edges of $G$ that are not irrelevant are called \emph{relevant edges}.
To complete the edge-coloring of $G$, color all the irrelevant edges with an arbitrary color from $[f+t+2]$.

\begin{lemma}
	\label{lem:b1}
 All the edges of $B_1$ are colored and they are colored with distinct colors.	
\end{lemma}
\begin{proof}
It is easy to see that the colors are distinct as each color is used only once while coloring $B_1$.	
Therefore, it only remains to prove that all edges are colored.
Assume for the sake of contradiction that there is an uncolored edge $\dir{uv}$ of $B_1$.
Observe that in the coloring procedure, for each $x_T\in V_{\calT}$, we have colored its outgoing edge and the outgoing edge of its parent (if such an edge exists) irrespective of whether Case 1 or 2 was applied.
Hence, neither $u$ or any child of $u$ is a tree vertex.
%We prove that either $u$, or a child of $u$, or a child of child of $u$ is a tree vertex.
%If $u$ is a tree vertex i.e, if $u\in V_{\calT}$, then it is easy to see that $\dir{uv}$ is colored by the procedure, irrespective of whether Case 1 or 2 occurs
%Hence assume that $u$ is not a tree vertex.
Since $u\notin V_{\calT}$, it holds that $u$ is not a leaf of $B_1$ by Corollary~\ref{cor:leafB1}, and consequently $u$ has at least one child in $B_1$. 
Let $x$ be one such child.
As discussed above, $x$ is not a tree vertex.
%If $x\in V_{\calT}$, then it is also easy to see that $\dir{uv}$ is colored by the procedure. 
%Hence assume that $x\notin V_{\calT}$. 
Now, the edge $\dir{xu}$ is a $1$-edge as both endpoints are in $S$.
Then, by Corollary~\ref{cor:1edge-child}, there is a tree vertex $x_T$ that is a child of $x$ with $\dir{x_Tx}$ being a 2-edge. 
Then, Case 1 of the coloring procedure was applied on $T$, during which edge $\dir{uv}$ would have been colored.
\end{proof}
We are now ready to prove that the coloring of the edges of $G$ produced above is indeed a rainbow coloring.
%Consider a pair of vertices $v_1,v_2\in V(G)$.
We will prove that there is a rainbow path between every pair of vertices in $G$.
For this, we first prove in Lemma~\ref{lem:b1pathtake2} that there is a rainbow path between any pair of vertices in $V(G)\setminus L_S$ and then in Lemma~\ref{lem:gpathtake2}, we show that there is a rainbow path between any pair of vertices in $V(G)$.
The following observation is helpful in proving these lemmas.

\begin{observation}
	\label{obs:edgeexclusion}
	Let $v_1$, $v_2$, and $v_3$ be three vertices in any tree $T$, and let $e$ be an edge in
	$T_{v_2v_3}$. Then either $T_{v_1v_2}$ or $T_{v_1v_3}$ does not contain the edge $e$.
\end{observation}

For a vertex $v$ in $V(G)$, if $v\in S$ define $h(v):=v$, otherwise (i.e., if $v\in F$) define $h(v):=x_T$, where $T\in \calT$ is the tree containing $v$.

\begin{lemma}
	\label{lem:b1pathtake2}
	For any pair of vertices $v_1,v_2\in V(G)\setminus L_S$, there is a rainbow path between $v_1$ and $v_2$ in $G$ that uses only colors in $[f+t]$. 
\end{lemma}
%Note that we have left some edges of $G$ still uncolored. 
%We will show rainbow path using only the colored edges.
%Hence the uncolored edges can be given any arbitrary color in $[f+t+2]$.
\begin{proof}
Consider a pair of vertices $v_1,v_2\in V(G)\setminus L_S$.
We will construct a rainbow path $P$ from $v_1$ to $v_2$ using only the edges of $G$ that have colors in $[f+t]$.
%First we define two vertices of $B_1$ as the \emph{entry points} of $v_1$ and $v_2$ into $B_1$, which we call $v_1'$ and $v_2'$ respectively.
%If $v_1\in V(B_1)$, then define $v_1':=v_1$.
%Otherwise if $v_1\in L_S$, then there is a $T\in \calT$ to which $v_1$ has an edge in $G$ colored with $g_1$. 
%This follows from our coloring procedure.
%Define $v_1'$ as $x_T$.
%The only case remaining is when $v_1$ is in $V(T)$ for some $T\in \calT$.
%In that case define $v_1':=x_T$.
%Similarly define the entry point $v_2'$ for $v_2$ in the exact same way except that in the case when $v_2\in L_S$, we select the edge incident on $v_2$ colored with $g_2$ (instead of $g_1$ in case of $v_1$).
From Lemma~\ref{lem:b1}, we know that there is a rainbow path between $v_1':=h(v_1)$ and $v_2':=h(v_2)$ in $\und{B_1}$ that uses only the colors in $[f+t]$. 
Let this path be $P'$. 
We will use the path $P'$ as a guide to construct our required rainbow path $P$ in $G$.
First, break $P'$ into two paths $P'_1$ and $P'_2$ as follows.
Let $v_3'$ be the least common ancestor of $v_1'$ and $v_2'$ in $B_1$, let $P_1'$ be the path from $v_1'$ to $v_3'$ and let $P_2'$ be the path from $v_2'$ to $v_3'$. 

We will first construct a path $P_1$ in $G$ starting from $v_1$ using the path $P_1'$ as a guide:
we start with $P_1$ being just the vertex $v_1$.
We maintain a current vertex in $G$, denoted by $v_G$, which is initialized to $v_1$.
Let $e$ be the outgoing edge from $h(v_G)$ in $B_1$. 
The edge $e$ is the next edge in $P_1'$ to be processed.
As long as $h(v_G)\neq v_3'$, repeat the following step.

\noindent {\bf Case 1.}\ $h(v_G)\in S$. 

Append $(e)_1$ to $P_1$.
Also, update 
%$h(v_G)$ to be the other endpoint of $e$ and
$v_G$ to be the other endpoint of $(e)_1$.

\noindent {\bf Case 2.}\ $h(v_G)\in V_{\calT}$.

Let $x_T=h(v_G)$. We then branch into two sub-cases.

\noindent {\bf Case 2.1}\ $e$ is a $1$-edge.

Let $z$ be the endpoint of $(e)_1$ in $T$.
Append the path $T_{v_Gz}$ to $P_1$.
Then append $(e)_1$ to $P_1$.
Update $v_G$ to be the endpoint of $(e)_1$ outside $T$.

\noindent {\bf Case 2.2}\ $e$ is a 2-edge.

Let $z_1$ and $z_2$ be the foots of $e$ in $T$.
Observe that $v_G\in V(T)$ as $h(v_G)= x_T$.
By Observation~\ref{obs:edgeexclusion}, there exists a $z\in \{z_1,z_2\}$ such that there is a path from $v_G$ to $z$ that excludes the color-giving edge of $T$.
Append this path to $P_1$.
Now, append to $P_1$ the representative of $e$ having one of the endpoints as $z$. 
Update $v_G$ to be the endpoint outside $T$ of this representative.

Similarly, we also construct $P_2$ starting from $v_2$ using $P_2'$ as a guide.
If $v_3'\in S$, then we take $P:=P_1P_2$.
Otherwise, if $v_3'\in V_{\calT}$, we define $P$ as follows. 
Let $T=f_{\calT}(v_3')$. 
There are vertices $w_1, w_2 \in V(T)$ that are the endpoints of $P_1$ and $P_2$, respectively. 
Let $P_3$ be the path $T_{w_1w_2}$. 
Take $P:=P_1P_3P_2$.

By construction, it is clear that $P$ is indeed a path from $v_1$ to $v_2$.
It only remains to prove that $P$ is a rainbow path whose colors are a subset of $[f+t]$.
%For this, first observe that all the edges that we have included in $P$ are relevant edges with colors in $[f+t]$. 
Observe that each edge that we have added to $P$ is either a representative of an edge in $P'$ (call such edges $E_{PP'}$) or an edge in $\calF$ (call such edges $E_{P\calF}$). Recall that the representatives of an edge $e$ in $B_1$ were colored with the same color as that of $e$, and that each edge of $P'$ is colored with a color from $[f+t]$.
Thus, each edge of $E_{PP'}$ is a relevant edge that having a color in $[f+t]$.
Also recall that each edge in $\calF$ was colored with a color from $[f-t]$.
Thus, each edge of $E_{P\calF}$ is a relevant edge having a color in $[f-t]$.
This means that all edges in $P$ are relevant edges with colors in $[f+t]$.
In our coloring of the edges of $G$, each of the colors $1,2,\ldots, f-t$  except the colors of the color-giving edges have been used only for one relevant edge. 
Also, the color of the color-giving edge of a tree $T$ can possibly be repeated only on the representatives of the outgoing edge of $x_T$ in $B$. 
But when constructing $P$, we have taken care not to include the color-giving edge of $T$ in $P$ if $P$ contains a representative of the outgoing edge in of $x_T$. 
Also, we have included at most one of the representatives of the outgoing edge of $x_T$ in $P$.
Thus, the colors $1,2,\ldots, f-t$ appear at most once in $P$.
Each of the colors from $f-t+1$ to $f+t$ (surplus colors of the trees) appear on at most two relevant edges. 
Further, if they appear on two edges, then those two edges are the two representatives of some 2-edge of $B$.
Since we constructed $P$ in such a way that at most one representative of any edge in $B$ is included in $P$, each surplus color appears at most once in $P$. Thus $P$ is a rainbow path, concluding the proof.
\end{proof}

%This path can be extended to a rainbow path $P_1$ in $G$ as follows: for each edge $e$ in $P$, we use a representative of it 
%to simulate the edge in $G$. (Recall that the representatives have the same color as $e$). For $2$-edges, sometimes we have to cleverly choose which representative to use. Such cases arise only when the $2$-edge is outgoing from a tree vertex, say $T$, and the outgoing edge is colored with the color of an edge inside $T$. Let $Tv$ be the outgoing edge (in $B$) and $e$ be the edge inside $T$ having the same color.
%In other cases, we arbitrarily choose any one representative. In the above case, the choice of representative depends on the previous edge in path $P$. Suppose $uT$ was the previous edge in $P$ and let $e_1$ be the representative of $uT$ that we have decided to use in $P_1$. (If there is no previous edge, i.e. if $T$ is the starting). Suppose $v_1$ is the endpoint of $e_1$ in $T$.
%Let $v_2$ and $v_3$ be the foots of $Tv$ in $T$.
%First notice that for tree vertex $T$ and a subpath $uTv$ of $P$, although we can use the representatives to simulate edges $uT$ and $Tv$, say $e_1$ and $e_2$ are the representatives used, the endpoint of $e_1$ in $T$ might be different from the endpoint of $e_2$ in $T$, say they are $v_1$ and $v_2$ respectively. So, we also have to find a path from $v_1$ to $v_2$. But since $T$ is a tree, such a path does exist and moreover this path is colored with distinct colors which are not used elsewhere.
%end of  subsection take_2_rcle_1_5f 
\begin{lemma}
	\label{lem:gpathtake2}
	For any pair of vertices $v_1,v_2\in V(G)$, there is a rainbow path between $v_1$ and $v_2$ in $G$. 
\end{lemma}
\begin{proof}
Consider any pair of vertices $v_1,v_2\in V(G)$.
If both $v_1,v_2\in V(G)\setminus L_S$, then we are done by Lemma~\ref{lem:b1pathtake2}.
So assume without loss of generality that $v_1\in L_S$. 
Recall that if $v_1\in L_S$, then there is $2$-edge $e$ in $B$ incident on $v_1$ and that the representatives of $e$ are colored with $g_1$ and $g_2$.
Let $e_1$ be the representative of $e$ that is colored $g_1$, and
let $a$ be the other endpoint of $e_1$.
If $v_2\in L_S$, let $e_2$ be the edge incident on $v_2$ that is colored $g_2$, and
let $b$ be the other endpoint of $e_2$.
If $v_2\notin L_S$, let $b=v_2$. 
We know that there is a rainbow path $P_{ab}$ from $a$ to $b$ that uses only colors in $[f+t]$ due to Lemma~\ref{lem:b1pathtake2}.
We define the path
%If $v_2\in L_S$, then
$P:=v_1aP_{ab}bv_2$.
%otherwise, i.e., if
%$v_2\notin L_S$, we set
%$P:=v_1aP_{ab}v_2$. 
Since the edge $v_1a$ is colored with $g_1=f+t+1$, the edge $bv_2$ is colored with $g_2=f+t+2$, and path $P_{ab}$ uses only colors in $[f+t]$, we have that
the path $P$ is indeed a rainbow path between $v_1$ and $v_2$.
\end{proof}
The following result is now immediate from the constructed coloring and Lemma~\ref{lem:gpathtake2}.
\begin{theorem}
	Any connected graph $G$ has $\rc(G) \leq 2\f(G) + 2$.
\end{theorem}

\subsection{Take 3: \texorpdfstring{$\rc(G) \le \f(G)+2$}{}}
\label{sec:take3}
In this section, we prove our final bound, by further developing the ideas from the previous takes.
Recall that in Take~2, we gave two surplus colors to each tree. Here, we give only one surplus color to each tree and thereby reduce the number of colors used.
However, our analysis has to be tighter to make the proof work with only one surplus color per tree.
This makes the proof much more technical and lengthy.

Similar to Take 2, we fix a maximum induced forest $\calF$ of $G$, but with an additional property as follows.
Let $\calF$ be the maximum induced forest that has the smallest number of connected components (trees) out of all the maximum induced forests of $G$.
Now that $\calF$ is fixed, 
we define $\calT,F,S,H,V_{\calT},E_1,E_2,f_{\calT}, f,t$, tree vertices, non-tree vertices, $1$-edges, 2-edges, representatives, foots, foot-path and skeleton,
the same way as in Take~2. 
However, the selection of the skeleton $B$ is done in a more involved way here.
Given a skeleton $B$ with root $r$, we define the {\bf level} of each node $v$, denoted by $\ell_B(v)$, as its distance (in terms of number of vertices) to $r$ in $B$.
Note that $\ell_B(r) = 1$ per this definition.
For a skeleton $B$, we define its {\bf configuration vector} as the following vector:
\begin{align*}
	\langle\; |E_2( B)|, \Sigma_{v:\ell_B(v)=1} \dgr_B(v), \Sigma_{v:\ell_B(v)=2} \dgr_B(v),\ldots, \Sigma_{v:\ell_B(v)=|V|} \dgr_B(v) \;\rangle,
\end{align*}
where $\dgr_B(v)$ is the total degree (sum of in-degree and out-degree) of $v$ in $B$.
We now fix a skeleton $B$ such that it has the lexicographically highest configuration vector out of all possible skeletons.
We also define $L_S$ and the inner skeleton $B_1$ the same way as in Take 2, i.e., $L_S$ is the set of vertices in $S$ that are leaves of $B$, and $B_1=B[V(B)\setminus L_S]$.
Let $\und{B_1}$ be the underlying undirected tree of $B_1$.
We note that since the first element of the configuration vector is the number of 2-edges, $B$ would have been a valid skeleton in Take 2 as well.
In particular, Lemma~\ref{lem:vertex2edge} and Corollaries~\ref{cor:L2edge},~\ref{cor:leafB1} and~\ref{cor:1edge-child} from Take~2 hold for $B$.

We define a mapping $h$ from $G$ to $H$ as follows.
For a vertex $v$ in $V(G)$, if $v\in S$ then define $h(v):=v$, otherwise (i.e., if $v\in F$) define $h(v):=x_T$, where $T\in \calT$ is the tree containing $v$.
For a non-tree edge $e=uv$ in $G$, we define $h(e)$ to be the edge $h(u)h(v)$. 
For a vertex subset $U$ of $V(G)$, we define $h(U)$ to be $\bigcup_{a\in U}h(a)$. 
For an edge subset $E'$ of $E(G)$, we define $h(E')$ to be $\left\{ h(e): e\in E'\text{ and } e \text{ is a non-tree edge} \right\}$. 
For a subgraph $G'$ of $G$, we define $h(G')$ as the subgraph of $H$ with vertex set $h(V(G'))$ and edge set $h(E(G'))$.

Let the palette of colors be $\left\{ 1,2,\ldots, f+2 \right\}$.
We call colors $f+1$ and $f+2$ the {\bf global surplus colors},
and denote them by $g_1$ and $g_2$.
We reserve $g_1$ and $g_2$ to color the edges incident on $L_S$.
We will first give a coloring of some edges of $G$ using colors $\left\{ 1,2,\ldots,f \right\}$ such that there is a rainbow path between every pair of vertices in $V(G)\setminus L_S$. 
Then we will extend the coloring to $L_S$ using the global surplus colors.
We give our coloring procedure as a list of {\bf coloring rules}. 

For $a,b\in V(G)\setminus L_S$, let $Q_{ab}$ denote the unique path in the inner skeleton $B_1$ between $h(a)$ and $h(b)$. 
For each such pair of vertices $(a,b)$, we will maintain a subgraph $P_{ab}$ of $G$. 
Each $P_{ab}$ is initialized to $\emptyset$.
After the application of each coloring rule, we will apply a {\bf path rule} for each pair $(a,b)$, which (possibly) adds some newly colored edges to $P_{ab}$. 
We say that an edge in $B_1$ is colored if its representatives in $G$ are colored (we will make sure that for a 2-edge, either both representatives are colored or both are uncolored at any point of time). 
	Whenever an edge in $B_1$ gets colored by a coloring rule and if it is in $Q_{ab}$, we make sure that we add exactly one of its representatives to $P_{ab}$ in the subsequent path rule.
	Whenever it happens during a path rule that two edges $u_1v_1$ and $u_2v_2$ are in $P_{ab}$ such that both $v_1$ and $u_2$ are in some $T\in \calT$, but $u_1$ and $v_2$ are not in $T$, then we add the path $T_{v_2u_1}$ to $P_{ab}$ (if it is not already included).
	Similarly, if it happens that there is an edge $uv$ in $P_{ab}$ such that $v,a\in V(T)$  ($v,b \in V(T)$ resp.) but $u\notin V(T)$, we add the path $T_{va}$ ($T_{vb}$ resp.) to $P_{ab}$.
	Also, if both $a$ and $b$ are in the same tree $T$, then we add the path $T_{ab}$ to $P_{ab}$ (during Path Rule~\ref{prul:forest} below).
	Thus, when all the coloring rules and path rules have been applied, we will have that for all $a,b\in V(G)\setminus L_S$, it holds that $P_{ab}$ is a path between $a$ and $b$.
We will prove that $P_{ab}$ is also a rainbow path.
For this, we will maintain the following invariant. 
\begin{invariant}
	\label{inv:part-rainbow}	
	For each pair $a,b\in V(G)\setminus L_S$, no two edges in $P_{ab}$ have the same color. 
\end{invariant}
We will prove that the invariant still holds after each path rule.
Since new edges are added to $P_{ab}$ only during path rules, this means that the invariant always holds.
We also maintain the following three auxiliary invariants. But they are rather straightforward to check from the coloring and path rules and hence we will not explicitly prove them.
\begin{invariant}
	\label{inv:2edgerepcol}
For any 2-edge in $B$, either both representatives of it are colored or both are uncolored. 
\end{invariant}
A vertex in $B_1$ is said to be {\bf completed} if all the incident edges on it in $B_1$ are colored and is said to be {\bf incomplete} otherwise.
\begin{invariant}
	\label{inv:internalcolor}
	For an incomplete tree-vertex $x_T$, the colors of $E(T)$ are disjoint from the colors of the rest of the graph $G$. 
\end{invariant}
\begin{invariant}
	\label{inv:internaledgespath}
	A nonempty subset of internal edges of a tree $T$ is contained in $P_{ab}$ only if a representative of each edge in $Q_{ab}$ that is incident on $x_T$ (there can be at most two of such edges as $Q_{ab}$ is a path) is in $P_{ab}$.
\end{invariant}

Now, we start with the coloring and path rules.
%The paths are initialized as follows.
%\begin{prul}
%	\label{prule:init}
%	For each pair of vertices $a,b$ in $G^*$, initialize $P_{ab}$ to be any arbitrary simple path
%	in $G^*$ between $a$ and $b$. (Recall that $G^*$ is a connected graph).
%\end{prul}
%Since any uncolored path is a partially rainbow colored path, the invariant is satisfied. 
%We might modify $P_{ab}$ later but will always maintain %that it is a partially rainbow colored path.
%Invariant~\ref{inv:part-rainbow}.
%Note that $h(P_{ab})=Q_{ab}$ as long as $E_S=\emptyset$.
%Note that for an edge $e\in Q_{ab}$, one of its representatives i.e., either $(e)_1$ or $(e)_2$ is in $P_{ab}$; we may say that $(e)_1$ (or $(e)_2$) represents $e$ in $P_{ab}$. 
%(In fact we modify $P_{ab}$ after a coloring rule so that the invariant is still satisfied).
%Note that for each $u,v\in V(B_1)$, there is a unique path between $u$ and $v$ in $B_1$, we call this path $P'_{uv}$.
\begin{rul}
	\label{rul:forest}
	Color all the edges in $\calF$ with distinct colors $1,2,\ldots, f-t$.
\end{rul}
\begin{prul}
	\label{prul:forest}
	For each $a,b\in V(G)\setminus L_S$,
	if $a$ and $b$ are in the same tree $T$ for some $T\in \calT$, then add the path $T_{ab}$ to $P_{ab}$. 
\end{prul}
It is easy to see that Invariant 1 is satisfied after the above Path rule as the color of each edge is distinct so far.

%After Rule~\ref{rul:forest}, the number of unused colors in our palette of colors is $t+2$.
For each tree $T\in \calT$, we designate a color in $[f-t+1, f]$ as its {\bf surplus color}, denoted by $s(T)$. 
More specifically, the surplus color of $i^{\text{th}}$ tree in $\calT$ is defined as the color $f-t+i$.
Also, the colors of the edges of $T$ (colored by Coloring Rule~\ref{rul:forest}) are called the {\bf internal colors} of $T$.
%The remaining $2$ unused colors, $f+1$ and $f+2$ are called the {\bf global surplus colors}, denoted by $g_1$ and $g_2$. 
%We use $g_1$ and $g_2$ the same way as in Take 2.
\begin{rul}
	\label{rul:1edge}
	For each $1$-edge $\dir{uv}$ in $B$: if $u$ is a tree vertex, then color $(uv)_1$ with $s(f_{\calT}(u))$; otherwise, i.e., if $u$ is not a tree vertex, by Corollary~\ref{cor:1edge-child}, 
	there is at least one child of $u$ in $B_1$ that is a tree vertex; pick one such tree vertex $x_T$ and color $(uv)_1$ with color $s(T)$.
\end{rul}

Note that after Coloring Rule~\ref{rul:1edge}, any tree vertex $x_T$ such that $T$ is just a single vertex, is completed.
\begin{prul}
	\label{prul:1edge}
	Do the following for each $a,b\in V(G)\setminus L_S$.
	For each $1$-edge $e$ in $Q_{ab}$, add $(e)_1$ to $P_{ab}$.
	Next, we add edges inside trees as follows.
	\begin{itemize}
	\item If for some tree $T$ it holds that $a \in V(T)$ and there is a 1-edge $ux_T$ in $Q_{ab}$, then add the path $T_{wa}$ to $P_{ab}$, where $w$ is the foot of the edge $ux_T$ in $T$.
	
	\item If for some tree $T$ it holds that $b \in V(T)$ and there is a 1-edge $ux_T$ in $Q_{ab}$, then add the path $T_{wb}$ to $P_{ab}$, where $w$ is the foot of the edge $ux_T$ in $T$.
	
	\item If for some tree $T$ there are two 1-edges $ux_T$ and $vx_T$ in $Q_{ab}$, add the path $T_{wz}$ to $P_{ab}$, where $w$ is the foot of the edge $ux_T$ in $T$ and $z$ the foot of the edge $vx_T$ in $T$.	
	\end{itemize}
\end{prul}
\begin{lemma}
	\label{lem:inv1-1edge}
	Invariant 1 is satisfied so far. Moreover,
	each color is used at most for one edge in $G$.
\end{lemma}
\begin{proof}
	It is clear that during Coloring Rule~\ref{rul:forest}, all edges are colored distinct.
	In Coloring Rule~\ref{rul:1edge}, we use the surplus colors, which are disjoint from the colors used in Coloring rule \ref{rul:forest}.
	It is also not difficult to see that during Coloring Rule~\ref{rul:1edge}, the surplus color of a tree vertex is used for only one $1$-edge.
\end{proof}

For a colored edge $e\in E(G)$, we define $c(e)$ to be the color of $e$.
For a subgraph $G'$ of $G$, we define $c(G')$ to be the set of colors used in $E(G')$. 
We call the number of 2-edges of $B_1$ incident on a vertex,
the \emph{2-edge degree} of it.
For any two vertices $u$ and $v$, the connected component of $B_1\setminus u$ containing $v$ is denoted by $\st{u}{v}$.
Note that $\st{u}{v}$ is a subtree of $B_1$.
The closest (breaking ties arbitrarily) tree vertex to $v$ in $\st{u}{v}$ in $\und{B_1}$ is denoted by $\ct{u}{v}$. 
Note that at least one tree vertex exists in $\st{u}{v}$ because all leaves of $B_1$ are tree vertices by Corollary~\ref{cor:leafB1}. 
Also note that if $v$ is a tree vertex, then $\ct{u}{v}=v$.
%We also define $\stg{u}{v}:=g(\st{u}{v})$.
\begin{rul}
	\label{rul:2edgedeg4}
	For each tree vertex $x_T$ with $2$-edge degree at least $4$ (see Figure~\ref{fig:2edgedeg4} for an illustration).
	Let $q$ be the $2$-edge degree of $x_T$.
	Let $w_0$, $w_1$, $w_2, \ldots, w_{q-1}$ be the other endpoints of the 2-edges incident on $x_T$.
For $i\in[0,q-1]$, let $x_{T_{i}}:=\ct{x_T}{w_i}$ and let $c_i:=s(T_i)$. 
For each $i\in[0,q-1]$, color the edge $(x_Tw_i)_1$ with $c_{((i+2)\bmod q)}$ and the edge $(x_Tw_i)_2$ with $c_{((i+3)\bmod q)}$.
\end{rul}

\begin{figure}[b]
	\centering
	\includegraphics[scale=1.25,keepaspectratio]{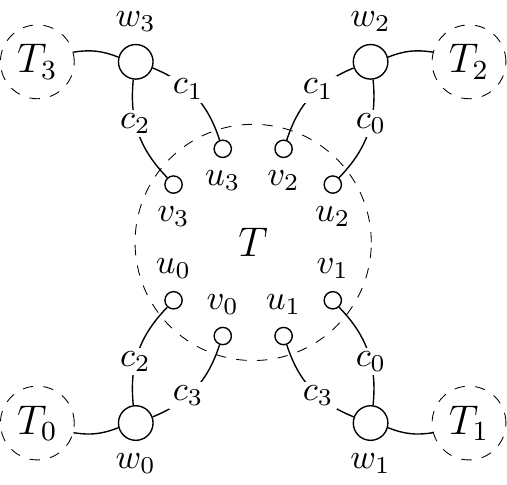}
	\caption{Illustration of Coloring Rule~\ref{rul:2edgedeg4} applied on a tree vertex $x_T$ with 2-edge degree $4$. Here, $c_i = s(T_i)$.}
	\label{fig:2edgedeg4}
\end{figure}

%This rule is illustrated in Figure~\ref{fig:2edgedeg4}.
The following lemma follows from the way in which we have colored the edges incident on $x_T$ in Coloring Rule~\ref{rul:2edgedeg4}.
\begin{lemma}
	\label{lem:2edgedeg4paths}
	For each tree vertex $x_T$ on which Coloring Rule~\ref{rul:2edgedeg4} has been applied as above, for all distinct $i,j\in [0,q-1]$, there is a rainbow path from $w_i$ to $w_j$ in $G$ 
	that uses only the colors from 
	$\left(\left\{ c_0,c_1,\ldots\allowbreak, c_{q-1}\allowbreak \right\}\allowbreak\setminus\left\{ c_i,c_j \right\}\allowbreak\right)\allowbreak\cup c(T)$.
	Moreover, for any $i\in [q-1]$ and some $u\in V(T)$, there is a rainbow path in $G$ from $u$ to $w_i$ that uses only colors from $\left(\left\{ c_0,c_1,\ldots, c_{q-1} \right\}\setminus\left\{ c_i \right\}\right)\cup c(T)$.
\end{lemma}
\begin{proof}
	Let $u_i$ and $v_i$ be the endpoints in $T$ of $(x_Tw_i)_1$ and $(x_Tw_i)_2$ respectively, for each $i\in \left\{ 0,1,\ldots,q-1 \right\}$.	
	First, we prove that there is a rainbow path from $w_i$ to $w_j$ with the required colors as claimed by the lemma.
	Suppose for the sake of contradiction that there was no such path.
	Consider the following three paths between $w_i$ and $w_j$:
	$P:=w_iu_iT_{u_iu_j}u_jw_{j}$,
	$P':=w_iv_iT_{v_iv_j}v_jw_{j}$, and
	$P'':=w_iv_iT_{v_iu_j}u_jw_{j}$. 
	By our assumption, each of these paths, is either not a rainbow path, or uses a color that is not in
	$\left(\left\{ c_0,c_1,\ldots\allowbreak, c_{q-1}\allowbreak \right\}\allowbreak\setminus\left\{ c_i,c_j \right\}\allowbreak\right)\allowbreak\cup c(T)$.
	Also, from Coloring rules~\ref{rul:forest} and \ref{rul:2edgedeg4}, we know that the only colors that are not in 
	$\left(\left\{ c_0,c_1,\ldots\allowbreak, c_{q-1}\allowbreak \right\}\allowbreak\setminus\left\{ c_i,c_j \right\}\allowbreak\right)\allowbreak\cup c(T)$
	that any of these three paths can use are $c_i$ and $c_j$.
	Thus, each of $P,P'$ and $P''$ is either not a rainbow path or uses $c_i$ or $c_j$.
	However, we know that the paths $T_{u_iu_j},T_{v_iv_j}$, and $T_{v_iu_j}$ are all rainbow paths due to Coloring Rule~\ref{rul:forest}, and moreover the colors used by them are disjoint from $\left\{ c_0,\ldots, c_{q-1} \right\}$. 
	For the path $P$, this means that either $c(w_iu_i)=c(u_jw_j)$ or $\left\{ c(w_iu_i),c(u_jw_j) \right\}\cap \left\{ c_i,c_j \right\}\neq \emptyset$.
	That is, either $c_{(i+2)\bmod q}=c_{(j+2)\bmod q}$ or $\left\{ c_{(i+2)\bmod q},c_{(j+2)\bmod q}\right\}\cap \left\{ c_i,c_j \right\}\neq \emptyset$.
	That is, either $i=j$ or $\left\{ {(i+2)\bmod q},{(j+2)\bmod q}\right\}\cap \left\{ i,j \right\}\neq \emptyset$.
	But we know that $(i+2)\bmod q\neq i$ and that $(j+2)\bmod q\neq j$.
	Therefore, either 
	 $(i+2)\bmod q=j$ or $(j+2)\bmod q=i$.
	%Also, the edge $w_iu_i$ is colored $c_{(i+2)\bmod q}$, and $u_jw_j$ is colored $c_{(j+2)\bmod q}$.
	%And, $c_{(i+2)\bmod q}\neq c_{(j+2)\bmod q}$ 
	%as $(i+2)\bmod q \neq (j+2)\bmod q$ for distinct $i,j\leq q$ and $q\ge 4$. 
	%This means that $P$ is a rainbow path.
	%Then, it should be the case that $P$ uses a color that is not in 
	%$\left(\left\{ c_0,c_1,\ldots\allowbreak, c_{q-1}\allowbreak \right\}\allowbreak\setminus\left\{ c_i,c_j \right\}\allowbreak\right)\allowbreak\cup c(T)$.
	%This implies that either $c_{(i+2)\bmod q}\in \{c_i,c_j\}$
	%or 
	%$c_{(j+2)\bmod q}\in \{c_i,c_j\}$.
	%But we know that
	%$c_{(i+2)\bmod q}\neq c_i$ and $c_{(j+2)\bmod q}\neq c_j$ as
	%$(k+2)\bmod q \neq k$ for any $k\leq q$ and $q\ge 4$.
	%Thus, the only possibility is that 
	%either
	%$c_{(i+2)\bmod q}=c_j$ or $ c_{(j+2)\bmod q}=c_i$.
%	That is,
%	either
%	${(i+2)\bmod q}=j$ or $ {(j+2)\bmod q}=i$.
	Without loss of generality assume that 
	 $(i+2)\bmod q=j$.

	 %Now consider another path
%	We also know that the path 
	 %$P':=w_iv_iT_{v_iv_j}v_jw_j$ between $w_i$ and $w_j$.
	%is either not a rainbow path or uses a color that is not in
	%$\left(\left\{ c_0,c_1,\ldots\allowbreak, c_{q-1}\allowbreak \right\}\allowbreak\setminus\left\{ c_i,c_j \right\}\allowbreak\right)\allowbreak\cup c(T)$.
	%We know that the path $T_{v_iv_j}$ uses only colors from $c(T)$ and is rainbow, 
%	We know that the edge $w_iv_i$ is colored $c_{(i+3)\bmod q}$, and $v_jw_j$ is colored $c_{(j+3)\bmod q}$.
	By using the same reasoning as above for path $P'$, we derive that either
	${(i+3)\bmod q}=j$ or $ {(j+3)\bmod q}=i$.
	Since we already have that $(i+2)\bmod q=j$, it should be the latter case, i.e, $ {(j+3)\bmod q}=i$.

	Now consider the third path $P''$.
	We have that either $c(w_iv_i)=c(u_jw_j)$ or $\left\{ c(w_iv_i),c(u_jw_j) \right\}\cap \left\{ c_i,c_j \right\}\neq \emptyset$.
	That is, either $c_{(i+3)\bmod q}=c_{(j+2)\bmod q}$ or $\left\{ c_{(i+3)\bmod q},c_{(j+2)\bmod q}\right\}\cap \left\{ c_i,c_j \right\}\neq \emptyset$.
	That is, either $(i+3)\bmod q =(j+2)\bmod q$ or $\left\{ {(i+3)\bmod q},{(j+2)\bmod q}\right\}\cap \left\{ i,j \right\}\neq \emptyset$.
	Substituting that $(i+3)\mod q=(((i+2)\mod q) +1)\mod q=(j+1)\mod q$ and that $i=(j+3)\mod q$, we get that either 
	$(j+1)\bmod q =(j+2)\bmod q$ or $\left\{ {(j+1)\bmod q},{(j+2)\bmod q}\right\}\cap \left\{ (j+3)\bmod q,j \right\}\neq \emptyset$.
	Since $j,(j+1)\bmod q,(j+2)\bmod q,$ and $(j+3)\bmod q$ are distinct for $q\ge 4$, we have a contradiction.

	Next, we prove the second part of the lemma, i.e., we prove that there is a rainbow path from $u$ to $w_i$ with the colors claimed by the lemma.
	Suppose for the sake of contradiction that there was no such path.
	Consider the path $P''':=w_iu_iT_{u_iu}$.
	%is either not a rainbow path or uses a color that is not in
	%$\left(\left\{ c_0,c_1,\ldots\allowbreak, c_{q-1}\allowbreak \right\}\allowbreak\setminus\left\{ c_i \right\}\allowbreak\right)\allowbreak\cup c(T)$.
	We know that the path $T_{u_iu}$ uses only colors from $c(T)$ and is rainbow, and that the edge $w_iu_i$ is colored $c_{(i+2)\bmod q}$.
	Also, 
	$c_{(i+2)\bmod q}\neq c_i$ as $(i+2)\bmod q \neq i$.
	Thus $P$ is a rainbow path and uses only the colors in 
	$\left(\left\{ c_0,c_1,\ldots\allowbreak, c_{q-1}\allowbreak \right\}\allowbreak\setminus\left\{ c_i \right\}\allowbreak\right)\allowbreak\cup c(T)$.
\end{proof}
\begin{prul}
	\label{prul:2edgedeg4}
	For each $x_T$ on which Coloring Rule~\ref{rul:2edgedeg4} has been applied as above and for each  $a,b\in V(G)\setminus L_S$ such that $Q_{ab}$ contains $x_T$ (we say that the path rule is being applied on the pair $\left( x_T,P_{ab} \right))$, do the following.
	
	\noindent {\bf Case 1:}\ There are two 2-edges incident on $x_T$ in $Q_{ab}$.
	
	Let $w_i$ and $w_j$ be the neighbors of $x_T$ in $Q_{ab}$.
	Add to $P_{ab}$ the rainbow path from $w_i$ to $w_j$ as given by Lemma~\ref{lem:2edgedeg4paths}.
	
	\noindent {\bf Case 2:}\ There is one $2$-edge and one $1$-edge incident on $x_T$ in $Q_{ab}$.
	
	Let $x_Tw_i$ be the 2-edge. 
	Let $u$ be the endpoint in $T$ of the representative of the $1$-edge.
	There is a rainbow path from $w_i$ to $u$ as given by Lemma~\ref{lem:2edgedeg4paths}.
	Add this path to $Q_{ab}$.
	(Note that the representative of the $1$-edge has been already added to $P_{ab}$ during Path Rule~\ref{prul:1edge}).
	
	\noindent {\bf Case 3:}\ $x_T$ is an endpoint of $Q_{ab}$ and the only edge incident on $x_T$ in $Q_{ab}$ is a $2$-edge.
	
	Let $w_i$ be the neighbor of $x_T$ in $Q_{ab}$.
	We know one of $a$ or $b$ is in $T$.
	From this vertex ($a$ or $b$ whichever is in $T$) to $w_i$, there is a 
	rainbow path as given by Lemma~\ref{lem:2edgedeg4paths}.
	Add this path to $P_{ab}$.
	%Let $x_T$ be a vertex on which Rule~\ref{rul:2edgedeg4} has been applied as above.
	%Modify each $P_{ab}$ that intersects $G':=G^*[V(T)\cup \left\{ w_0,w_1,\cdots w_{q-1} \right\}]$ on at least $1$-edge, as follows:
	%Let $P_1$ be the intersection of $P_{ab}$ with $G'$.
	%Note that $P_1$ is a subpath of $P_{ab}$.
	%If $P_1$ is completely contained in $T$ then leave $P_{ab}$ unmodified.
	%Let $z_1$ and $z_2$ be the endpoints of $P_1$.
	%%If $z_1,z_2\in V(T)$, then leave $P_{ab}$ unmodified. 
%We can assume without loss of generality $z_1=w_i$ for some $i\in [0,q-1]$ and either $z_2=w_j$ for some $j\in [0,q-1]\setminus \left\{ i \right\}$
%(we say then Case A occurs)
%or $z_2\in V(T)$
%(we say then Case B occurs).
	%Thus, there is a rainbow path from $z_1$ to $z_2$ as given by Lemma~\ref{lem:2edgedeg4paths}.
	%Let this path be $P_2$.
	%Replace $P_1$ by $P_2$ in $P_{ab}$.
\end{prul}
The following lemma follows from Lemma~\ref{lem:2edgedeg4paths} and Path Rule~\ref{prul:2edgedeg4}.
\begin{lemma}
	\label{lem:2edgedeg4paths2}
	Suppose for some $a,b\in V(G)\setminus L_S$ and for some tree $T'\in\calT$, $P_{ab}$ contains an edge $e$ that was colored with $s(T')$ during the application of Coloring Rule~\ref{rul:2edgedeg4} on some tree vertex $x_T$.
	Then, $T'\neq T$ and 
	%$\st{x_{T'}}{x_T}$ does not contain $h(a)$ or $h(b)$ 
$Q_{ab}$ does not intersect $\st{x_{T}}{x_{T'}}$.
\end{lemma}
\begin{proof}
	Since $s(T')$ was used during the application of Coloring Rule~\ref{rul:2edgedeg4} on $x_{T}$, the vertex $x_{T'}$ should have been taken as $x_{T_i}$ (in Coloring Rule~\ref{rul:2edgedeg4}) for some $i$
	and $s(T')$ was taken as $c_i$ (in Coloring Rule~\ref{rul:2edgedeg4}).
	Since $T_i\neq T$, it is clear that $T'\neq T$.
Suppose 	
$Q_{ab}$ intersects $\st{x_{T}}{x_{T'}}$ for the sake of contradiction.
That is, $Q_{ab}$ intersects $\st{x_{T}}{x_{T_i}}$. 
Then the color $c_i$ was not used in Path Rule~\ref{prul:2edgedeg4} according to
Lemma~\ref{lem:2edgedeg4paths}.
That means $e$ was not colored with $c_i$, which
is a contradiction.
\end{proof}
\begin{lemma}
	\label{lem:inv2edgedeg4}
Invariant \ref{inv:part-rainbow} is not violated during 
	Path Rule~\ref{prul:2edgedeg4}.
\end{lemma}
\begin{proof}
	Suppose Invariant \ref{inv:part-rainbow} is violated during the application of Path Rule~\ref{prul:2edgedeg4} on the pair $\left(x_T,P_{ab}\right)$.
Then there exist edges $e$ and $e'$ in $P_{ab}$ having the same color after the application of the path rule.
We can assume without loss of generality that $e$ was added during the application of Path Rule~\ref{rul:2edgedeg4} on $(x_T,P_{ab})$.
That means $e$ was colored during the application of Coloring Rule~\ref{rul:2edgedeg4} on $x_T$.
Then either $e\in E(T)$ or $h(e)=w_ix_T$ for some $i\in [0,q-1]$.
Since each color in $c(T)$ has been used only in one edge in $G$, we have
that $h(e)=w_ix_T$ for some $i\in [0,q-1]$ 
and hence $c(e)=s(T_j)$ for some $j\in [0,q-1]\setminus{i}$.
Also $Q_{ab}$ does not intersect $\st{x_T}{x_{T_j}}$
by Lemma~\ref{lem:2edgedeg4paths2}.
Since the application of Path Rule~\ref{prul:2edgedeg4} on $(x_T,P_{ab})$ added a rainbow path to $P_{ab}$,
the edge $e'$ was not added during this application.
Since each color in $c(F)$ has been used for only one edge in $G$ so far, we know that $e'$ was not added during Path Rule~\ref{prul:forest}.
Hence, the following two cases are exhaustive and in both cases we derive a contradiction.

\noindent {\bf Case 1:}\ $e'$ was added during the application of Path Rule~\ref{prul:2edgedeg4} on $(x_{T'}, P_{ab})$ for some tree $T'\neq T$.

Since $P_{ab}$ contains $e'$, we have that $Q_{ab}$ contains $h(e')$.
Since $e'$ was added during the application of Path Rule~\ref{prul:2edgedeg4} on $(x_{T'},P_{ab})$, 
either $e'\in E(T')$ or $h(e')$ is incident on $x_{T'}$. 
In either case, $x_{T'}$ is in $Q_{ab}$.
Since $Q_{ab}$ does not intersect $\st{x_T}{x_{T_j}}$, we have that $x_{T'}$ is not in $\st{x_T}{x_{T_j}}$.
This implies that $\dist_{\und{B_1}}(x_{T'},x_{T})< \dist_{\und{B_1}}(x_{T'},x_{T_j})$
But then during the application of Coloring Rule~\ref{rul:2edgedeg4} on $x_{T'}$, the color $s(T_j)$ would never be used as $x_{T_j}\neq \ct{x_{T'}}{v}$ for any vertex $v$. 
Thus, the color of $e'$ is not $s(T_j)$.
But
we know that $c(e')=c(e)=s(T_j)$, a contradiction.

\noindent {\bf Case 2:}\ $e'$ was added during the application of Path Rule~\ref{prul:1edge} on $P_{ab}$.
 
This means $e'$ is the representative of a $1$-edge and was colored during Coloring Rule~\ref{rul:1edge}. 
Since $e'$ is colored with $s(T_j)$, we have that $h(e')$ should either be the outgoing edge of $x_{T_j}$ or the outgoing edge of the parent of $x_{T_j}$, from Coloring Rule~\ref{rul:1edge}.
This implies that $h(e')$ is in $\st{x_T}{x_{T_j}}$,
as the parent of $x_{T_j}$ is a non-tree vertex. 
But then $Q_{ab}$ does not contain $h(e')$ as $Q_{ab}$ does not intersect $\st{x_T}{x_{T_j}}$.
Thus $P_{ab}$ does not contain $e'$, which is a contradiction.
\end{proof}

\begin{figure}[b]
	\centering
	\includegraphics[scale=1.25,keepaspectratio]{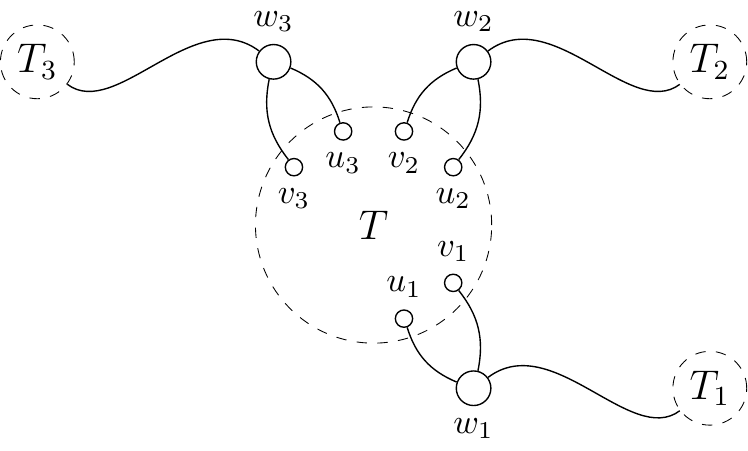}
	\caption{A scenario in which Coloring Rule~\ref{rul:2edgedeg3} is applicable on $x_T$.}
	\label{fig:2edgedeg3-1}
\end{figure}

\begin{figure}[t]
	\centering
	\includegraphics[scale=1.25,keepaspectratio]{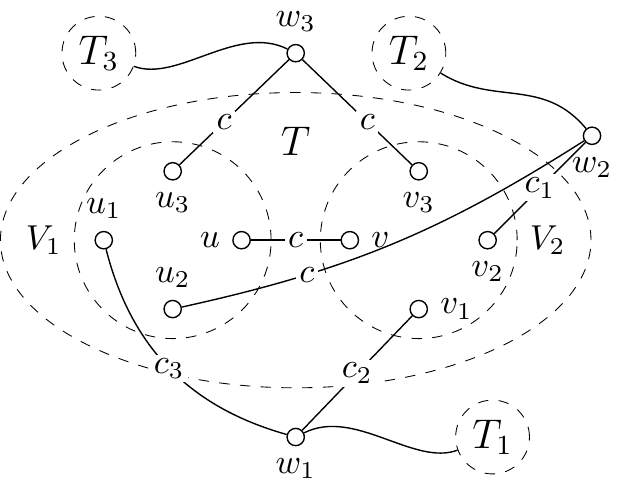}
	\caption{Case~1 of Coloring Rule~\ref{rul:2edgedeg3}.}
	\label{fig:2edgedeg3-2}
\end{figure}

\begin{figure}
    \centering
    \begin{subfigure}{0.33\textwidth}
    \centering
        \includegraphics[scale=1.0,keepaspectratio]{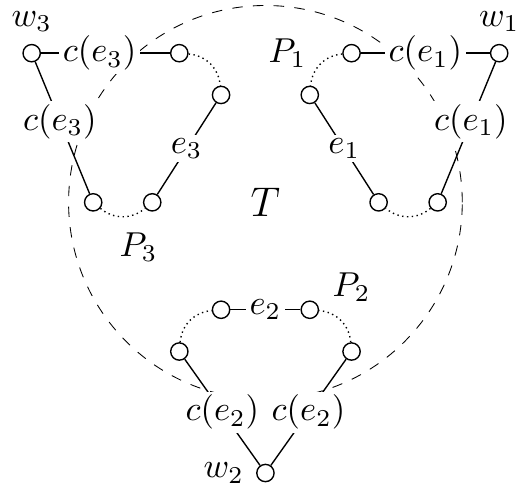}
        \caption{}
        %\label{fig:forest}
    \end{subfigure}%
    \hfill
    \begin{subfigure}{0.33\textwidth}
    \centering
		\includegraphics[scale=1.0,keepaspectratio]{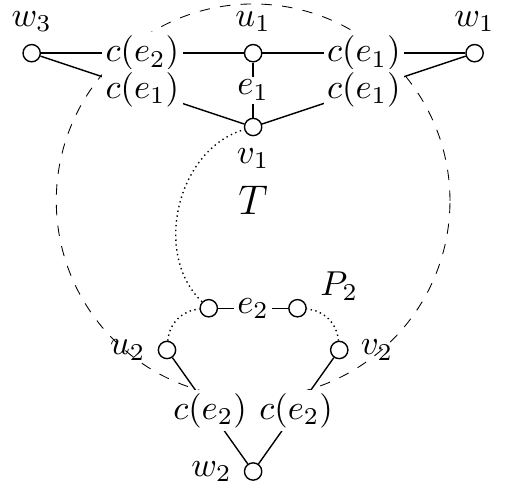}
		\caption{}
        %\label{fig:forest-contracted}
    \end{subfigure}%
    \hfill
    \begin{subfigure}{0.33\textwidth}
    \centering
		\includegraphics[scale=1.0,keepaspectratio]{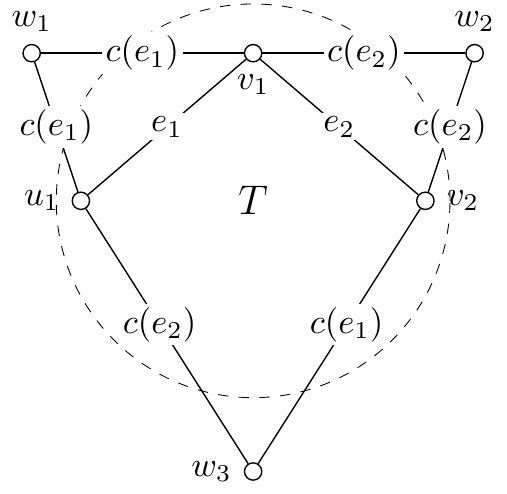}
		\caption{}
        %\label{fig:forest-contracted}
    \end{subfigure}
	\caption{Cases 2 and 3 of Coloring Rule~\ref{rul:2edgedeg3}. \textbf{(a)} Case~2. Note that $P_1$, $P_2$, and $P_3$ are not necessarily disjoint. \textbf{(b)} Case~3, scenario~1. Note that $u_1 = u_3$ and $v_1=v_3$. \textbf{(c)} Case~3, scenario~2.   Note that $u_1 = u_3$, $v_1 = u_2$ and $v_3=v_2$.}
	\label{fig:2edgedeg3-3-4-5}
\end{figure}

\begin{rul}
\label{rul:2edgedeg3}
For each tree vertex $x_T$ with 2-edge degree exactly $3$ (see Figure~\ref{fig:2edgedeg3-1}), 	
let $w_1$, $w_2$, and $w_3$ be the other endpoints of the three 2-edges incident on $x_T$.
Further, for $i\in\left\{ 1,2,3\right\}$, let $x_{T_{i}}=\ct{x_T}{w_i}$, 
let $u_i$ and $v_i$ be the foots of $x_Tw_i$ in $T$, 
let $P_i:=T_{u_iv_i}$,
and let $c_i:=s(T_i)$.

\noindent {\bf Case 1:}\ There exists an edge $uv$ in $T$ such that the cut $(V_1,V_2)$ induced by $uv$ in $T$ is such that for all $i\in\{1,2,3\}$, $|V_1\cap\left\{ u_i,v_i \right\}|=1$ and $|V_2\cap\left\{ u_i,v_i \right\}|=1$.
(For an illustration, see Figure~\ref{fig:2edgedeg3-2}).

Without loss of generality, let $u_i$ and $v_i$ be the foots of $x_Tw_i$ in $V_1$ and $V_2$ respectively for each $i\in \left\{ 1,2,3 \right\}$.
Let $c$ be the color of $uv$.
Color $u_1w_1$ with $c_3$, $v_1w_1$ with $c_2$,
$u_2w_2$ with $c$, $v_2w_2$ with $c_1$,
$u_3w_3$ with $c$, and $v_3w_3$ with $c$,
as shown in Figure~\ref{fig:2edgedeg3-2}.

\noindent {\bf Case 2:}\ There exist distinct edges $e_1,e_2,e_3$ such that $e_i\in E(P_i)$  for each $i\in \left\{ 1,2,3 \right\}$.
(For an illustration, see Figure~\ref{fig:2edgedeg3-3-4-5}~(a)).

Color both the representatives of $x_Tw_i$ with the color of $e_i$ for each $i\in \left\{ 1,2,3 \right\}$.

\noindent {\bf Case 3:}\ Case 1 and 2 do not apply. 

Because Case~1 and~2 do not apply, there exist $i,j\in \{1,2,3\}$ such that $E(P_i)\cap E(P_j)=\emptyset$, because otherwise $E(P_1)\cap E(P_2)\cap E(P_3)\neq \emptyset$ using the Helly property of trees\footnote{We use the following Helly property of trees: if $T_1,T_2,\ldots, T_k$ are subtrees of a tree $T$ that pairwise intersect each other on at least one edge, then there is an edge of $T$ that is common to all of $T_1,T_2,\ldots,T_k$.} 
and then any edge in this intersection qualifies as $uv$ of Case 1.
So, without loss of generality assume that $E(P_1)\cap E(P_2)=\emptyset$.
Also, note that $E(P_3)\subseteq E(P_1)\cup E(P_2)$ because otherwise Case 2 applies. 
So, without loss of generality assume that $E(P_3)\cap E(P_1)\neq \emptyset$.
But then $E(P_3)\cap E(P_1)=E(P_1)$ and $P_1$ consists of a single edge so that Case 2 does not apply. 
Let this edge be $e_1$.
Note that $e_1=u_1v_1$.
Furthermore, at least one of the end-vertices of $P_1$ and $P_3$ coincide so that Case 2 does not apply.
Thus, assume without loss of generality that $u_1=u_3$.
Let $e_2$ be any edge in $P_2$.
Without loss of generality assume that $v_1$ is the closer vertex among $u_1$ and $v_1$ to path $P_2$ in $T$.
The two possible scenarios in this case are shown in Figure~\ref{fig:2edgedeg3-3-4-5}~(b)~and~(c).
%Since $E(P_1)\cap E(P_2)$, we can assume without loss of generality that the path from $u_1$ to $v_2$ goes through $v_1$ and $u_1$.
Color $w_1u_1$ and $w_1v_1$ with $c(e_1)$, $w_2u_2$ and $w_2v_2$ with $c(e_2)$, $w_3u_3$ with $c(e_2)$ and $w_3v_3$ with $c(e_1)$.\\
%and $|E(P_3)\cap E(P_1)|\leq 1$ and $E(P_3)\cap E(P_2)\leq 1$ because otherwise Case 2 applies.
%{Case 3.1}. If $E(P_3)\cap E(P_1)\neq \emptyset$ and $E(P_3)\cap E(P_2)\neq \emptyset$: then $P_1$ and $P_2$ have to be single edges as otherwise Case 2 is applicable.
\end{rul}

The following lemma follows from the way in which we have colored the edges incident on $x_T$ in Coloring Rule~\ref{rul:2edgedeg3}.
\begin{lemma}
	\label{lem:2edgedeg3paths}
	For each tree vertex $x_T$ on which Rule~\ref{rul:2edgedeg3} has been applied as above, for distinct $i,j\in \left\{ 1,2,3 \right\}$, there is a rainbow path from $w_i$ to $w_j$ in $G$ that 
	uses only the colors from $(\left\{ c_1,c_2,c_3 \right\}\setminus\left\{ c_i,c_j \right\})\cup c(T)$.
	Also, for any $i\in\left\{ 1,2,3 \right\}$, and any $z\in V(T)$,
	there is a rainbow path from $z$ to $w_i$,
	that uses only the colors from $(\left\{ c_1,c_2,c_3 \right\}\setminus\left\{ c_i \right\})\cup c(T)$.
\end{lemma}
\begin{proof}
	We demonstrate the required paths in each of the three cases of Coloring Rule~\ref{rul:2edgedeg3}.	

	\noindent {\bf Case 1:}\ Between $w_1$ and $w_2$, there is the rainbow path $w_1u_1T_{u_1u_2}u_2w_2$  that uses only the colors in $c(T)\cup \left\{ c_3 \right\}$.
	Between $w_1$ and $w_3$, there is the rainbow path $w_1v_1T_{v_1v_3}v_3w_3$  that uses only the colors in $c(T)\cup \left\{ c_2 \right\}$.
	Between $w_2$ and $w_3$, there is the rainbow path $w_2v_2T_{v_2v_3}v_3w_3$  that uses only the colors in $c(T)\cup \left\{ c_1 \right\}$.

	Now consider any vertex $z\in V(T)$.
	Suppose $z\in V_1$.
	Between $z$ and $w_1$, there is the rainbow path $T_{zu_1}u_1w_1$  that uses only the colors in $c(T)\cup \left\{ c_3 \right\}$.
	Between $z$ and $w_2$, there is the rainbow path $T_{zv_2}v_2w_2$  that uses only the colors in $c(T)\cup \left\{ c_1 \right\}$.
	Between $z$ and $w_3$, there is the rainbow path $T_{zu_3}u_3w_3$  that uses only the colors in $c(T)$.

	Now, suppose $z\in V_2$.
	Between $z$ and $w_1$, there is the rainbow path $T_{zv_1}v_1w_1$  that uses only the colors in $c(T)\cup \left\{ c_2 \right\}$.
	Between $z$ and $w_2$, there is the rainbow path $T_{zv_2}v_2w_2$  that uses only the colors in $c(T)\cup \left\{ c_1 \right\}$.
	Between $z$ and $w_3$, there is the rainbow path $T_{zv_3}v_3w_3$  that uses only the colors in $c(T)$.

	\noindent {\bf Case 2:}\ First we show the path between $w_1$ and $w_2$. 
	By Observation~\ref{obs:edgeexclusion}, either $T_{u_1u_2}$ or $T_{u_1v_2}$ does not contain the edge $e_2$.
	If $T_{u_1u_2}$ does not contain $e_2$, then the path $w_1u_1T_{u_1u_2}u_2w_2$ is a rainbow path and uses only the colors
	in $c(T)$;
	otherwise (i.e., if $T_{u_1v_2}$ does not contain $e_2$) then the path $w_1u_1T_{u_1v_2}v_2w_2$ is a rainbow path and uses only the colors
	in $c(T)$.
	The required paths between $w_2$ and $w_3$, and between $w_1$ and $w_2$ can be shown in a similar way.

	Now for any vertex $z$ in $T$, we show the required path between $z$ and $w_1$.
	By Observation~\ref{obs:edgeexclusion}, either $T_{zu_1}$ or $T_{zv_1}$ does not contain the edge $e_1$.
	If $T_{zu_1}$ does not contain $e_1$, then the path $T_{zu_1}u_1w_1$ is a rainbow path and uses only the colors
	in $c(T)$;
	otherwise (i.e., if $T_{zv_1}$ does not contain $e_1$) then the path $T_{zv_1}v_1w_1$ is a rainbow path and uses only the colors
	in $c(T)$.
	The required paths between $w_2$ and $z$, and between $w_3$ and $z$ can be shown in a similar way.

	\noindent {\bf Case 3:}\ 
	First we show the required path between $w_1$ and $w_2$. 
	By Observation~\ref{obs:edgeexclusion}, either $T_{v_1u_2}$ or $T_{v_1v_2}$ does not contain the edge $e_2$.
	Observe that both $T_{v_1u_2}$ and $T_{v_1v_2}$ does not contain the edge $e_1$ as $v_1$ is closer than $u_1$ to $P_2$, as mentioned in the Coloring Rule. 
	Hence, if $T_{v_1u_2}$ does not contain $e_2$, then
	the path $w_1v_1T_{v_1u_2}u_2w_2$ is a rainbow path that uses only the colors in $\left\{ c(e_2),c(e_1) \right\}\cup c(T)$;
	and otherwise (i.e., if $T_{v_1v_2}$ does not contain $e_2$),
	the path $w_1v_1T_{v_1v_2}v_2w_2$ is a rainbow path that uses only the colors in $\left\{ c(e_2),c(e_1) \right\}\cup c(T)$.

	Since $u_1=u_3$, there is the path $w_1u_1w_3$ between $w_1$ and $w_3$ that uses only the colors in $\left\{ c(e_1),c(e_2) \right\}$.
	Next, we show the required path between $w_3$ and $w_2$. 
	By Observation~\ref{obs:edgeexclusion}, either $T_{v_3u_2}$ or $T_{v_3v_2}$ does not contain the edge $e_2$.
	Let $v'$ be the vertex in $\{u_2,v_2\}$ such that $T_{v_3v'}$ does not contain edge $e_2$.
	We show that the path $w_3v_3T_{v_3v'}v'w_2$ is the required path between $w_3$ and $w_2$.
	We know $w_3v_3$ is colored $c(e_1)$ and $w_2v'$ is colored $c(e_2)$.
	So, it is sufficient to show that $c(e_1)$ and $c(e_2)$ does not appear in $T_{v_3v'}$.
	For this, it is sufficient to prove that $e_1$ and $e_2$ is not in $T_{v_3v'}$.
	Since we picked $v'$ such that $T_{v_3v'}$ does not contain $e_2$, 
	it only remains to prove that $e_1$ is not in $T_{v_3v'}$. 
	Suppose for the sake of contradiction that $e_1$ is in $T_{v_3v'}$.
	That means both $v_1$ and $u_1$ are in $T_{v_3v'}$.
We know that $v_1$ is closer than $u_1$ to $P_2$, as mentioned in the Coloring Rule.
Hence, $v_1$ is closer than $u_1$ to $v'$ in $T$. 
This also implies that $v_1$ is closer than $u_1$ to $v'$ in $T_{v_3v'}$.
Then $u_1$ is closer than $v_1$ to $v_3$ in $T_{v_3v'}$ and hence also in $T$.
But then $P_3$ contains edges that are not in both $P_1$ and $P_2$, a contradiction.

Now consider any vertex $z\in V(T)$.
For each $i\in \left\{ 1,3 \right\}$, let $v_i'$ be the closer vertex among $u_i,v_i$ to $z$.
and let $P_i$ be the path $T_{zv'_i}v'_iw_i$.
We show that $P_i$ is the required path from $z$ to $w_i$ for $i\in \left\{ 1,3 \right\}$.
The path from $z$ to $v'_1$ does not contain $e_1$.
Also, the edge $w_1v'_1$ is colored with $c(e_1)$.
Hence $P_1$ is a rainbow path from $z$ to $w_1$ that uses only colors in $c(T)$.
Now consider path $P_3$.
First consider the case when $v'_3=u_3=u_1$.
Then $e_2$ is not in $T_{zv'_3}$, because otherwise either $P_3$ contains edges that are not in $P_1\cup P_2$ or $u_1$ is closer than $v_1$ to $P_2$.
Since $u_3w_3$ is colored $c(e_2)$, the path $P_3$ satisfies the requirements. 
Now consider the case when $v'_3=v_1$.
Then $e_1$ is clearly not in $T_{zv'_3}$. 
Since $v_3w_3$ is colored $c(e_1)$, the path $P_3$ satisfies the requirements. 

Now we show the required path from $z$ to $w_2$.
By Observation~\ref{obs:edgeexclusion}, either $T_{zv_2}$ or $T_{zu_2}$ does not contain the edge $e_2$.
	Let $v'_2$ be the vertex in $\{u_2,v_2\}$ such that $T_{zv'_2}$ does not contain edge $e_2$.
	Then the path $T_{zv'_2}v'_2w_2$ is the required path between $z$ and $w_2$, as $v_2'w_2$ is colored $c(e_2)$.
\end{proof}
\begin{prul}
	\label{prul:2edgedeg3}
	For each $x_T$ on which Coloring Rule~\ref{rul:2edgedeg3} has been applied as above and for each  $P_{ab}$ such that $Q_{ab}$ contains $x_T$ $($we say that the rule is being applied on the pair $\left( x_T,P_{ab} \right))$, do the following.
	
	\noindent {\bf Case 1:}\ $x_T$ has two 2-edges incident in $Q_{ab}$.
	
	Let $w_i$ and $w_j$ be the neighbors of $x_T$ in $Q_{ab}$.
	Add to $P_{ab}$ the rainbow path from $w_i$ to $w_j$ as given by Lemma~\ref{lem:2edgedeg3paths}.
	
	\noindent {\bf Case 2:}\ $x_T$ has exactly one 2-edge and exactly one 1-edge incident in $Q_{ab}$.
	
	Let $x_Tw_i$ be the 2-edge and let $z$ be the endpoint in $T$ of the $1$-edge.
	Add to $P_{ab}$ the rainbow path from $w_i$ to $z$ as given by Lemma~\ref{lem:2edgedeg3paths}.
	
	\noindent {\bf Case 3:}\ $x_T$ is an endpoint of $Q_{ab}$ and has one 2-edge incident in $Q_{ab}$.
	
	Let $w_i$ be the neighbor of $x_T$ in $Q_{ab}$.
	We know one of $a$ or $b$ is in $T$.
	From this vertex ($a$ or $b$, whichever is in $T$) to $w_i$, there is a 
	rainbow path as given by Lemma~\ref{lem:2edgedeg3paths}.
	Add this path to $P_{ab}$.
\end{prul}
The following lemma follows from Lemma~\ref{lem:2edgedeg3paths} and Path Rule~\ref{rul:2edgedeg3}.
The proof is similar to that of Lemma~\ref{lem:2edgedeg4paths2} and is omitted.
\begin{lemma}
	\label{lem:2edgedeg3paths2}
	Suppose for some $a,b\in V(G)\setminus L_S$ and for some tree $T'\in\calT$, $P_{ab}$ contains an edge $e$ that was colored with $s(T')$ during the application of Coloring Rule~\ref{rul:2edgedeg3} on some tree vertex $x_T$.
	Then, $T'\neq T$ and 
	%$\st{x_{T'}}{x_T}$ does not contain $h(a)$ or $h(b)$ 
$Q_{ab}$ does not intersect $\st{x_{T}}{x_{T'}}$.
\end{lemma}
\begin{lemma}
	\label{lem:inv2edgedeg3}
Invariant \ref{inv:part-rainbow} is not violated during 
	Path Rule~\ref{prul:2edgedeg3}.
\end{lemma}
\begin{proof}
	%The proof is similar to that of Lemma~\ref{lem:inv2edgedeg4}.
Suppose for the sake of contradiction that Invariant \ref{inv:part-rainbow} is violated during the application of Path Rule~\ref{prul:2edgedeg3} on the pair $(x_T,P_{ab})$ as above.
Then there exist edges $e$ and $e'$ in $P_{ab}$ having the same color.
We can assume without loss of generality that $e$ was colored during the application of Coloring Rule~\ref{rul:2edgedeg3} on $x_T$.
This means $e\in E':=E(T)\cup R$, where $R$ is defined as the set of representatives of  $w_1x_T,w_2x_T,$ and $w_3x_T$.
Since the application of Path Rule~\ref{prul:2edgedeg3} on $(x_T,P_{ab})$ added a rainbow path to $P_{ab}$,
the edge $e'$ was not added during this application and hence $e'\notin E'$.
Each color in $c(T)$ have been used only in $E'$ so far.
That means $c(e)=c(e')\notin c(T)$.
Hence $e\in E'\setminus E(T)=R$.
Without loss of generality assume that $e$ is a representative of $w_1x_T$.
Now, $c(e)=s(T_j)$ where $j\in \{2,3\}$.
Without loss of generality assume that $c(e)=s(T_2)$.
This also means $c(e')=s(T_2)$.
That means $e'$ was colored during Coloring Rules~\ref{rul:1edge}, \ref{rul:2edgedeg4} or \ref{rul:2edgedeg3}. 
%Also $h(a)$, and $h(b)$ are not in $\st{x_T}{T_2}$ due to Lemma~\ref{lem:inv2edgedeg4}. 
Hence the following two cases are exhaustive and in each case we prove a contradiction.

\noindent {\bf Case 1:}\ $e'$ was colored during the application of Coloring Rules~\ref{rul:2edgedeg4} or~\ref{rul:2edgedeg3} on $x_{T'}$, for some tree $T'\neq T$.

Since $P_{ab}$ contains $e'$, we have that $Q_{ab}$ contains $h(e')$.
Since $e'$ was colored during the application of Coloring Rules~\ref{rul:2edgedeg4} or~\ref{rul:2edgedeg3} on $x_{T'}$, 
either $e'\in E(T')$ or $h(e')$ is incident on $x_{T'}$, and hence $x_{T'}$ is in $Q_{ab}$.
Since $Q_{ab}$ does not intersect $\st{x_T}{x_{T_2}}$ by Lemmas~\ref{lem:2edgedeg4paths2} and \ref{lem:2edgedeg3paths2}, 
we have that $x_{T'}$ is not in $\st{x_T}{x_{T_2}}$.
Then $\dist_{\und{B_1}}(x_{T'},x_{T})< \dist_{\und{B_1}}(x_{T'},x_{T_2})$. 
But then during the application of Coloring Rule~\ref{rul:2edgedeg4} or~\ref{rul:2edgedeg3} on $x_{T'}$, the color $s(T_2)$ would never be used as $x_{T_2}\neq \ct{x_{T'}}{v}$ for any vertex $v$. 
Thus, the color of $e'$ is not $s(T_2)$.
But
we know that $c(e')=c(e)=s(T_2)$, a contradiction.

\noindent {\bf Case 2:}\ $e'$ was colored during the application of Coloring Rule~\ref{prul:1edge}.

This means $e'$ is the representative of a $1$-edge. 
Since $e'$ is colored with $s(T_2)$, we have that $h(e')$ should either be the outgoing edge of $x_{T_2}$ or the outgoing edge of the parent of $x_{T_2}$, from Coloring Rule~\ref{rul:1edge}.
This implies that $h(e')$ is in $\st{x_T}{x_{T_2}}$,
as the parent of $x_{T_2}$ is a non-tree edge. 
But then $Q_{ab}$ does not contain $h(e')$ as $Q_{ab}$ does not intersect $\st{x_T}{x_{T_2}}$,
by Lemmas~\ref{lem:2edgedeg4paths2} and \ref{lem:2edgedeg3paths2}. 
Thus $P_{ab}$ does not contain $e'$, which is a contradiction.
\end{proof}

\begin{figure}[t]
    \centering
    \begin{subfigure}{0.33\textwidth}
    \centering
        \includegraphics[scale=0.9,keepaspectratio]{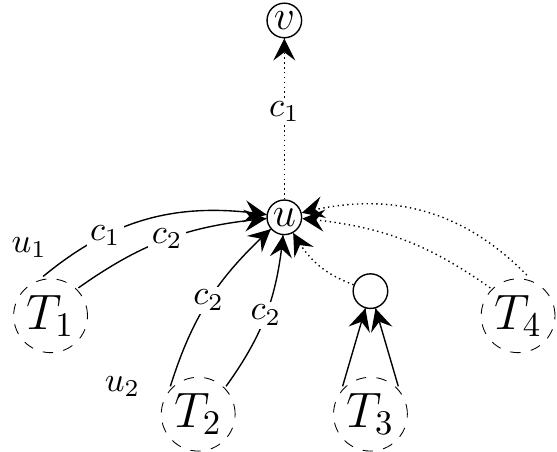}
        \caption{}
        %\label{fig:forest}
    \end{subfigure}%
    \hfill
    \begin{subfigure}{0.33\textwidth}
    \centering
		\includegraphics[scale=0.9,keepaspectratio]{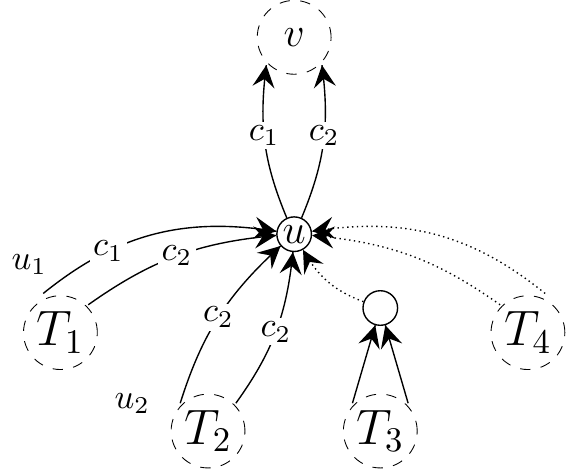}
		\caption{}
        %\label{fig:forest-contracted}
    \end{subfigure}%
    \hfill
    \begin{subfigure}{0.33\textwidth}
    \centering
		\includegraphics[scale=0.9,keepaspectratio]{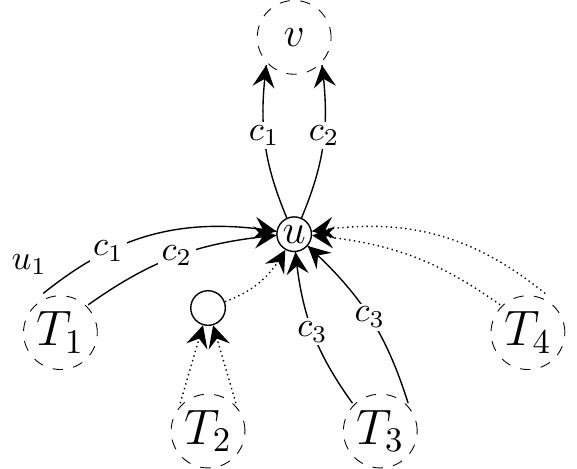}
		\caption{}
        %\label{fig:forest-contracted}
    \end{subfigure}
	\caption{Three examples of Coloring Rule~\ref{rul:nontreedeg3}. Here $c_i=s(T_i)$. The edges that were colored before the application of the rule are drawn as densely dotted lines.}
	\label{fig:nontreedeg3-all}
\end{figure}

%\begin{figure}
%    \centering
%    \begin{subfigure}[b]{\textwidth}
%    \centering
%        \includegraphics[scale=0.3]{nontreedeg3-1.png}
%		\caption{Coloring Rule~\ref{rul:nontreedeg3}: example 1}
%        %\caption{A gull}
%        \label{fig:nontreedeg3-1}
%    \end{subfigure}
%    ~ %add desired spacing between images, e. g. ~, \quad, \qquad, \hfill etc. 
%      %(or a blank line to force the sub-figure onto a new line)
%    \begin{subfigure}[b]{\textwidth}
%    \centering
%        \includegraphics[scale=0.3]{nontreedeg3-2.png}
%		\caption{Coloring Rule~\ref{rul:nontreedeg3}: example 2}
%		%\caption{Outgoing edge from $u$ is a 2-edge}
%        \label{fig:nontreedeg3-2}
%    \end{subfigure}
%	\caption{Three examples of Coloring Rule~\ref{rul:nontreedeg3}. Here $c_i=s(T_i)$. The blue edges are the edges that were already colored before the application of the rule. The figure is continued on the next page.}
%        \label{fig:nontreedeg3}
%\end{figure}
%\begin{figure}
%	\ContinuedFloat
%    \begin{subfigure}[b]{\textwidth}
%    \centering
%        \includegraphics[scale=0.3]{nontreedeg3-3.png}
%		\caption{Coloring Rule~\ref{rul:nontreedeg3}: example 3}
%        \label{fig:nontreedeg3-3}
%    \end{subfigure}
%\end{figure}

\begin{rul}
	\label{rul:nontreedeg3}
	For each non-tree vertex $u$ with degree at least $3$ in $B_1$ (see Figure~\ref{fig:nontreedeg3-all}),
	let $q$ be the number of children of $u$ (note that $q\ge 2$ as degree of $u$ is at least $3$), 
	let $u_1$, $u_2, \ldots, u_q$ be the children of $u$
	and let $x_{T_i}$ be $\ct{u}{u_i}$.
	Let $\dir{uv}$ be the outgoing edge from $u$ in $B_1$.
	If $uv$ is a $1$-edge,
	due to Coloring Rule~\ref{rul:1edge}, we know that there exist an $i\in[q]$
	such that
	$u_i$ is a tree vertex (and hence $T_i=f_{\calT}(u_i))$, 
	and $uv$ is colored with $s(T_i)$.
	Hence, if $uv$ is a $1$-edge, assume without loss of generality that 
	$u_1$ is a tree vertex 
	$($and hence $x_{T_1}=u_1) $
	and that $uv$ is colored with $s(T_1)$.
	
	\begin{itemize}

		\item If $u_1u$ is uncolored (then $u_1u$ is a 2-edge due to Coloring Rule~\ref{rul:1edge}, implying that $u_1$ is a tree vertex and hence $T_1=f_{\calT}(u_1)$), 
	then color 
	$(u_1u)_1$ with $s(T_1)$ and
	$(u_1u)_2$ with $s(T_2)$.

\item For each $2\le i\le q$, if $u_iu$ is uncolored (then $u_iu$ is a 2-edge due to Coloring Rule~\ref{rul:1edge}, implying that $u_i$ is a tree vertex and hence $T_i=f_{\calT}(u_i)$), 
	then color both its representatives with $s(T_i)$. 
\item
	Let $\dir{uv}$ be the outgoing edge from $u$.
	If $uv$ is uncolored (in which case it is a 2-edge due to Coloring Rule~\ref{rul:1edge})
	then color $(uv)_1$ with $s(T_1)$ and $(uv)_2$ with $s(T_2)$.
	\end{itemize}
\end{rul}

\begin{prul}
	\label{prul:nontreedeg3}
	For each non-tree vertex $u$ on which Coloring Rule~\ref{rul:nontreedeg3} has been applied as above and for each $P_{ab}$ such that $Q_{ab}$ contains $u$ (we say that the rule is being applied on the pair $\left( u,P_{ab} \right))$, execute the following two parts (in the mentioned order).

	\noindent {\bf Part 1}\
	\begin{itemize}
	
	\item If $Q_{ab}$ contains edge $u_1u$ and $u_1u$ is colored during the application of Coloring Rule~\ref{rul:nontreedeg3} on $u$, do the following.
	If the other neighbor (if any) of $u$ in $Q_{ab}$ is $u_2$, then add $(u_1u)_1$ (which has color $s(T_1))$ to $P_{ab}$.
	Otherwise, add $(u_1u)_2$ (which has color $s(T_2)$) to $P_{ab}$. 

	\item For each $i\in [2,q]$, if $Q_{ab}$ contains edge $u_iu$ and $u_iu$ is colored during the application of Coloring Rule~\ref{rul:nontreedeg3} on $u$,
		add $(u_iu)_1$ (which has color $s(T_i)$) to $P_{ab}$.

	\item If $Q_{ab}$ contains edge $uv$ and $uv$ is colored during the application of Coloring Rule~\ref{rul:nontreedeg3} on $u$:
	if the other neighbor (if any) of $u$ in $Q_{ab}$ is $u_1$ and $u_1u$ is a $1$-edge, then add $(uv)_2$ (which has color $s(T_2)$) to $P_{ab}$;
	otherwise add $(uv)_1$ (which has color $s(T_1)$) to $P_{ab}$.
	\end{itemize}
	\noindent {\bf Part 2}\
	\begin{itemize}
	\item For each tree vertex $x_T$ such that the degree of $x_T$ in $h(P_{ab})$ became $2$ during the addition of above edges in Part 1,
	let $x$ and $y$ be the endpoints in $T$ of the two edges of $P_{ab}$ incident on $T$.
	Add $T_{xy}$ to $P_{ab}$.

	\item For each tree vertex $x_T\in \left\{ h(a),h(b) \right\}$ such that the degree of $x_T$ in $h(P_{ab})$ became $1$ during the addition of above edges in Part~1, let $x$ be the endpoint in $T$ of the edge of $P_{ab}$ incident on $T$.
	If $x_T=h(a)$, add $T_{ax}$ to $P_{ab}$; otherwise (i.e., if $x_T=h(b)$), add $T_{bx}$ to $P_{ab}$.
	\end{itemize}
\end{prul}
\begin{lemma}
	\label{lem:invnontreedeg3}
Invariant \ref{inv:part-rainbow} is not violated during 
	Path Rule~\ref{prul:nontreedeg3}.
\end{lemma}
\begin{proof}
	Suppose Invariant \ref{inv:part-rainbow} is violated during the application of Path Rule~\ref{prul:nontreedeg3} on the pair $(u,P_{ab})$ as above.
Then there exist edges $e$ and $e'$ in $P_{ab}$ having the same color.
We can assume without loss of generality that $e$ was colored during the application of Coloring Rule~\ref{rul:nontreedeg3} on $u$.
Suppose $e$ was added during Part 2 of Path Rule~\ref{prul:nontreedeg3}. 
Observe that if we add a path inside a tree $T$ in Part 2, then $x_T$ was \emph{incomplete} before the application of Coloring Rule~\ref{rul:nontreedeg3}. 
By Invariant~\ref{inv:internalcolor}, this implies that 
the internal colors of $T$ were not used anywhere else so far. 
Thus, the color of $e$ is unique, in particular $c(e')\neq c(e)$, a contradiction.
Thus, the edge $e$ was not added during Part 2. 
Then $e$ was added during Part 1 and hence $c(e)=c(e')=s(T_i)$ for some $i\in [q]$.
Then $e'$ was colored during one of Coloring Rules \ref{rul:nontreedeg3}, \ref{rul:2edgedeg3}, \ref{rul:2edgedeg4}, or \ref{rul:1edge}.

\noindent {\bf Case 1:}\ $e'$ was colored during the Coloring Rule~\ref{rul:nontreedeg3}.

Note that during the application of Path Rule~\ref{prul:nontreedeg3} on $(u,P_{ab})$, 
we have added at most two edges to $P_{ab}$.
And, if we have added two edges, they are of different color.
Thus $e'$ was not added to $P_{ab}$ during the application of Path Rule~\ref{prul:nontreedeg3} on $(u,P_{ab})$ and 
hence was not colored during the application of Coloring Rule~\ref{rul:nontreedeg3} on $u$.
So, $e'$ was colored during the application of Coloring Rule~\ref{rul:nontreedeg3} on some non-tree vertex $u'\neq u$.
Notice that for any tree $T\in \calT$, $s(T)$ is used during the application of Coloring Rule~\ref{rul:nontreedeg3} only when the rule is applied to an ancestor of $x_T$ in $B_1$.
Hence, both $u$ and $u'$ are ancestors of $x_{T_i}$.
Without loss of generality, assume that $u'$ is closer than $u$ to $x_{T_i}$.
Then, $u$ cannot have any tree vertices as children because otherwise $x_{T_i}\neq \ct{u}{u_i}$. 
Then, the only edges colored during the application of Coloring Rule~\ref{rul:nontreedeg3} on $u$, are the representatives of $uv$. 
Thus $h(e)=uv$.

\noindent {\bf Case 1.1}\ $e=(uv)_2$.

We know that $e=(uv)_2$ is colored with $s(T_2)$ by Coloring Rule~\ref{rul:nontreedeg3}.
Thus, $c(e')=c(e)=s(T_2)$ and $T_i=T_2$.
Since the edge $(uv)_2$ is added during application of Path Rule~\ref{prul:nontreedeg3} on $(u,P_{ab})$, 
the neighbors of $u$ in $Q_{ab}$ are $v$ and $u_1$, by Path Rule~\ref{prul:nontreedeg3}. 
Since $u'\in Q_{ab},$ we have that $u'$ is a descendant of $u_1$ and not $u_2$ in $B_1$. 
This implies $T_i=T_1\neq T_2$, 
a contradiction.

\noindent {\bf Case 1.2} $e=(uv)_1$.

Since the edge $(uv)_1$ is added during application of Path Rule~\ref{prul:nontreedeg3} on $(u,P_{ab})$, 
either $u_1$ is not a neighbor of $u$ in $Q_{ab}$, or $uu_1$ is a 2-edge, by Path Rule~\ref{prul:nontreedeg3}.
But $uu_1$ cannot be a 2-edge as both $u$ and $u_1$ are non-tree vertices. (Recall that we said all children of $u$ are non-tree vertices in Case 1).
Hence $u_1$ is not a neighbor of $u$ in $Q_{ab}$.
Since $u'\in Q_{ab},$ this implies that $T_i\neq T_1$, and hence $c(e)=c(e')=s(T_i)\neq s(T_1)$.
But we know that $e=(uv)_1$ is colored with $s(T_1)$, by Coloring Rule~\ref{rul:nontreedeg3}.
Thus, we have a contradiction.

\noindent {\bf Case 2:}\ $e'$ was colored during the Coloring Rules~\ref{rul:2edgedeg3} or \ref{rul:2edgedeg4}.

Let $T'$ be the tree on which $e'$ is incident.
Then $e'$ was colored with $s(T_i)$ during the application of Coloring Rules~\ref{rul:2edgedeg3} or \ref{rul:2edgedeg4} on $x_{T'}$. 
Then $Q_{ab}$ does not intersect $\st{T'}{T_i}$ due to Lemmas~\ref{lem:2edgedeg3paths2} and \ref{lem:2edgedeg4paths2}.
Since $x_{T_i}=\ct{u}{u_i}$, there is no other tree-vertex in the path from $u$ to $x_{T_i}$. 
Thus, $u$ is in $\st{T'}{T_{i}}$. 
Hence, we have that $u$ is not in $Q_{ab}$.
We know that $e$ is adjacent on $u$ as every edge colored during the application of Coloring Rule~\ref{rul:nontreedeg3} on $u$ is incident on $u$.
But then $e\notin P_{ab}$ as $u$ is not in $Q_{ab}$.
This is a contradiction.

\noindent {\bf Case 3:}\ $e'$ was colored during the Coloring Rule~\ref{rul:1edge}.

This means that $e'$ is a $1$-edge.
%The only possibility for $1$-edge $e'$ to have color $s(T_{i})$ is either $e'=uv$ or that $T_i=T_2$ and $e'$ is on the path between $x_{T_2}$ and $u$.

\noindent {\bf Case 3.1}\ $h(e')=uv$.

In the case when $uv$ is a $1$-edge, we selected $u_1$ during Coloring Rule~\ref{rul:nontreedeg3} in such a way that $c(uv)=s(T_1)$.
Thus $c(e)=c(e'=uv)=s(T_1)$.
The only edges that can be potentially colored with $s(T_1)$ during the application of Coloring Rule~\ref{rul:nontreedeg3} on $u$ are $(uv)_1$ and $(uu_1)_1$.
Since $e$ and $e'$ are distinct
we have $e=(u_1u)_1$.
But since $uv$ is in $Q_{ab}$, 
we would have added $(u_1u)_2$ and not $(u_1u)_1$ to $P_{ab}$ during Path Rule~\ref{prul:nontreedeg3}.
Thus we have a contradiction.

\noindent {\bf Case 3.2.}\ $h(e')\neq uv$.

Then $e'$ is on the path between $x_{T_i}$ and $u$. 
Also, $x_{T_i}$ is not a child of $u$. 
Then, the only possibility for $e$ to have color $s(T_i)$ is if $i=2$ and $e=(u_1u)_2$.
Then $Q_{ab}$ contains both $u_1$ and $u_2$.
In that case, 
we would have added $(u_1u)_1$ and not $(u_1u)_2$ to $P_{ab}$ during Path Rule~\ref{prul:nontreedeg3}.
Hence, $e\neq (u_1u)_2$, a contradiction.
\end{proof}
%For a 2-edge $e$ in $B$ incident on tree vertex $x_T$, we define its {\bf foot-path} as the path between the foots of $e$ in $T$.
\begin{rul}
	\label{rul:treevertex2deg1}
For each incomplete tree vertex $x_T$ having 2-edge degree exactly $1$: let $e$ be the only 2-edge incident on $x_T$, 
pick an edge $e_1$ in the foot-path of $e$, 
color the representatives of $e$ with the color of $e_1$. 
\end{rul}
\begin{prul}
	\label{prul:treevertex2deg1}
	For each tree vertex $x_T$ on which Coloring Rule~\ref{rul:treevertex2deg1} has been applied as above and for each $P_{ab}$ such that $Q_{ab}$ contains $h(e)$ (we say that the path rule is being applied on the pair $\left( x_T,P_{ab} \right)$), do the following.

	We pick vertex $w$ as follows. If $a\in V(T)$, let $w:=a$, and if $b\in V(T)$ let $w:=b$.
	(Note that both $a$ and $b$ cannot be in $T$ as $Q_{ab}$ contains $h(e)$).
	If $a,b\notin V(T)$ then there is an edge $e_2\neq e$ of $Q_{ab}$ incident on $x_T$.
	Furthermore, since $e$ is the only 2-edge incident on $x_T$, the edge $e_2$ is a $1$-edge.
	In this case, let $w$ be the endpoint of $(e_2)_1$ in $T$.

	By Observation~\ref{obs:edgeexclusion}, there is a path in $T$ that excludes $e_1$, from $w$ to one of the foots of $e$.
	Let this foot be $z$.
	Add the path in $T$ between $w$ and $z$ to $P_{ab}$.
	Also add to $P_{ab}$ the representative of $e$ having $z$ as its endpoint in $T$.
\end{prul}
\begin{lemma}
	\label{lem:invtreevertex2deg1}
Invariant \ref{inv:part-rainbow} is not violated during 
	Path Rule~\ref{prul:treevertex2deg1}.
\end{lemma}
\begin{proof}
	Let $E_N$ be the set of new edges
	added to $P_{ab}$ during the application of Path Rule~\ref{prul:treevertex2deg1} to $(x_T,P_{ab})$ and let $E_O$ be the set of already included edges in $P_{ab}$ before this application.
	Suppose Invariant \ref{inv:part-rainbow} is violated for the sake of contradiction.
	Then either there are two edges in $E_N$ with the same color or $c(E_N)\cap c(E_O)\neq \emptyset$.
	Recall that $E_N$ consists of $E(T_{wz})$ and a representative of $e$, say $(e)_j$.
	Note that $c((e)_j)=c(e_1)$ by Coloring Rule~\ref{rul:treevertex2deg1}. 
	So, all the  edges in $E_N$ are colored from $c(T)$, the internal colors of $T$.
	Recall that $e_1$ is not in $T_{wz}$  by our choice of $z$. 
	Thus the edges in $E_N$ all have distinct colors.
	So, it has to be the case that $c(E_O)\cap c(E_N)\neq \emptyset$.
	Since $c(E_N)\subseteq c(T)$, this implies that $c(E_O)\cap c(T)\neq \emptyset$.
	Let $d$ be an edge in $E_O$ with color in $c(T)$.
	%Notice that we add the internal edges of a tree $T$ to $P_{ab}$ only once the representatives of all the edges of $Q_{ab}$ incident on $x_T$ has been added to $P_{ab}$.
	Since the representative of at least one edge of $Q_{ab}$ incident on $x_T$ (namely $e$) was not added to $P_{ab}$ before the application of Path Rule~\ref{prul:treevertex2deg1}, we have that $E(T)\cap E_O=\emptyset$ by Invariant~\ref{inv:internaledgespath}.
	Thus $d\notin E(T)$ but has color in $c(E(T))$ and $d\in E_O$.
	By Invariant~\ref{inv:internalcolor}, this implies that $x_T$ was complete before the application of Coloring rule~\ref{rul:treevertex2deg1}, making the rule not applicable on $x_T$, a contradiction.
\end{proof}

\begin{figure}[t]
    \centering
    \begin{subfigure}{0.49\textwidth}
    \centering
        \includegraphics[scale=1.00,keepaspectratio]{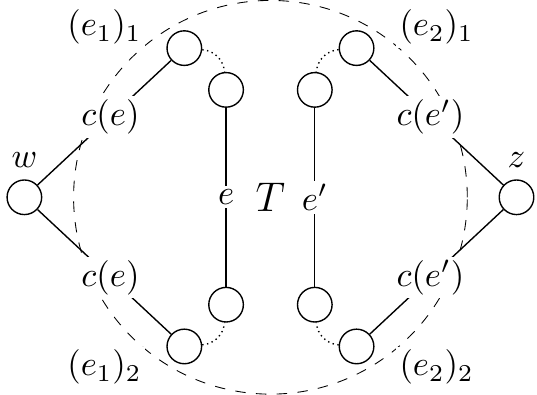}
        \caption{}
        %\label{fig:forest}
    \end{subfigure}%
    \hfill
    \begin{subfigure}{0.49\textwidth}
    \centering
		\includegraphics[scale=1.00,keepaspectratio]{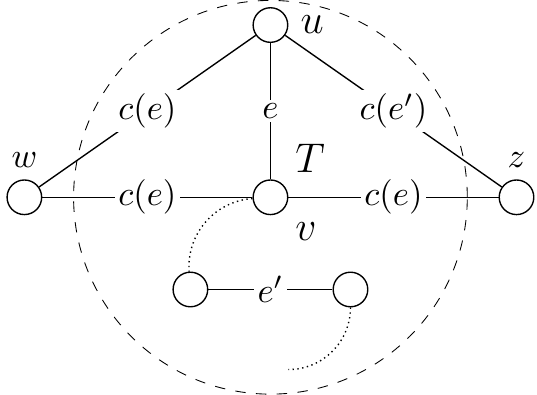}
		\caption{}
        %\label{fig:forest-contracted}
    \end{subfigure}
    \caption{Coloring Rule~\ref{rul:treevertex2deg2}. \textbf{(a)} Case~1 \textbf{(b)} Case~2.}
	\label{fig:tree2deg2-all}
\end{figure}

\begin{rul}
	\label{rul:treevertex2deg2}
	For each tree vertex $x_T$ such that the 2-edge degree of $x_T$ is exactly $2$ and $E(T)$ contains at least $2$ edges,
	let $e_1$ and $e_2$ be the 2-edges incident on $x_T$,
	let $w$ and $z$ be the other endpoints of $e_1$ and $e_2$, respectively, and
	let $P_1$ and $P_2$ be the foot-paths of $e_1$ and $e_2$, respectively.
	
	\noindent {\bf Case 1:}\ $|E(P_1\cup P_2)|\ge 2$. (See Figure~\ref{fig:tree2deg2-all}~(a)).
	
	Pick distinct edges $e$ and $e'$ from $P_1$ and $P_2$ respectively.
	If $(e_1)_1$ and $(e_1)_2$ are uncolored, color them with color of $e$ and 
	if $(e_2)_1$ and $(e_2)_2$ are uncolored, color them with color of $e'$.
	
	\noindent {\bf Case 2:}\ Case 1 does not hold. (See Figure~\ref{fig:tree2deg2-all}~(b)).
	
	Clearly, $P_1$ and $P_2$ both are a single edge and they are the same edge.
	Let this edge be $e=uv$.
	Pick any other edge $e'$ in $T$ (such an edge exists because we said that the rule is applicable only if $E(T)$ contains at least two edges).
	Without loss of generality, assume that $e'$ is closer to $v$ than $u$ in $T$.
	If $uw$ and $vw$ are uncolored, color them with color of $e$. 
	If $uz$ and $vz$ are uncolored, color them with colors of $e'$ and $e$, respectively.
\end{rul}

\begin{lemma}
	\label{lem:treevertex2deg2}
	Consider a tree vertex $x_T$ on which Coloring Rule~\ref{rul:treevertex2deg2} has been applied as above.
	There is a rainbow path in $G$ from $w$ to $z$ using only the colors in $c(T)$.
	Also, from any vertex $x\in V(T)$, there is a rainbow path to both $w$ and $z$
	using only the colors in $c(T)$.
\end{lemma}
\begin{proof}
	If Case 2 of Coloring Rule \ref{rul:treevertex2deg2} has been applied then this is rather easy to see as follows.
	The path $wuz$ is a rainbow path from $w$ to $z$.
	Also for the second statement, note that at least one of $T_{xu}$ or $T_{xv}$ avoids $e$. Then at least one of the paths $T_{xu}$ followed by $uw$ or $uz$, or $T_{xv}$ followed by $vw$ or $vz$, is a rainbow path (it avoids $e$ and if it contains $e'$, then it is the second case and $c(e')$ is used only once).

	So, it only remains to prove the lemma when Case 1 of Coloring Rule \ref{rul:treevertex2deg2} is applied.
	Let $w_1,w_2$ be the foots of $e_1$ and $z_1,z_2$ be the foots of $e_2$.
	To prove the first statement, it is sufficient to prove that at least one of the four paths $T_{w_1z_1},T_{w_1z_2},T_{w_2z_1}$ and $T_{w_2z_2}$ contains neither $e$ nor $e'$.
	Given this, it is easy to show the necessary rainbow path from $w$ to $z$: if the path $T_{w_iz_j}$ contains neither $e$ nor $e'$ then the path $ww_iT_{w_iz_j}z_jz$ is the required path. So for the sake of contradiction assume that each $T_{w_iz_j}$ contains either $e$ or $e'$.
	Without loss of generality assume that $T_{w_1z_1}$ contains $e$. 
	Let $y$ be the last vertex on $T_{w_1z_1}$ that is in $P_1$ (while going from $w_1$ to $z_1$).
	Now, $T_{yz_1}$ and $T_{yw_2}$ do not contain $e$ and hence $T_{z_1w_2}=T_{yz_1}\cup T_{yw_2}$ does not contain $e$.
	This implies that  $T_{z_1w_2}$ contains $e'$.
	Let $y'$ be the last vertex on $T_{z_1w_2}$ that is in $P_2$ (while going from $z_1$ to $w_2$).
	Now, $T_{y'z_2}$ and $T_{w_2y'}$ contains neither $e$ nor $e'$ and hence $T_{w_2z_2}=T_{y'z_2}\cup T_{w_2y'}$ contains neither $e$ nor $e'$.

	To prove the second statement, observe that there is 
	a rainbow path from $x$ to either $w_1$ or $w_2$ not containing $e$, and 
	a rainbow path from $x$ to either $z_1$ or $z_2$ not containing $e'$,
	due to Observation~\ref{obs:edgeexclusion}.
\end{proof}
\begin{prul}
	\label{prul:treevertex2deg2}
	For each tree vertex $x_T$ on which Coloring Rule~\ref{rul:treevertex2deg2} has been applied as above and for each  $P_{ab}$ such that $Q_{ab}$ contains at least one of $e_1$ and $e_2$ (we say that the path rule is being applied on the pair $\left( x_T,P_{ab} \right))$, do the following.
	
	If $Q_{ab}$ contains both $e_1$ and $e_2$ then let $y_1:=w$ and  $y_2:=z$.
	If $Q_{ab}$ contains only $e_1$ and not $e_2$ then let $y_1:=w$. 
	If $Q_{ab}$ contains only $e_2$ and not $e_1$ then let $y_1:=z$. 
	If $a\in V(T)$, let $y_2:=a$, and if $b\in V(T)$, let $y_2:=b$.
	(Note that both $a$ and $b$ cannot be in $T$ as $Q_{ab}$ contains $e_1$ or $e_2$).
	If $a,b\notin V(T)$ and only one of $e_1,e_2$ is in $Q_{ab}$, then there is an edge $e''\notin \{e_1,e_2\}$ incident on $x_T$ in $Q_{ab}$;
	and since $e_1$ and $e_2$ are the only 2-edges incident on $x_T$, the edge $e''$ is a $1$-edge;
	let $y_2$ be the endpoint of $(e'')_1$ in $T$.
	Add to $P_{ab}$ the path between $y_1$ and $y_2$ 
given by Lemma~\ref{lem:treevertex2deg2}.
\end{prul}
\begin{lemma}
	\label{lem:invtreevertex2deg2}
Invariant \ref{inv:part-rainbow} is not violated during 
	Path Rule~\ref{prul:treevertex2deg2}.
\end{lemma}
\begin{proof}
	Let $E_N$ be the set of new edges
	added to $P_{ab}$ during the application of Path Rule~\ref{prul:treevertex2deg2} to $(x_T,P_{ab})$ and let $E_O$ be the set of already included edges in $P_{ab}$ before this application.
	Suppose Invariant \ref{inv:part-rainbow} is violated for the sake of contradiction.
	Then either there are two edges in $E_N$ with the same color or $c(E_N)\cap c(E_O)\neq \emptyset$.
	By Lemma~\ref{lem:treevertex2deg2}, all edges in $E_N$ are colored from $c(T)$ and have distinct colors.
	So, it has to be the case that $c(E_O)\cap c(E_N)\neq \emptyset$.
	Since $c(E_N)\subseteq c(T)$, this implies that $c(E_O)\cap c(T)\neq \emptyset$.
	Let $d$ be an edge in $E_O$ with color in $c(T)$.
	The representative of at least one edge of $Q_{ab}$ incident on $x_T$  was not in $P_{ab}$ before the application of Path Rule~\ref{prul:treevertex2deg2} on $(x_T,P_{ab})$, because otherwise the path rule is not applicable on $(x_T,P_{ab})$. Then, by Invariant~\ref{inv:internaledgespath}, we have that $E(T)\cap E_O=\emptyset$.
	Thus $d\notin E(T)$ but has color in $c(E(T))$ and $d\in E_O$.
	But then by Invariant~\ref{inv:internalcolor}, we have that $x_T$ was completed before the application of Coloring Rule~\ref{rul:treevertex2deg2}, thereby making the rule not applicable on $x_T$, which is a contradiction.
\end{proof}
\begin{figure}[t]
    \centering
    \begin{subfigure}[b]{0.49\textwidth}
    \centering
        \includegraphics[scale=1.0,keepaspectratio]{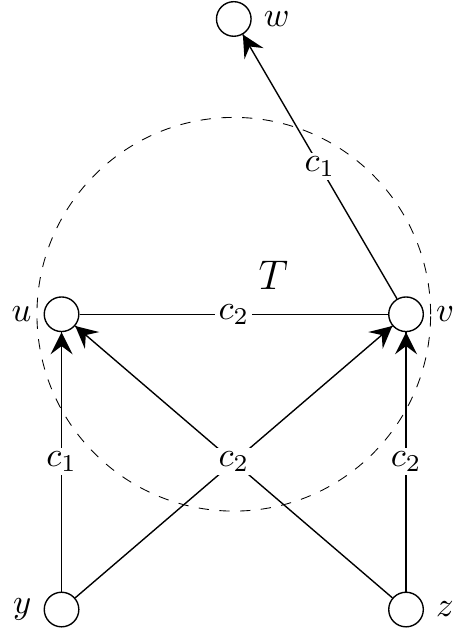}
        \caption{}
		%\caption{Coloring Rule~\ref{rul:treevertexdeg3} Case 1. Here $c_1=s(T)$ and $c_2=c(uv)$.}
        %\label{fig:treedeg3-1}
    \end{subfigure}%
    \hfill
    \begin{subfigure}[b]{0.49\textwidth}
    \centering
        \includegraphics[scale=1.0,keepaspectratio]{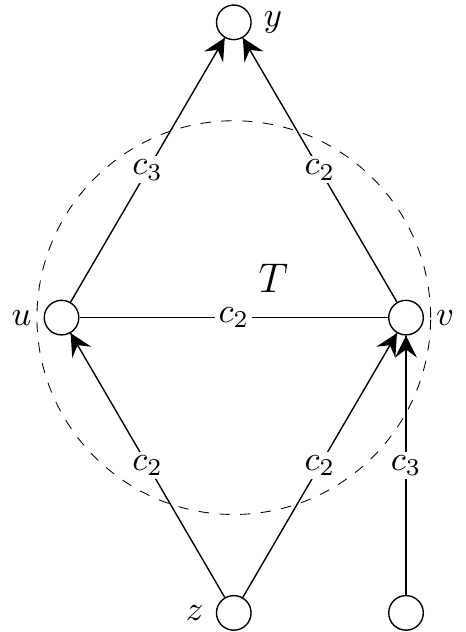}
        \caption{}
		%\caption{Coloring Rule~\ref{rul:treevertexdeg3} Case 2. Here $c_2=c(uv)$ and $c_3$ is the color of the representative of an arbitrarily chosen $1$-edge incident on $x_T$.}
        %\label{fig:treedeg3-2}
    \end{subfigure}
		\caption{Coloring Rule~\ref{rul:treevertexdeg3}. \textbf{(a)} Case~1 where $c_1=s(T)$ and $c_2=c(uv)$. \textbf{(b)} Case~2 where $c_2=c(uv)$ and $c_3$ is the color of the representative of an arbitrarily chosen 1-edge incident on $x_T$.}
\label{fig:treedeg3-1-2}
\end{figure}
\begin{rul}
	\label{rul:treevertexdeg3}	
	For each incomplete tree vertex $x_T$ having degree at least $3$ in $B_1$: 
	We can assume that Coloring Rules~\ref{rul:2edgedeg4}, \ref{rul:2edgedeg3}, \ref{rul:treevertex2deg1}, \ref{rul:treevertex2deg2} are not applicable on $x_T$ as otherwise $x_T$ would have been completed.
	If $x_T$ had at least three 2-edges incident on it, then Coloring Rule~\ref{rul:2edgedeg4} or \ref{rul:2edgedeg3} would have been applicable on $x_T$. 
	If it had 2-edge degree $1$, then Coloring Rule~\ref{rul:treevertex2deg1} would have been applicable on $x_T$.
	If it had 2-edge degree $0$, then it would have been completed after Coloring Rule~\ref{rul:1edge}. 
	Hence, we can assume that $x_T$ has 2-edge degree exactly $2$.
	Now, if $|E(T)|\ge 2$, Coloring Rule~\ref{rul:treevertex2deg2} becomes applicable on $x_T$.
	Hence, we can assume that the tree $T$ is just an edge.
	Let $uv$ be this edge.
	Let the two 2-edges incident on $x_T$ be $yx_T$ and $zx_T$.
	
	\noindent {\bf Case 1:}\ Both  $yx_T$ and $zx_T$ are incoming to $x_T$ (see Figure~\ref{fig:treedeg3-1-2}~(a)).
	
	Then the outgoing edge of $u$ in $B_1$ is a $1$-edge, say $\dir{x_Tw}$.
	Assume without loss of generality that 
	 its representative is $vw$. 
	Let $c_1=s(T)$ and $c_2$ be the color of $uv$.
	Note that $vw$ is colored with $c_1$ due to Coloring Rule~\ref{rul:1edge}.
	If $yu$ and $yv$ are uncolored, color them with $c_1$ and $c_2$ respectively. 
	If $zu$ and $zv$ are uncolored, color both of them with $c_2$.
	
	\noindent {\bf Case 2:}\ One of the 2-edges, say $yx_T$ is outgoing from $x_T$ (see Figure~\ref{fig:treedeg3-1-2}~(b)).
	
	Let $c_2$ be the color of $uv$ and $c_3$ be the color of representative of any $1$-edge incoming on $x_T$.
	Note that at least one such $1$-edge exists as the degree of $x_T$ is at least $3$.
	If $yu$ and $yv$ are uncolored, color them with $c_3$ and $c_2$ respectively. 
	If $zu$ and $zv$ are uncolored, color both of them with $c_2$. 
\end{rul}

\begin{prul}
	\label{prul:treevertexdeg3}
	For each tree vertex $x_T$ on which Coloring Rule~\ref{rul:treevertexdeg3} has been applied as above and for each  $P_{ab}$ such that $Q_{ab}$ contains at least one of $yx_T$ and $zx_T$ (we say that the path rule is being applied on the pair $\left( x_T,P_{ab} \right)$), do the following:
	\begin{itemize}
	\item If $Q_{ab}$ contains both $yx_T$ and $zx_T$ then add $yu$ and $uz$ to $P_{ab}$. 
	
	\item If $Q_{ab}$ contains only $yx_T$ and not $zx_T$ then let $y_1:=y$. 
	If $Q_{ab}$ contains only $zx_T$ and not $yx_T$ then let $y_1:=z$. 
	If $a\in V(T)$, let $y_2:=a$, and if $b\in V(T)$ let $y_2:=b$.
	(Note that both $a$ and $b$ cannot be in $T$ as $Q_{ab}$ contains $yx_T$ or $zx_T$).
	If $a,b\notin V(T)$ and only one of $yx_T,zx_T$ is in $Q_{ab}$ then there is an edge $e''\notin \{yx_T,zx_T\}$ incident on $x_T$ in $Q_{ab}$.
	Further, since $yx_T$ and $zx_T$ are the only 2-edges incident on $x_T$, the edge $e''$ is a $1$-edge.
	Note that $(e'')_1$ is already added to $P_{ab}$ in Path rule~\ref{prul:1edge}.
	Let $y_2$ be the endpoint of $(e'')_1$ in $T$.\\
	Note that in all cases $y_2\in \left\{ u,v \right\}$ and $y_1\in \{y,z\}$. Hence, the edge $y_1y_2$ exists.
	Add the edge $y_1y_2$ to $P_{ab}$.
	\end{itemize}
\end{prul}
\begin{lemma}
	\label{lem:invtreevertexdeg3}
Invariant \ref{inv:part-rainbow} is not violated during 
	Path Rule~\ref{prul:treevertexdeg3}.
\end{lemma}
\begin{proof}
	Suppose the invariant is violated.
Then there exist edges $e$ and $e'$ in $P_{ab}$ having the same color.
We can assume without loss of generality that $e$ was colored during the application of Coloring Rule~\ref{rul:treevertexdeg3} on $(x_T,P_{ab})$.
We added at most two edges during the application of Path Rule~\ref{prul:treevertexdeg3} on $x_T$ and if we added two edges we have made sure they have distinct colors.
Thus $e'$ was not added during the application of Path Rule~\ref{prul:treevertexdeg3} on $(x_T,P_{ab})$.
The colors that are possible for $e$ are $c_1,c_2$ and $c_3$ according to Coloring Rule~\ref{rul:treevertexdeg3}.

\noindent {\bf Case 1:}\ $c(e)=c(e')=c_2$.

Recall $c_2=c(uv)$.	
If $e'=uv$, then by Invariant~\ref{inv:internaledgespath}, the representatives of all the edges of $Q_{ab}$ incident on $x_T$ are in $P_{ab}$ even before the application of Path Rule~\ref{prul:treevertexdeg3} on $(x_T,P_{ab})$. Then Path Rule \ref{prul:treevertexdeg3} is not applicable on $\left( x_T,P_{ab} \right)$.
Thus $e'\notin E(T)$ but $c(e')\in c(E(T))$. 
Then by Invariant~\ref{inv:internalcolor}, $x_T$ was completed before the application of Coloring Rule~\ref{rul:treevertexdeg3}, making the rule not applicable on $x_T$.
Thus, such an $e'$ does not exist.
	
	\noindent {\bf Case 2:}\ $c(e)=c(e')=c_1=s(T)$. 
	
	This means $e=yu$ and that Case 1 of Coloring Rule~\ref{rul:treevertexdeg3} (see Figure~\ref{fig:treedeg3-1-2}~(a)) was applied on $x_T$.
	The only coloring rules so far that use surplus colors are Coloring Rules
	~\ref{rul:treevertexdeg3}, \ref{rul:nontreedeg3}, \ref{rul:2edgedeg3}, \ref{rul:2edgedeg4}, and \ref{rul:1edge}.
	
	\noindent {\bf Case 2.1}\ $e'$ was colored during Coloring Rule~\ref{rul:1edge}.
	 
	Note that this means $e'$ is a $1$-edge
	and the only way $e'$ can have color $s(T)$ is if $e'=vw$. 
	But, in Path Rule~\ref{prul:treevertexdeg3}, we add $yu=e$ to $P_{ab}$ only when $vw$ is not in $Q_{ab}$.
	Thus we have a contradiction.
	
	\noindent {\bf Case 2.2}\ $e'$ was colored during Coloring Rules~\ref{rul:2edgedeg3} or \ref{rul:2edgedeg4}.
	
Let $T'$ be the tree on which $e'$ is incident.
Since $e'$ was colored with $s(T)$ during Coloring Rules~\ref{rul:2edgedeg3} or \ref{rul:2edgedeg4},
we know that $Q_{ab}$ does not intersect $\st{x_{T'}}{x_T}$ due to Lemmas~\ref{lem:2edgedeg3paths2} and \ref{lem:2edgedeg4paths2}.
Then $Q_{ab}$ does not contain $x_{T}$ and hence
$P_{ab}$ does not contain $e$,
which is a contradiction.

	\noindent {\bf Case 2.3}\ $e'$ was colored during Coloring Rule~\ref{rul:nontreedeg3}.
	
	Then $h(e')$ is a 2-edge in the path from $x_T$ to root of $B_1$. 
	Since $e'$ is in $Q_{ab}$, this means that $x_Tw$ is in $Q_{ab}$.
	But then by Path Rule~\ref{prul:nontreedeg3}, we would have added $yv$ instead of $yu=e$ to $P_{ab}$, a contradiction.
	
	\noindent {\bf Case 2.4}\ $e'$ was colored during Coloring Rule~\ref{rul:treevertexdeg3}.
	
	The only application of Coloring Rule~\ref{rul:treevertexdeg3} that uses $s(T)$ is the application on $x_T$. 
	But since $e'$ was not colored during this application, we have a contradiction.
	
	\noindent {\bf Case 3:}\ $c(e)=c(e')=c_3$.
	
	This means $e=uy$ and that Case 2 of Coloring Rule~\ref{rul:treevertexdeg3} was applied on $x_T$.
	Let $x$ be the neighbor of $x_T$ such that $xx_T$ is the $1$-edge incident on $x_T$ whose representative is colored with $c_3$. 
	By Coloring Rule~\ref{rul:1edge}, there exist a tree $T'$ that is a descendant of $x$ such that $s(T')=c_3$.
	%and $x_{T'}$ is a descendant of $x$.
	%Moreover, $x_{T'}=\ct{x_T}{x}$.
	The only coloring rules so far that use surplus colors are Coloring Rules
	~\ref{rul:treevertexdeg3}, \ref{rul:nontreedeg3}, \ref{rul:2edgedeg3}, \ref{rul:2edgedeg4}, and \ref{rul:1edge}.
	%Hence we can take the following cases according to how $e'$ was colored.
	
	\noindent {\bf Case 3.1}\ $e'$ was colored during Coloring Rule~\ref{rul:1edge}.
	 
	This means $e'$ is a $1$-edge.
	Since $xx_T$ is the only $1$-edge with color $s(T')$ by Lemma~\ref{lem:inv1-1edge}, we have that $e'=xx_T$.
	Hence, $xx_T$ is in $Q_{ab}$.
	But if $xx_T$ is in $Q_{ab}$, we would have added $vy$ and not $uy$ in Path Rule~\ref{prul:treevertexdeg3}.
	This is a contradiction to $e=uy$.
	
	\noindent {\bf Case 3.2}\ $e'$ was colored during Coloring Rules~\ref{rul:2edgedeg3} or \ref{rul:2edgedeg4}.
	
	Let $T''$ be such that $e'$ is adjacent on $x_{T''}$.
	Since $e'$ was colored with $s(T')$ during Coloring Rules~\ref{rul:2edgedeg3} or \ref{rul:2edgedeg4},
	we have that $Q_{ab}$ does not intersect $\st{x_{T''}}{x_{T'}}$ due to Lemmas~\ref{lem:2edgedeg3paths2} and \ref{lem:2edgedeg4paths2}.
	%Since $P_{ab}$ contains $e'$ that is incident on $x_{T''}$, we have that $Q_{ab}$ contains $x_{T''}$.
	Since $P_{ab}$ contains $e$ that is incident on $x_{T}$, we have that $Q_{ab}$ contains $x_{T}$.
	%This implies that $Q_{ab}$ contains $x$ as $x$ is on the path from $x_T$ to $x_{T''}$.
	This implies that $x_T\notin \st{x_{T''}}{x_{T'}}$ implying that $x_{T''}$ is on the path between $x_T$ and $x_{T'}$. 
	But then $s(T'')$ and not $s(T')$ would have been used to color $xx_T$, a contradiction.
	
	\noindent {\bf Case 3.3}\ $e'$ was colored during Coloring Rule~\ref{rul:nontreedeg3} or \ref{rul:treevertexdeg3}.
	
	Since $e'$ is colored with $s(T')$, by Coloring Rule~\ref{rul:nontreedeg3} and \ref{rul:treevertexdeg3} this 
	implies $e'$ is in $\st{x_T}{x_{T'}}=\st{x_T}{x}$,
	implying that $Q_{ab}$ contains $x$. 
	But then we would have added $vy$ and not $uy=e$ to $P_{ab}$ according to Path Rule~\ref{prul:treevertexdeg3}, a contradiction.
\end{proof}
The following Lemma follows from the previous coloring rules.
\begin{lemma}
	\label{lem:rul1-8}
	Consider an edge $e$ in $B_1$ that remains uncolored after the application of Coloring Rules~\ref{rul:forest} through~\ref{rul:treevertexdeg3}.
	Let $x_T$ and $v$ be the endpoints of $e$
	(Note that due to Coloring Rule~\ref{rul:1edge}, $e$ is a $2$-edge, and hence one of its endpoints is a tree-vertex). 
	Then, both $x_T$ and $v$ have degree exactly $2$ in $B_1$,
	both edges incident on $x_T$ are 2-edges, and $T$ consists of a single edge. 
\end{lemma}
\begin{proof}
	Suppose $u\in \left\{ v,x_T \right\}$ had degree not equal to $2$ in $B_1$.	
	First, suppose the degree was greater than $2$.
Then Coloring Rule~\ref{rul:treevertexdeg3} or \ref{rul:nontreedeg3} would have been applicable on $u$, and hence $u$ would have been completed. 
Therefore, $u$ has degree $1$ in $B_1$.
By Corollary~\ref{cor:leafB1}, every leaf of $B_1$ is a tree vertex.
Hence, $u$ is a tree vertex and $u=x_T$. 
But then Coloring Rule~\ref{rul:treevertex2deg1} would have been applicable on $x_T$, and $x_T$ would have been completed.
Thus, $e$ is already colored, which is a contradiction.
Hence, $x_T$ and $v$ have degree $2$ in $B_1$.

Now, suppose $x_T$ 
has only one 2-edge incident in $B_1$. 
Then, Coloring Rule~\ref{rul:treevertex2deg1} would have 
been applied on $x_T$ and $x_T$ would have been completed.
Thus, both edges incident on $x_T$ in $B_1$ are 2-edges.
If $T$ contained at least two edges, Coloring Rule~\ref{rul:treevertex2deg2} would have
been applied on $x_T$ and $x_T$ would have been completed.
Hence, $T$ contains only one edge.
\end{proof}

\begin{rul}
	\label{rul:1uncol2edge}
	For each tree vertex $x_T$ with exactly one uncolored 2-edge $e$ incident on it: note that it follows by Lemma~\ref{lem:rul1-8} that the tree $T$ comprises of a single edge $e'$. Let $e=vx_T$ and $e':=u_1u_2$.
	Color $(e)_1=vu_1$ and $(e)_2=vu_2$ with the color of $e'$.
\end{rul}
\begin{prul}
	\label{prul:1uncol2edge}
	For each tree vertex $x_T$ on which Coloring Rule~\ref{rul:1uncol2edge} has been applied as above and for each  $P_{ab}$ such that $Q_{ab}$ contains $e$ (we say that the path rule is being applied on the pair $\left( x_T,P_{ab} \right)$), do the following:

		First we pick vertex $w\in V(T)=\{u_1,u_2\}$ as follows: if $a\in V(T)$, let $w:=a$; if $b\in V(T)$ let $w:=b$;
	(note that both $a$ and $b$ cannot be in $T$ as $Q_{ab}$ contains $e$);
	if $a,b\notin V(T)$ then there is an edge $e_2\neq e$ of $Q_{ab}$ incident on $x_T$;
	furthermore, since $e$ is the only uncolored edge incident on $x_T$, a representative of the edge $e_2$ is already in $P_{ab}$;
	let $w$ be the endpoint in $T$ of this representative of $e_2$. \\
	Add $vw$ to $P_{ab}$.
\end{prul}
\begin{lemma}
	\label{lem:inv1uncol2edge}
Invariant \ref{inv:part-rainbow} is not violated during 
	Path Rule~\ref{prul:1uncol2edge}.
\end{lemma}
\begin{proof}
	The edge added to $P_{ab}$ during the application of Path Rule~\ref{prul:1uncol2edge} to $(x_T,P_{ab})$
	has color $c(e')\in c(T)$.
	If the invariant is violated, then there was an edge $e''$ in $P_{ab}$ already with color $c(e')$.
	By Invariant~\ref{inv:internaledgespath} $e'$ was not already in $P_{ab}$ as the edge $e$ incident on $x_T$ is in $Q_{ab}$ and the representative of $e$ was not added to $P_{ab}$ before.
	Thus $e''\neq e'$ but $c(e'')=c(e')$.
	Since $x_T$ was incomplete before the application of current coloring rule, by Invariant~\ref{inv:internalcolor}, none of the colors in $c(T)$ were used before anywhere outside of $T$.
	So, such an $e''$ does not exist, a contradiction.
\end{proof}
\begin{lemma}
	\label{lem:rul1-9}
	Consider a 2-edge $e$ incident on tree vertex $x_T$ that remains uncolored after the application of Rules~\ref{rul:forest} to~\ref{rul:1uncol2edge}.
	Then, $x_T$ has degree exactly $2$ in $B_1$, $T$ contains only one edge, and the other edge incident on $x_T$ is an uncolored 2-edge. 
\end{lemma}
\begin{proof}
	By Lemma~\ref{lem:rul1-8} it follows that $x_T$ has	
degree exactly $2$ in $B_1$, $T$ contains only one edge, and the other edge incident on $x_T$ is a 2-edge. 
If this other 2-edge is colored, then Coloring Rule~\ref{rul:1uncol2edge} would have been applied on
$x_T$ and $x_T$ would have been completed.
\end{proof}
\begin{figure}[t]
    \centering
        \includegraphics[scale=1.0,keepaspectratio]{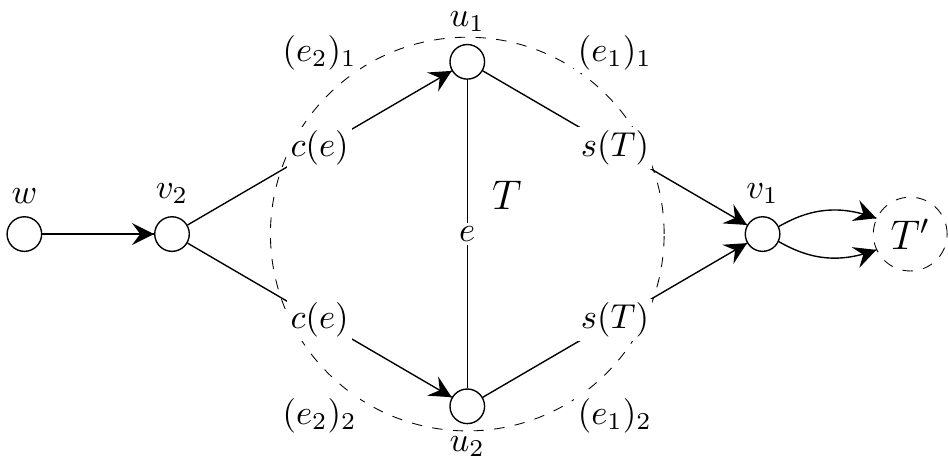}
		\caption{Coloring Rule~\ref{rul:treevertexdeg2par2edge}}
        \label{fig:treevertexdeg2par2edge}
\end{figure}

\begin{rul}
	\label{rul:treevertexdeg2par2edge}
	For each incomplete tree vertex $x_T$ whose parent's outgoing edge is a 2-edge: 
	(See Figure~\ref{fig:treevertexdeg2par2edge} for an Illustration).
	Let $v_1$ be the parent of $x_T$.
	From Lemma~\ref{lem:rul1-9}, it follows that $x_T$ has degree exactly $2$ in $B_1$, has one incoming and one outgoing 2-edge incident on it, both the 2-edges are uncolored,
	and the tree $T$ is just a single edge.
	Let $e_1$ and $e_2$ respectively be the outgoing and incoming 2-edges of $x_T$.
	Let $e$ be the only edge in $T$.
	Let $v_2$ be the other end point of $e_2$.
	Let $u_1$ be the endpoint of $(e_1)_1$ and $(e_2)_1$ in $T$.
	Let $u_2$ be the endpoint of $(e_1)_2$ and $(e_2)_2$ in $T$.
	From Lemma~\ref{lem:rul1-8}, we know that $v_1$ and $v_2$ have degree exactly $2$.
	Let $\dir{v_1x_{T'}}$ be the outgoing $2$-edge from $v_1$ and 
	let $\dir{wv_2}$ be the incoming edge on $v_2$ in $B_1$.\\
	Color $(e_2)_1$ and $(e_2)_2$ with the color of $e$,
	and color $(e_1)_1$ and $(e_1)_2$ with $s(T)$.
\end{rul}

\begin{prul}
	\label{prul:treevertexdeg2par2edge}
	For each tree vertex $x_T$ on which Coloring Rule~\ref{rul:treevertexdeg2par2edge} has been applied as above and for each  $P_{ab}$ such that $Q_{ab}$ contains $e_1$ or $e_2$
	(we say that the path rule is being applied on the pair $(x_T,P_{ab}))$, do the following.
	
	\begin{itemize}
	\item If $Q_{ab}$ contains both $e_1$ and $e_2$, add $v_1u_1$ and $u_1v_2$ to $P_{ab}$.
	
	\item If $Q_{ab}$ contains exactly one edge among $e_1$ and $e_2$, then either $a$ or $b$ is in $V(T)$.
	Also both of them cannot be in $V(T)$.
	Let $z$ be the one among $a$ or $b$ that is in $V(T)$.
	If $Q_{ab}$ contains $e_1$, add $v_1z$ to $P_{ab}$; otherwise, i.e.,
	if $Q_{ab}$ contains $e_2$, add $v_2z$ to $P_{ab}$.

	\end{itemize}
\end{prul}

\begin{lemma}
	\label{lem:invtreevertexdeg2par2edge}
Invariant \ref{inv:part-rainbow} is not violated during 
	Path Rule~\ref{prul:treevertexdeg2par2edge}.
\end{lemma}
\begin{proof}
	Suppose for the sake of contradiction that the invariant is violated.
Then there exist distinct edges $d_1$ and $d_2$ in $P_{ab}$ having the same color.
We can assume without loss of generality that $d_1$ was colored during the application of Coloring Rule~\ref{rul:treevertexdeg2par2edge} on $x_T$.
We added at most two edges during the application of Path Rule~\ref{prul:treevertexdeg2par2edge} on $x_T$, and in the cases where we added two edges, the two edges have distinct colors.
Thus, $d_2$ was not added during the application of Path Rule~\ref{prul:treevertexdeg2par2edge} on $x_T$ and hence was not colored during the application of Coloring Rule~\ref{rul:treevertexdeg2par2edge} on $x_T$.

	The colors that are possible for $d_1$ are $s(T)$ and $c(e)$.
	
	\noindent {\bf Case 1:}\ $c(d_1)=c(d_2)=c(e)$.
	
	This is not possible since the color of $e$ has not been used to color any other edges so far by Invariant~\ref{inv:internalcolor}, and $e$ is not in $P_{ab}$ by Invariant~\ref{inv:internaledgespath}.
	
	\noindent {\bf Case 2:}\ $c(d_1)=c(d_2)=s(T)$. 
	
	This means $h(d_1)=v_1x_T$. 
	The only coloring rules so far that use the surplus colors of trees are Coloring Rules \ref{rul:1edge}, \ref{rul:2edgedeg4}, \ref{rul:2edgedeg3}, \ref{rul:nontreedeg3}, \ref{rul:treevertexdeg3}, and \ref{rul:treevertexdeg2par2edge}.
	Hence, $d_2$ was colored with $s(T)$ during one of them.
	
	\noindent {\bf Case 2.1}\ $d_2$ was colored during Coloring Rule~\ref{rul:1edge}. 
	
	This means that $d_2$ is a $1$-edge.
	According to Coloring Rule~\ref{rul:1edge}, the only $1$-edge that can be colored with $s(T)$ is either the outgoing edge of $x_T$ or the outgoing edge of the parent of $x_T$.
	However, both of them are 2-edges and hence we have a contradiction.
	
	\noindent {\bf Case 2.2}\ $d_2$ was colored during Coloring Rules~\ref{rul:2edgedeg3} or \ref{rul:2edgedeg4}.
	
	Let $T''$ be such that $d_2$ is adjacent on $x_{T''}$.
	Then, by Lemmas~\ref{lem:2edgedeg3paths2} and \ref{lem:2edgedeg4paths2}, we know that $Q_{ab}$ does not intersect $\st{x_{T''}}{x_T}$, in particular $Q_{ab}$ does not contain $x_T$.
	Since $h(d_1)=v_1x_T$,
	this implies $P_{ab}$ does not contain $d_1$, which is a contradiction.
	
	\noindent {\bf Case 2.3}\ $d_2$ was colored during application of Coloring Rule~\ref{rul:nontreedeg3}. 
	
	From Coloring Rule~\ref{rul:nontreedeg3}, this implies that $d_2$ was colored during application of Coloring Rule~\ref{rul:nontreedeg3} on some ancestor $v'$ of $x_T$ such that there are no other tree vertices in the path from $x_T$ to $v'$. 
	Then, the only possibility for $v'$ is $v_1$ as the parent of $v_1$ is a tree vertex.
	However, we know that $v_1$ has degree $2$ in $B_1$, and hence Coloring Rule~\ref{rul:nontreedeg3} could not have been applied on $v_1$.
	Thus, we have a contradiction.
	
	\noindent {\bf Case 2.4}\ $d_2$ was colored during application of Coloring Rule~\ref{rul:treevertexdeg3}. 
	
	Since $d_2$ is colored with $s(T)$ during Coloring Rule~\ref{rul:treevertexdeg3}, Case 1 of the rule (see Coloring Rule~\ref{rul:treevertexdeg3}) was applied on $x_T$ and hence the outgoing edge from $x_T$ is a $1$-edge.
	%The only $1$-edge that can be colored with $s(T)$ is either the outgoing edge of $x_T$ or the outgoing edge of the parent of $x_T$.
	However, this is a 2-edge and hence we have a contradiction.
	
	\noindent {\bf Case 2.5}\ $d_2$ was colored during Coloring Rule~\ref{rul:treevertexdeg2par2edge}. 
	
	Since $d_2$ was not colored during the application of Coloring Rule~\ref{rul:treevertexdeg2par2edge} on $x_T$,
	we have that $d_2$ was colored during the application of Coloring Rule~\ref{rul:treevertexdeg2par2edge} on some $x_{T''}\neq x_{T}$.
	But then $d_2$ is not colored with $s(T)$, a contradiction.
\end{proof}

\begin{figure}[t]
    \centering
    \begin{subfigure}[b]{0.49\textwidth}
    \centering
        \includegraphics[scale=1,keepaspectratio]{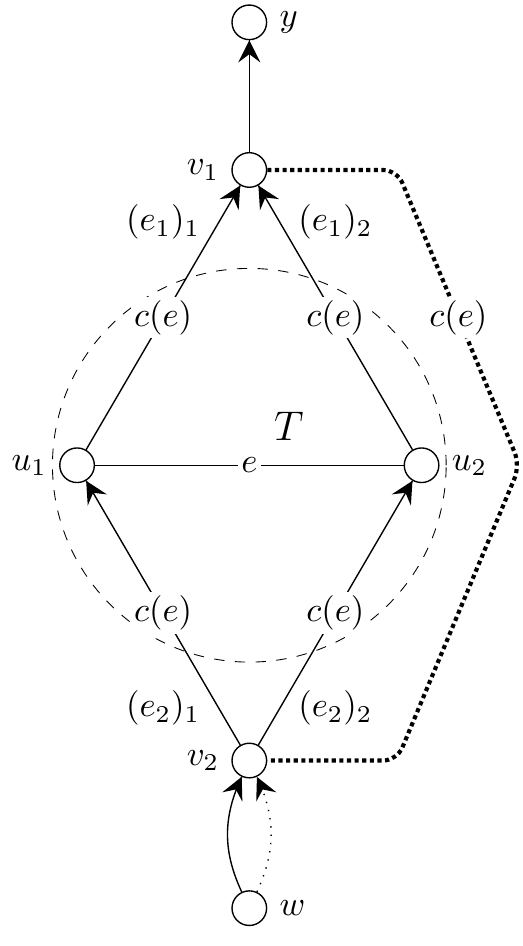}
        \caption{}
		%\caption{Coloring Rule~\ref{rul:final}: Case 1. The red color of edge $v_1v_2$ is to highlight that the edge is not in $B_1$.} 
        \label{fig:final-1}
    \end{subfigure}
    \hfill
    \begin{subfigure}[b]{0.49\textwidth}
    \centering
	\includegraphics[scale=1,keepaspectratio]{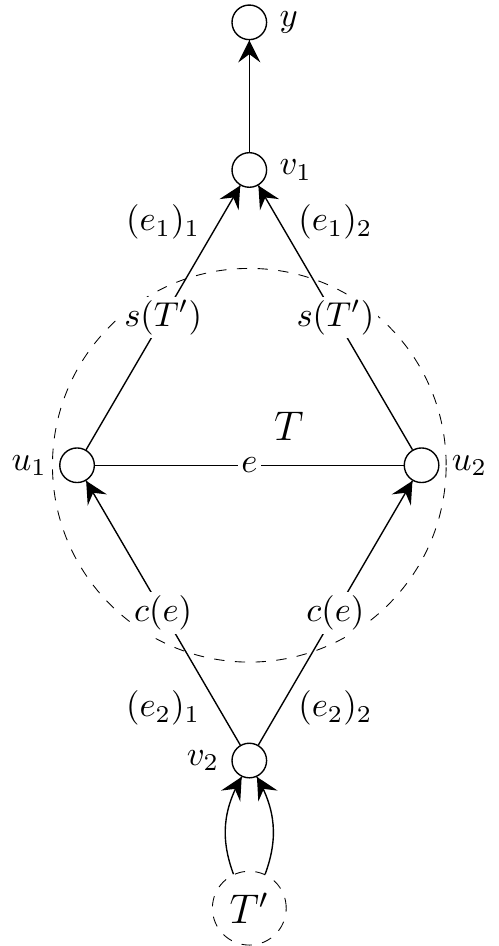}
        \caption{}
		%\caption{Coloring Rule~\ref{rul:final}: Case 2} 
        \label{fig:final-2}
    \end{subfigure}
	\caption{Coloring Rule~\ref{rul:final}. \textbf{(a)} Case~1; here the edge $v_1v_2$ is drawn as a thick dotted line to highlight that it is not in $B_1$, and the edge $wv_2$ is drawn with one solid line and one dotted line to denote that it could be a 1-edge or a 2-edge\textbf{(b)} Case~2; here $w=x_{T'}$}
  	\label{fig:final}
\end{figure}

\begin{rul}
	\label{rul:final}
	For each incomplete tree vertex $x_T$:
	from Lemma~\ref{lem:rul1-9}, it follows that $x_T$ has degree exactly $2$ in $B_1$, has one incoming and one outgoing 2-edge incident on it, both the 2-edges are uncolored,
	and the tree $T$ is just a single edge.
	Let $e_1$ and $e_2$ respectively be the outgoing and incoming 2-edges of $x_T$.
	Let $v_1$ be the other end point of $e_1$ and $v_2$ be the other end point of $e_2$.
	Let $e=u_1u_2$ be the only edge in $T$.
	Without loss of generality, $u_1$ be the endpoint of $(e_1)_1$ and $(e_2)_1$ in $T$,
	and $u_2$ be the endpoint of $(e_1)_2$ and $(e_2)_2$ in $T$.
	From Lemma~\ref{lem:rul1-8}, we know that $v_1$ and $v_2$ have degree exactly $2$.
	Let $\dir{v_1y}$ be the outgoing edge from $v_1$ and
	$\dir{wv_2}$ be the incoming edge on $v_2$ in $B_1$.
	We have that $v_1y$ is a $1$-edge because otherwise
	Coloring Rule~\ref{rul:treevertexdeg2par2edge} would have been applicable on $x_T$, and $x_T$ would have been already completed.
	
	\noindent {\bf Case 1:}\ There is an edge between $v_1$ and $v_2$ in $G$. (See Figure~\ref{fig:final}~(a) for an illustration).
	
	Color the representatives of $e_1$ and $e_2$ with $c(e)$.
	Color $v_1v_2$ with $c(e)$.
	We say that $v_1v_2$ is a {\bf shortcut edge}.
	Note that shortcut edges are the only colored edges in $G$ that are not representatives of edges in $B$. 
	%\todo{what is $g$? is $h$ meant?} 
	% g is the inverse of h
	
	\noindent {\bf Case 2:}\ Case 1 does not apply. (See Figure~\ref{fig:final}~(b) for an illustration).
	
	We will prove in Lemma~\ref{lem:v2incoming} that $wv_2$ is a 2-edge in this case.
	Let $T'=f_{\calT}(w)$. 
	Color $(e_1)_1$ and $(e_1)_2$ with $s(T')$
	and color $(e_2)_1$ and $(e_2)_2$ with color of $e$.
\end{rul}

\begin{prul}
	\label{prul:final}
	For each tree vertex $x_T$ on which Coloring Rule~\ref{rul:final} has been applied as above and for each  $P_{ab}$ such that $Q_{ab}$ contains $e_1$ or $e_2$
	(we say that the path rule is being applied on the pair $(x_T,P_{ab}))$, do the following.
	
	\begin{itemize}
	\item If $Q_{ab}$ contains both $e_1$ and $e_2$: 
	if $v_1v_2\in E(G)$, add $v_1v_2$ to $P_{ab}$;
	otherwise add $v_1u_1$ and $v_2u_1$ to $P_{ab}$.
	
	\item If $Q_{ab}$ contains exactly one edge among $e_1$ and $e_2$, then either $a$ or $b$ is in $V(T)$.
	Also, both of them cannot be in $V(T)$.
	Let $z$ be the one among $a$ or $b$ that is in $V(T)$.
	If $Q_{ab}$ contains $e_1$, add $v_1z$ to $P_{ab}$.
	If $Q_{ab}$ contains $e_2$, add $v_2z$ to $P_{ab}$.
	\end{itemize}
\end{prul}

\begin{lemma}
	\label{lem:invfinal}
Invariant \ref{inv:part-rainbow} is not violated during 
	Path Rule~\ref{prul:final}.
\end{lemma}
\begin{proof}
	Suppose for the sake of contradiction that the invariant is violated.
Then there exist distinct edges $d_1$ and $d_2$ in $P_{ab}$ having the same color.
We can assume without loss of generality that $d_1$ was colored during the application of Coloring Rule~\ref{rul:final} on $x_T$.
We added at most two edges during the application of Path Rule~\ref{prul:final} on $x_T$, and in the cases where we added two edges, the two edges have distinct colors.
Thus $d_2$ was not added during the application of Path Rule~\ref{prul:final} on $x_T$ and hence was not colored during Coloring Rule~\ref{rul:final} on $x_T$.

	The colors that are possible for $d_1$ are $c(e)$ and $s(T')$.
	
	\noindent {\bf Case 1:}\ $c(d_1)=c(d_2)=c(e)$.
	
	This is not possible since the color of $e$ has not been used to color any other edges so far by Invariant~\ref{inv:internalcolor}, and $e$ is not in $P_{ab}$ by Invariant~\ref{inv:internaledgespath}.
	
	\noindent {\bf Case 2:}\ $c(d_1)=c(d_2)=s(T')$.
	
	This means $h(d_1)=e_1$ and that Case 2 of Coloring Rule~\ref{rul:final} was applied on $x_T$.
	The only coloring rules so far that use the surplus colors of trees are Coloring Rules \ref{rul:1edge}, \ref{rul:2edgedeg4}, \ref{rul:2edgedeg3}, \ref{rul:nontreedeg3}, \ref{rul:treevertexdeg3}, \ref{rul:treevertexdeg2par2edge}, and \ref{rul:final}.
	Hence, $d_2$ was colored with $s(T')$ during one of them.
	
	\noindent {\bf Case 2.1}\ $d_2$ was colored during Coloring Rule~\ref{rul:1edge}.
	
	This means $d_2$ is a $1$-edge and $h(d_2)$ is either $x_{T'}v_2$ or $v_2x_T$.
	But since both $x_{T'}v_2$ and $v_2x_T$ are 2-edges (since Case 2 of Coloring Rule~\ref{rul:final} was applied on $x_T$),
	this is not possible.
	
	\noindent {\bf Case 2.2}\ $d_2$ was colored during Coloring Rules~\ref{rul:2edgedeg3} or \ref{rul:2edgedeg4}.
	
	Let $T''$ be the tree such that $d_2$ is adjacent on $x_{T''}$.
	By Lemmas~\ref{lem:2edgedeg3paths2} and \ref{lem:2edgedeg4paths2},
	we know that $Q_{ab}$ does not intersect $\st{x_{T''}}{x_{T'}}$.
	Then $x_T$ is not in $\st{x_{T''}}{x_{T'}}$.
	This implies $x_{T''}$ is in the path from $x_T$ to $x_{T'}$.
	But the only vertex in the path from $x_T$ to $x_{T'}$ is $v_2$, a non-tree vertex.
	Thus, we have a contradiction.
	
	\noindent {\bf Case 2.3}\ $d_2$ was colored during application of Coloring Rule~\ref{rul:nontreedeg3}. 
	
	From Coloring Rule~\ref{rul:nontreedeg3}, this implies that $d_2$ was colored during application of Coloring Rule~\ref{rul:nontreedeg3} on some ancestor $v'$ of $x_{T'}$ such that there are no other tree vertices in the path from $x_{T'}$ to $v'$. 
	Then, the only possibility for $v'$ is $v_2$ as the parent of $v_2$ is a tree vertex.
	However, we know that $v_2$ has degree $2$ in $B_1$, and hence Coloring Rule~\ref{rul:nontreedeg3} could not have been applied on $v_2$.
	Thus, we have a contradiction.
	
	\noindent {\bf Case 2.4}\ $d_2$ was colored during application of Coloring Rule~\ref{rul:treevertexdeg3}.
	
	Since $d_2$ is colored with $s(T')$ during Coloring Rule~\ref{rul:treevertexdeg3}, Case 1 of the rule (see Coloring Rule~\ref{rul:treevertexdeg3}) was applied on $x_{T'}$ and hence the outgoing edge from $x_{T'}$ is a $1$-edge.
	%The only $1$-edge that can be colored with $s(T)$ is either the outgoing edge of $x_T$ or the outgoing edge of the parent of $x_T$.
	However, this is a 2-edge and hence we have a contradiction.
	
%We consider the two following sub-cases.
%
%\noindent {\bf Case 2.4.1}\ $d_2$ was colored during application of Coloring Rule~\ref{rul:treevertexdeg3}~Case~1.
%
%Then,  
%$d_2$ was colored during application of Coloring Rule~\ref{rul:treevertexdeg3} on $x_{T'}$.
%But the outgoing edge of $x_{T'}$ in $B_1$ is a $1$-edge.
%This means that Case 1 of Coloring Rule\ref{rul:treevertexdeg3} could not have been applied on $x_{T'}$, a contradiction.
%
%\noindent {\bf Case 2.4.2}\ $d_2$ was colored during application of Coloring Rule~\ref{rul:treevertexdeg3}~Case~2.
%
%Then, there is a $1$-edge having the color $s(T')$.
%But there are no $1$-edges having the color $s(T')$, as for $x_{T'}$, both the outgoing edge and the outgoing edge of the parent are 2-edges.
%Hence, we have a contradiction.
	
	\noindent {\bf Case 2.5}\ $d_2$ was colored during application of Coloring Rule~\ref{rul:treevertexdeg2par2edge}.
	
	Since $c(d_2)=s(T')$, from Coloring Rule~\ref{rul:treevertexdeg2par2edge} we get that $d_2$ was colored during the application of Coloring Rule~\ref{rul:treevertexdeg2par2edge} on $x_{T'}$.
	In Lemma~\ref{lem:Tdashdegree3}, we will prove that Coloring Rule~\ref{rul:treevertexdeg2par2edge} was not applied on $x_{T'}$.
	Thus, we have a contradiction.
	
	\noindent {\bf Case 2.6}\ $d_2$ was colored during application of Coloring Rule~\ref{rul:final}.
	
	Since $c(d_2)=s(T')$, from Coloring Rule~\ref{rul:final}, we get that $d_2$ was colored during the application of Coloring Rule~\ref{rul:final} on $x_{T}$.
	Moreover, $h(d_2)=e_1$. 
	Recall that we have $h(d_1)=e_1$ too.
	Since we picked only one representative of $e_1$ into $P_{ab}$ by Path Rule~\ref{prul:final},
	we have that $d_1=d_2$.
	This is a contradiction to the fact that $d_1$ and $d_2$ are distinct.
\end{proof}

Now we proceed towards proving Lemmas \ref{lem:v2incoming} and \ref{lem:Tdashdegree3} that were used above.
For this we need to prove some auxiliary lemmas first.
\begin{figure}[t]
    \centering
    \begin{subfigure}[b]{0.49\textwidth}
    \centering
        \includegraphics[scale=1,keepaspectratio]{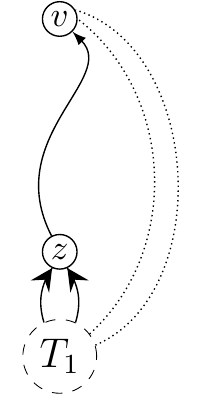}
        \caption{}
		%\caption{Coloring Rule~\ref{rul:final}: Case 1. The red color of edge $v_1v_2$ is to highlight that the edge is not in $B_1$.} 
        \label{fig:descendant-1}
    \end{subfigure}
    \hfill
    \begin{subfigure}[b]{0.49\textwidth}
    \centering
        \includegraphics[scale=1,keepaspectratio]{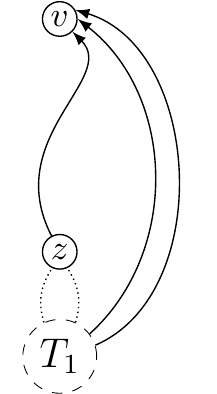}
        \caption{}
		%\caption{Coloring Rule~\ref{rul:final}: Case 2} 
        \label{fig:descendant-2}
    \end{subfigure}
	\caption{An illustration of the proof of Lemma~\ref{lem:descendant}. The densely dotted edges denote the edges of $H$ that are not in $B$. \textbf{(a)} The scenario given by the precondition of the lemma, and \textbf{(b)} the transformation to new skeleton $B'$ as described in the proof.}
  	\label{fig:descendant}
\end{figure}
\begin{lemma}
	\label{lem:descendant}
	Let $v$ be a non-tree vertex and $x_{T_1}$ be a tree-vertex that is a descendant of $v$ in $B_1$.	 
	If $vx_{T_1}$ is a $2$-edge in $H$ then $x_{T_1}$ is a child of $v$ in $B_1$.
\end{lemma}
\begin{proof}
	See Figure~\ref{fig:descendant} for an illustration of the proof.
	Suppose $x_{T_1}$ is not a child of $v$ in $B_1$.
	%Then there is some child $v'$ of $v$ such that $x_{T_1}\in \st{v}{v'}$.
	Let $\dir{x_{T_1}z}$ be the outgoing edge of $x_{T_1}$ in $B_1$.
	Let $B'$ be the \emph{skeleton} obtained by deleting  $\dir{x_{T_1}z}$ from $B$ and adding $\dir{x_{T_1}v}$.
	Going from $B$ to $B'$, the number of 2-edges is non-decreasing, the degree of $v$ increases, the degree of $z$ decreases, and the degree of all other
	vertices remains same.
	Since $v$ is at a smaller level than $z$ in both $B$ and $B'$, 
	$B'$ has lexicographically higher configuration vector than $B$.
	Thus we have a contradiction to the choice of $B$.
\end{proof}
\begin{figure}[t]
    \centering
    \begin{subfigure}[b]{0.49\textwidth}
    \centering
        \includegraphics[scale=1,keepaspectratio]{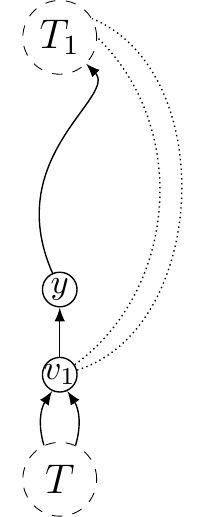}
        \caption{}
        \label{fig:ancestor-2edge-1}
    \end{subfigure}
    \hfill
    \begin{subfigure}[b]{0.49\textwidth}
    \centering
        \includegraphics[scale=1,keepaspectratio]{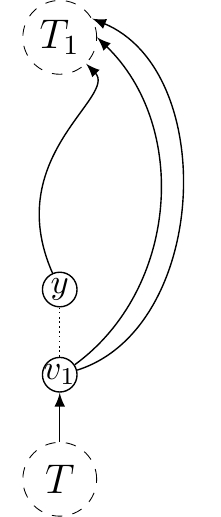}
        \caption{}
		%\caption{Coloring Rule~\ref{rul:final}: Case 2} 
        \label{fig:ancestor-2edge-2}
    \end{subfigure}
	\caption{An illustration of the transformation in the proof of Lemma~\ref{lem:v1noedge}. The densely dotted edges denote the edges of $H$ that are not in $B$. \textbf{(a)} The scenario given by the precondition of the lemma, and \textbf{(b)} the transformation to the new skeleton $B'$ as described in the proof.}
  	\label{fig:ancestor-2edge}
\end{figure}
\begin{lemma}
	\label{lem:v1noedge}
	Let $x_T$ be a vertex on which Coloring Rule~\ref{rul:final} is being applied. 
	Let $v_1$ be as defined in Coloring Rule~\ref{rul:final}.
	The vertex $v_1$ has no 2-edge in $H$ to any vertex except $x_T$.
	%Similarly, $v_2$ has no 2-edge in $H$ to any vertex in $V(H)\setminus \left\{ w,x_T \right\}$.
\end{lemma}
\begin{proof}
	See Figure~\ref{fig:ancestor-2edge} for an illustration of the proof.
	Suppose for the sake of contradiction that $v_1$ has a 2-edge in $H$ to a tree vertex $x_{T_1}\in V(H)\setminus \{x_T\}$.
	The edge $v_1y$ is a $1$-edge as otherwise Coloring Rule~\ref{rul:treevertexdeg2par2edge} would have been applicable on $x_T$ and $x_T$ would have been completed already.
	Thus $x_{T_1}\neq y$. Then the edge $v_1x_{T_1}$ is not in $B_1$ as $y$ and $x_T$ are the only neighbors of $v_1$ in $B_1$. 
	Since $x_{T_1}$ is a tree vertex, it is not in $L_S$ (recall that $L_S$ is the set of non-tree vertices of $H$). 
	Thus, since the edge $v_1x_{T_1}$ is not in $B_1$, it is not in $B$ also (recall $B_1=B\setminus L_S$).
Thus, $v_1x_{T_1}\in E(H)\setminus E(B)$.
Also $x_{T_1}$ is not a descendant of $v_1$ due to Lemma~\ref{lem:descendant}.
Thus, $x_{T_1}\in \st{v_1}{y}$.
	Then, by deleting  the 1-edge $\dir{v_1y}$ from $B$ and adding the 2-edge $\dir{v_1x_{T_1}}$,
	we get a skeleton $B'$ that has a higher number of $2$-edges than $B$ and hence has a lexicographically higher configuration vector.
	Thus we have a contradiction to the choice of $B$.
\end{proof}
\begin{figure}[t]
    \centering
    \begin{subfigure}{0.49\textwidth}
    \centering
        \includegraphics[scale=1,keepaspectratio]{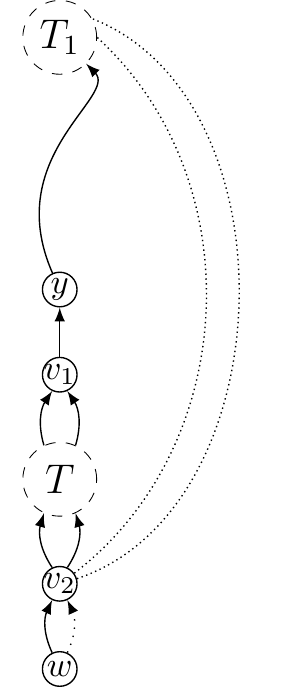}
        \caption{}
        \label{fig:ancestor-v2-1}
    \end{subfigure}
    \hfill
    \begin{subfigure}{0.49\textwidth}
    \centering
        \includegraphics[scale=1,keepaspectratio]{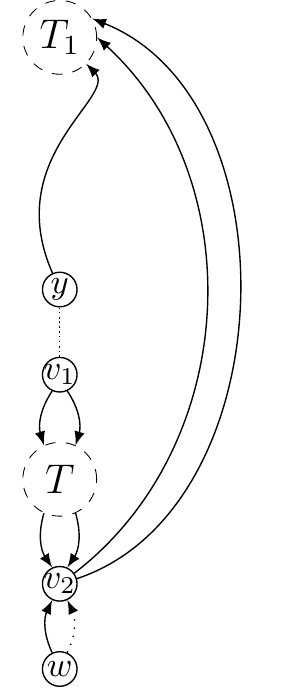}
        \caption{}
		%\caption{Coloring Rule~\ref{rul:final}: Case 2} 
        \label{fig:ancestor-v2-2}
    \end{subfigure}
	\caption{An illustration of the transformation in the proof of Lemma~\ref{lem:v2noedge}. The densely dotted edges denote the edges of $H$ that are not in $B$ and the edge $wv_2$ is drawn with 1 solid line and 1 dotted line to denote that it could be a 1-edge or a 2-edge.
	\textbf{(a)} The scenario given by the precondition of the lemma, and \textbf{(b)} the transformation to the new skeleton $B'$ as described in the proof.}
  	\label{fig:ancestor-v2}
\end{figure}
\begin{lemma}
	\label{lem:v2noedge}
	Let $x_T$ be a vertex on which Coloring Rule~\ref{rul:final} is being applied. 
	Let $v_2,w$ be as defined in Coloring Rule~\ref{rul:final}.
	Then, $v_2$ has no 2-edge in $H$ to any vertex in $V(H)\setminus \left\{ w,x_T \right\}$.
\end{lemma}
\begin{proof}
	See Figure~\ref{fig:ancestor-v2} for an illustration of the proof.
	Suppose for the sake of contradiction that $v_2$ has a 2-edge in $H$ to a tree vertex $x_{T_1}\in V(H)\setminus \{x_T,w\}$.
	Then the edge $v_2x_{T_1}$ is not in $B_1$ as the only neighbors of $v_2$ in $B_1$ are $x_T$ and $w$.
	Since $x_{T_1}$ is a tree vertex, it is not in $L_S$ (recall that $L_S$ is the set of non-tree vertices of $H$). 
	Thus, since the edge $v_2x_{T_1}$ is not in $B_1$, it is not in $B$ also (recall $B_1=B\setminus L_S$).
Thus, $v_2x_{T_1}\in E(H)\setminus E(B_1)$.
Also $x_{T_1}$ is not a descendant of $v_2$ due to Lemma~\ref{lem:descendant}.
Clearly, then $x_{T_1}\in \st{v_2}{x_T}$. 
	Since $x_{T_1}\neq x_T$, and the degree of $x_T$ and $v_1$ in $B_1$ is $2$, we have that the edge $v_1y$ is on the path from $v_2$ to $x_{T_1}$ in $B_1$.
	Then, by deleting the $1$-edge $v_1y$ from $B$ and adding the $2$-edge $v_2x_{T_1}$, we get a 
	skeleton $B'$ that has a higher number of $2$-edges and hence has a lexicographically higher configuration vector.
	Thus we have a contradiction to the choice of $B$.
\end{proof}
\begin{figure}
    \centering
    \begin{subfigure}[b]{0.49\textwidth}
    \centering
        \includegraphics[scale=1,keepaspectratio]{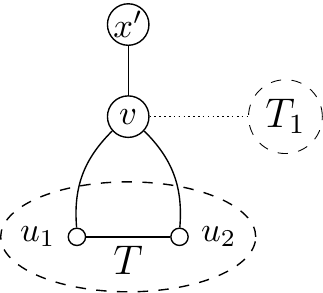}
        \caption{}
        \label{fig:tree-change-1}
    \end{subfigure}
    \hfill
    \begin{subfigure}[b]{0.49\textwidth}
    \centering
        \includegraphics[scale=1,keepaspectratio]{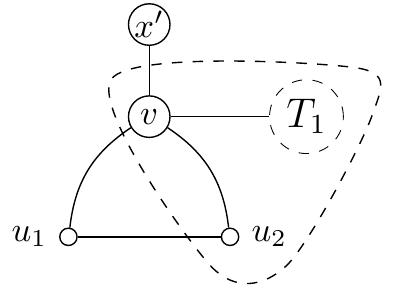}
        \caption{}
		%\caption{Coloring Rule~\ref{rul:final}: Case 2} 
        \label{fig:tree-change-2}
    \end{subfigure}
	\caption{An illustration of the transformation in the proof of Lemma~\ref{lem:noedgetotree}. The densely dotted edges denote the edges of $H$ that are not in $B$.
	\textbf{(a)} The scenario given by the precondition of the lemma, and \textbf{(b)} the transformation to the new forest $\calF'$ as described in the proof where $T$ is removed and a new tree including $v,u_2$ and the vertices of $T_1$ are added.}
  	\label{fig:tree-change}
\end{figure}
\begin{lemma}
	\label{lem:noedgetotree}
	Let $v$ be a non-tree vertex having degree $2$ in $B_1$ such that the neighbors of $v$ in $B_1$ are a tree-vertex $x_T$ and a vertex $x'$, and the tree $T$ consists of a single edge $u_1u_2$. If $v_1$ does not have a $2$-edge to any tree-vertex except $x_T$ in $H$, then $v_1$ does not have edges to any tree vertex in $H$ except $x_T$.
\end{lemma}
\begin{proof}
	See Figure \ref{fig:tree-change} for an illustration.
	Suppose $v$ has an edge in $H$ to a tree vertex $x_{T_1}\neq x_T$.
	Note that $vx_{T_1}$ is not a $2$-edge by assumption. 
	Thus $vx_{T_1}$ is a $1$-edge. 
	We define the forest $\calF'$ as $F':=(F\setminus \left\{ u_1 \right\})\cup\{v\}$ and $\calF':=G[F']$.
	We prove that $\calF'$ is a forest with fewer trees than 
	$\calF$, which is a contradiction to the choice of $\calF$.
	(Recall that out of all maximum induced forests, we picked $\calF$ to be one having the fewest number of trees).

	First, we prove that $\calF'$ is indeed a forest. 
	Suppose for the sake of contradiction that there is a cycle $C$ in $\calF'$.
	The cycle $C$ has to contain $v$ because otherwise $C$ is also a cycle in $\calF$. 
	The cycle $C$ can intersect at most one tree in $\calT$ as there are no edges across the trees.
	Let this tree be $T_2$.
	Then $v$ should have two edges to $T_2$ in $C$ in order to complete the cycle.
	We know that $v$ does not have 2-edges to any tree in $\calT\setminus \{T\}$ by assumption.
	Hence $T_2=T$.
	But, since $|V(T)\cap F'|=1$, $v$ can only have one edge in $C$ to $T$. 
	Thus, we have a contradiction. 
	Hence, $\calF'$ is indeed a forest.

	Now, we show that the number of trees in $\calF'$ is smaller than that of $\calF$.
	Let $\calT_1$ be the set of trees in $\calT\setminus \{T\}$ that have an edge from $v$ in $G$.
	Clearly, $T_1\in\calT_1$ and hence $|\calT_1|\ge 1$.
	The vertex set $\left\{ u_2,v \right\}\cup\bigcup_{T''\in \calT_1}V(T'')$ induces a tree in $G$.
	Hence, the number of trees in $\calF'$ is at most $|\calT\setminus (\calT_1\cup \left\{ T \right\})|+1\le |\calT| -|\calT_1|\le |\calT|-1$.
\end{proof}
\begin{lemma}
	\label{lem:v1noedgetotree}
	Let $x_T$ be a vertex on which Coloring Rule~\ref{rul:final} is being applied. 
	Let $v_1,v_2,y,w$ be as defined in Coloring Rule~\ref{rul:final}.
	\begin{enumerate}
			\item
	$v_1$ has no edge in $H$ to any tree vertex except $x_T$ (which also implies that $y$ is a non-tree vertex), and
	\item
	if $wv_2$ is a $1$-edge, vertex $v_2$ has no edge in $H$ to any tree vertex except $x_T$ (which also implies that $w$ is a non-tree vertex in this case).
	\end{enumerate}
\end{lemma}
\begin{proof}
	The first statement follows from Lemmas~\ref{lem:v1noedge}, and \ref{lem:noedgetotree} and the fact that $v_1y$ is a $1$-edge.
	The second statement follows from Lemmas \ref{lem:v2noedge} and \ref{lem:noedgetotree}.
\end{proof}
\begin{lemma}
	\label{lem:v2incoming}
	Let $x_T$ be a vertex on which Coloring Rule~\ref{rul:final} is being applied and suppose the precondition of Case $1$ of the rule is \emph{not} satisfied.
	Let $v_2$ and $w$ be as defined in Coloring Rule~\ref{rul:final}.
	Then, $wv_2$ is a 2-edge.
\end{lemma}
\begin{proof}
	Suppose for the sake of contradiction that $wv_2 $ is a $1$-edge.
	Let $v_1,u_1,u_2$ be also as given in Coloring Rule~\ref{rul:final} (see Figure~\ref{fig:final}).
	Let $F'=\left( F\setminus\left\{ u_1 \right\}\cup\left\{ v_1,v_2 \right\} \right)$.
	Let $\calF'=G[F']$.
	We will show that $\calF'$ is a forest. 
	Then since $|F'|>|F|$, we have that $\calF$ is not a maximum induced forest, a contradiction. 

	Suppose for the sake of contradiction that there is a cycle $C$ in $\calF'$.
%	The cycle $C$ has to contain at least one of $v_1$ and $v_2$, because otherwise $C$ is also a cycle in $\calF$. 

	\noindent {\bf Case 1:}\ $C$ intersects more than two trees of $\calT$.
	
	 Since there are only two vertices in $F'$ that are not in any tree in $\calT$, there have to be two vertices in $2$ different trees of $\calT$ that are adjacent in $C$.
	This is a contradiction as the trees are connected components of $\calF$ and therefore have no edges between them in $G$. 

	\noindent {\bf Case 2:}\ $C$ intersects exactly two trees of $\calT$.
	
	Let $T_1$ and  $T_2$ be the trees that $C$ intersects.
	Since $wv_2$ is a $1$-edge, by Lemma~\ref{lem:v1noedgetotree}, we have that both $v_1$ and $v_2$ are not adjacent in $H$ to any tree vertex except $x_T$.
	Hence, at least one of $T_1$ and $T_2$ has no edges to $v_1$ and $v_2$.
	This means that there should be an edge in $C$ between $T_1$ and $T_2$.
	This is a contradiction as $T_1$ and $T_2$ are connected components of $\calF$ and therefore have no edges between them in $G$. 

	\noindent {\bf Case 3:}\ $C$ intersects exactly one tree of $\calT$.
	
	Let $T_1$ be the tree that $C$ intersects.
	By Lemma~\ref{lem:v1noedgetotree}, we have that $v_1$ and $v_2$ do not have an edge to any tree vertex in $B_1$ except $x_T$.
	Hence, $T_1=T$.
	Since only one vertex of $T$ is in $F'$, $v_1$ and $v_2$ both can have at most one edge to $T$ in $C$.
	This means $v_1$ and $v_2$ should be adjacent in $C$ to complete the cycle.
	Thus, the precondition of Case 1 of Rule~\ref{rul:final} is satisfied,
	a contradiction to our assumption.
\end{proof}

\begin{figure}
    \centering
    \begin{subfigure}[b]{0.49\textwidth}
    \centering
        \includegraphics[scale=1,keepaspectratio]{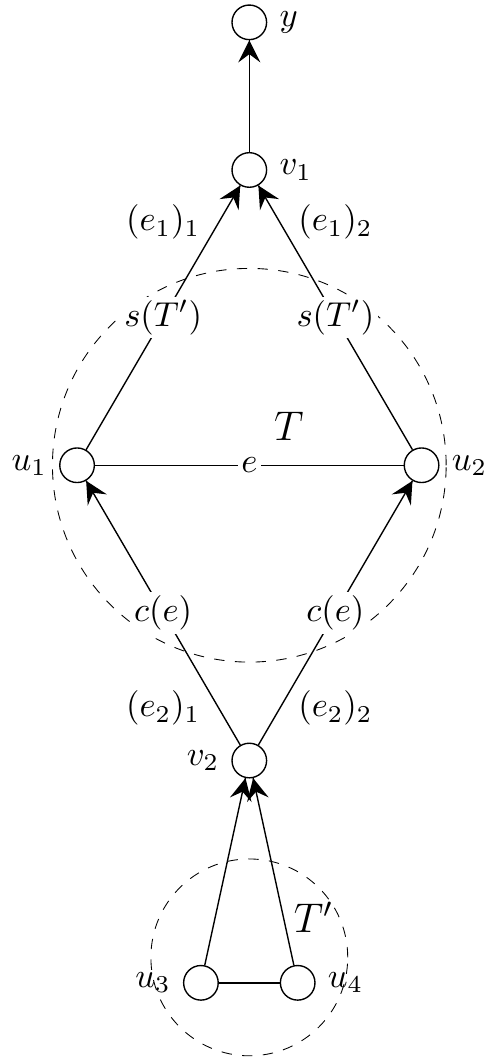}
        \caption{}
        \label{fig:forest-selection-1}
    \end{subfigure}
    \hfill
    \begin{subfigure}[b]{0.49\textwidth}
    \centering
        \includegraphics[scale=1,keepaspectratio]{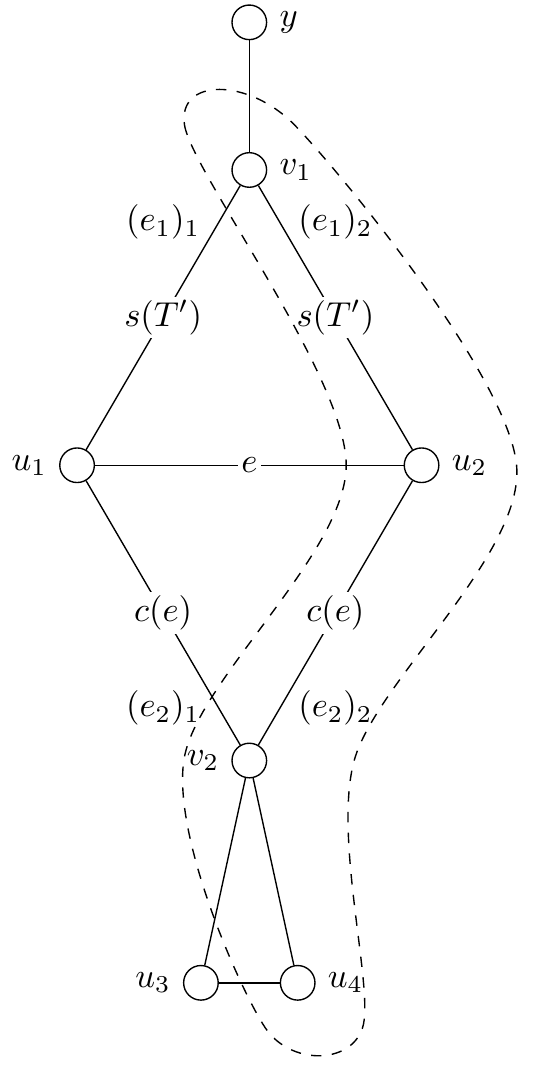}
        \caption{}
		%\caption{Coloring Rule~\ref{rul:final}: Case 2} 
        \label{fig:forest-selection-2}
    \end{subfigure}
	\caption{An illustration of the transformation in the proof of Lemma~\ref{lem:Tdashdegree3}. 
	\textbf{(a)} The initial scenario before transformation, and \textbf{(b)} the transformation to the new forest $\calF'$ as described in the proof where $T$ and $T'$ are removed and a new tree including $u_4,u_2,v_1,v_2$ is added, thereby reducing the number of trees.}
  	\label{fig:forest-selection}
\end{figure}
\begin{lemma}
	\label{lem:Tdashdegree3}
	Let $x_T$ be a vertex on which Coloring Rule~\ref{rul:final} is being applied and suppose the precondition of Case $1$ of the rule is not satisfied.
	Let $T'$ be as defined in Coloring Rule~\ref{rul:final} (see Figure~\ref{fig:final-2}).
	Then,  $x_{T'}$ is not a vertex on which Coloring Rule~\ref{rul:treevertexdeg2par2edge} was applied.
\end{lemma}
\begin{proof}
	Suppose for the sake of contradiction that Coloring Rule~\ref{rul:treevertexdeg2par2edge} was applied on $x_{T'}$. 
	Then $x_{T'}$ has degree $2$ in $B_1$ and $T'$ consists of a single edge.	
	Let this edge be $u_3u_4$ (see Figure~\ref{fig:forest-selection-1}).
	Observe that the representatives of the outgoing edge of $x_{T'}$ are $u_3v_2$ and $u_4v_2$.
	Let $F':=\left( F\setminus\left\{ u_1,u_3 \right\}\cup\left\{ v_1,v_2 \right\} \right)$ and
	$\calF':=G[F']$.
	Note that $|F'|=|F|$.
	We will show that $\calF'$ is a forest with fewer trees than $\calF$, thereby showing a contradiction to the choice of $\calF$.
	(Recall that out of all maximum induced forests, we picked $\calF$ to be one having the fewest number of trees).

	First, we show that $\calF'$ is indeed a forest.
	Suppose for the sake of contradiction that there is a cycle $C$ in $\calF'$.

	\noindent {\bf Case 1:}\ $C$ intersects more than two trees of $\calT$.
	
	Since there are only two vertices in $F'$ that are not in any tree in $\calT$, there have to be two vertices in two different trees of $\calT$ that are adjacent in $C$.
	This is a contradiction as the trees are connected components of $\calF$ and therefore have no edges between them in $G$. 

	\noindent {\bf Case 2:}\ $C$ intersects exactly two trees of $\calT$.
	
	Let $T_1$ and  $T_2$ be the trees that $C$ intersects.
	Then in order to complete the cycle $C$, there should be at least one edge from each of $v_1$ and $v_2$ to each of the trees $T_1$ and $T_2$.
	This implies $v_1$ has edges to each of $x_{T_1}$ and $x_{T_2}$ in $H$.
	But, by Lemma~\ref{lem:v1noedgetotree}, we have that $v_1$ does not have an edge to any tree vertex in $H$ except $x_T$.
	Therefore, $x_{T_1}=x_{T_2}$, implying that $T_1=T_2$, implying that $C$ intersects only one tree, a contradiction.

	\noindent {\bf Case 3:}\ $C$ intersects exactly one tree of $\calT$. 
	
	Let $T_1$ be the tree that $C$ intersects.
	Since the precondition of Case 1 of Rule~\ref{rul:final} is not satisfied, $v_1$ and $v_2$ are not adjacent in $G$.
	Hence, only one of them is in $C$.

	\noindent {\bf Case 3.1}\ $v_1$ is in $C$.
	
	Then, $v_1$ should have two edges in $C$ to $T_1$ in order to complete the cycle. 
	This means that $v_1$ has a $2$-edge to $x_{T_1}$ in $H$.
	By Lemma~\ref{lem:v1noedgetotree}, we have that $v_1$ does not have an edge in $H$ to any tree vertex except $x_T$.
	Hence, $T_1=T$.
	Since only one vertex of $T$ is in $F'$, $v_1$ can have at most one edge to $T$ in $C$.
	Thus we have a contradiction.

	\noindent {\bf Case 3.2}\ $v_2$ is in $C$.
	
	Then, $v_2$ should have two edges in $C$ to $T_1$ in order to complete the cycle. 
	This means that $v_2$ has a $2$-edge to $x_{T_1}$ in $H$.
	By Lemma~\ref{lem:v2noedge}, we have that $v_2$ does not have an edge in $H$ to any tree vertex in $B_1$ except $x_T$ and $x_{T'}$.
	Hence, $T_1\in \{T,T'\}$.
	However, since only one vertex of each $T$ and $T'$ is in $F'$, $v_1$ can have at most one edge to each of $T$ and $T'$ in $C$.
	Hence, $v_2$ has at most one edge to $T_1$ in $G$.
	Thus, we have a contradiction.

	Thus, we have proved that $\calF'$ is indeed a forest.
	Now, we prove that $\calF'$ has fewer trees than $\calF$, which concludes the proof.
	Let $\calT_1$ be the set of trees in $\calT\setminus \{T,T'\}$ that have an edge from $v_1$ or $v_2$ in $G$.
	The vertex set $\left\{ v_1,u_2,v_2,u_4 \right\}\cup\bigcup_{T''\in \calT_1}V(T'')$ induces a tree in $G$.
	Hence, the number of trees in $\calF'$ is at most $|\calT\setminus (\calT_1\cup \left\{ T,T'\right\})|+1\le  |\calT|-1$.
\end{proof}

%%Now, in order to complete the $P_{ab}$'s into a path we might need to add the missing paths in some trees.
%\begin{prul}
%	\label{prul:treepaths}
%	For each tree vertex $x_T$ and $P_{ab}$ such that $Q_{ab}$ contains $x_T$:\\
%	If $a\in V(T) $ and there is a neighbor $u$ of $x_T$ in $Q_{ab}$: 
%	let $w$ be the foot of edge $ux_T$ in $T$; 
%	add the path in $T$ between $w$ and $a$ to $P_{ab}$, if it is not already added.\\
%	If $b\in V(T) $ and there is a neighbor $u$ of $x_T$ in $Q_{ab}$: 
%	let $w$ be the foot of edge $ux_T$ in $T$, 
%	add the path in $T$ between $w$ and $b$ to $P_{ab}$, if it is not already added.\\
%	If there are $2$ neighbors $u$ and $v$ of $x_T$ in $Q_{ab}$: 
%	let $w$ be the foot of edge $ux_T$ in $T$ and
%	$z$ be the foot of edge $vx_T$ in $T$;
%	add the path in $T$ between $w$ and $z$ to $P_{ab}$, if it is not already added.
%\end{prul}	
%\begin{lemma}
%	\label{lem:invfinal}
%Invariant \ref{inv:part-rainbow} is not violated during 
%	Path Rule~\ref{prul:treepaths} to any vertex $x_T$ as above.
%\end{lemma}
%\begin{proof}
%	Observe that only Path Rules that could have been applied to $x_T$ are Path Rules~\ref{prul:1edge} and \ref{prul:}.
%	Note that before applying the rule, there was at least one edge adjacent on $T$ in $P_{ab}$.
%	Let $e$ be the last added edge to $P_{ab}$ among such edges. 
%\end{proof}
%
Thus we have proved the Lemmas that we used in Coloring Rule~\ref{rul:final}.

By the end of Coloring Rule~\ref{rul:final}, we have colored the representatives of all edges in $B_1$.
We may have also colored some additional edges of $G$ that are not in $B_1$,  namely the shortcut edges (during Coloring Rule \ref{rul:final}).
We next show that the vertices in $B_1$ are now rainbow connected through these colored edges.
\begin{lemma}
	\label{lem:b1path}
	For any pair of vertices $v_1,v_2\in V(G)\setminus L_S$, $P_{ab}$ is a rainbow path between $v_1$ and $v_2$ in $G$ and uses only colors in $[f]$. 
\end{lemma}
\begin{proof}
	There are no more incomplete tree-vertices because Coloring Rule~\ref{rul:final} is applicable on each incomplete tree-vertex and each tree-vertex on which the rule is applied is completed during the rule.
	This means there are no uncolored 2-edges in $B_1$.
	Also, Coloring Rule~\ref{rul:1edge} colors all $1$-edges in $B_1$.	
Thus, each edge in $B_1$, and hence their representatives in $G$, have been colored.

	Whenever an edge in $B_1$ is colored by a coloring rule and if it is in $Q_{ab}$, we have added exactly one of its representatives to $P_{ab}$ in the proceeding path rule, except possibly Path Rule~\ref{prul:final} where we might have added a shortcut edge instead.
	In the case when a shortcut edge is added, the shortcut edge shortcuts the two consecutive edges in $Q_{ab}$ whose representatives were not added to $P_{ab}$ and hence the path is not broken.

	Also, whenever a tree $T$ has two edges of $P_{ab}$ incident on it, we have added the path between the endpoints of the edges in the tree to $P_{ab}$.
	And, whenever a tree $T$ with $a\in V(T)$ has one edge of $P_{ab}$ incident on it, we have added the path between the endpoints of the edge and $a$ in the tree to $P_{ab}$.
	Similarly, whenever a tree $T$ with $b\in V(T)$ has one edge of $P_{ab}$ incident on it, we have added the path between the endpoints of the edge and $b$ in the tree to $P_{ab}$.
	If there is a tree $T$ with $a,b\in V(T)$ we added the path between the endpoints of $a$ and $b$ in the tree to $P_{ab}$ during Path Rule~\ref{prul:forest}.
	It follows that  
 $P_{ab}$ is indeed a path between $a$ and $b$ in $G$.
 Since Invariant~\ref{inv:part-rainbow} holds, we know that $P_{ab}$ is a rainbow path.
 Since we have used only the colors from $1$ to $f$ so far, the lemma follows.
	%Let $P'$ be the unique path in $\und{B_1}$ between $v_1$ and $v_2$.
	%We will use $P'$ as a guide to construct a rainbow path $P$ between $v_1$ and $v_2$ in $G$.
	%For each edge in $P'$, we will use a representative of the edge in $P$, 
	%for $1$-edges it is clear as there is only $1$ representative, 
	%whereas for 2-edges we need to make a choice on which representative to pick,
	%and we need to make this choice cleverly so that the path becomes rainbow.
	%For a tree vertex in $x_T$ in $P'$,  say $ux_T$ and $x_Tv$ are the edges of $P'$ incident on $x_T$, 
	%say we choose $(ux_T)_1$ and $(x_Tv)_1$ in $P$,
	%say $w$ and $z$ are the endpoints of $(ux_T)_1$ and $(x_Tv)_1$ in $T$,
	%then in $P$, we also need to include a path from $w$ to $z$.
	%We start with $P$ being just the vertex $v_1$.
%We maintain a current vertex in $G$, denoted by $v_G$, which is initialized to $v_1$.
%We also maintain a current vertex of $B_1$ denoted by $v_B$ which we now initialize as follows: 
%if $v_1\in S$ then $v_B:=v_1$, otherwise $v_1\in V(T)$ for some $T\in \calT$, take $v_B:=x_T$. 
%We will maintain the invariant that $v_G=v_B$ when $v_B\in S$ and $v_G\in V(f_{\calT}(v_B))$ when $v_B\in V_{\calT}$.
%
%Repeat the following until $v_G=v_2$:\\
%Let $u$ be the next vertex in $P'$. 
%Let $e$ be the edge $v_Bu$.
%If $v_B\notin S$, let $T=f_{\calT}(v_B)$: 
%then $v_G\in V(T)$ due to our invariant;
%let $w_1$ and $w_2$ be the foots of $(e)_1$ and $(e)_2$ respectively in $T$,
%\underline{Case 1}: $e$ is a $1$-edge \\
%\underline{Case 1.1}: $e$ is a $1$-edge \\
\end{proof}

So, now we only need to worry about how to rainbow connect vertices in $L_S$ between themselves and to the other vertices.
For this, we give the following coloring rule.

\begin{rul}
	\label{rul:leaves}
	For each $v\in L_S$, let $e$ be the unique 2-edge incident on $v$ which exists by Lemma~\ref{lem:vertex2edge}. Color $(e)_1$ with $g_1 = f+1$ and $(e)_2$ with $g_2 = f+2$. (Recall that $g_1$ and $g_2$ are the global surplus colors).
\end{rul}

Now, we complete the proof of the main theorem.

\begin{proof}[Proof of Theorem~\ref{thm:rc-fc}]
Consider any pair of vertices $a_1,a_2\in V(G)$.
If $a_1\in L_S$, let $e_1$ be the edge incident on $a_1$ that is colored with $g_1$, and
let $a$ be the other end of $e_1$.
If $a_1\notin L_S$, let $a=a_1$. 
If $a_2\in L_S$, let $e_2$ be the edge incident on $a_2$ that is colored with $g_2$, and
let $b$ be the other end of $e_2$.
If $a_2\notin L_S$, let $b=a_2$. 
We know there is a rainbow path $P_{ab}$ from $a$ to $b$ that uses only colors in $[f]$ due to Lemma~\ref{lem:b1path}.
We define path $P$ as follows.
If $a_1,a_2\in L_S$, then
$P:=a_1aP_{ab}ba_2$. 
If $a_1\in L_S$ but $a_2\notin L_S$, then
$P:=a_1aP_{ab}$. 
If $a_2\in L_S$ but $a_1\notin L_S$, then
$P:=P_{ab}ba_2$. 
If $a_1,a_2\notin L_S$, then
$P:=P_{ab}$. 
It is clear from the construction that $P$ is a path between $a_1$ and $a_2$.
Since edge $a_1a$ is colored with $g_1=f+1$, edge $ba_2$ is colored with $g_2=f+2$,
and path $P_{ab}$ uses only colors in $[f]$, 
the path $P$ is indeed a rainbow path.
\end{proof}

\section{Conclusions} 
\label{sec:conclusion}
We gave an upper bound of $\f(G)+2$ on $\rc(G)$, strengthening the intuition that tree-like and dominating structures are helpful in rainbow-connecting graphs. Our bound is tight up to an additive factor of $3$ as shown by any tree.
The question remains whether the bound can be improved to $f(G)-1$ so that the bound is tight even with respect to additive factors.
Also, then the bound would be a strict improvement over the bound $n-1$ obtained by coloring the edges of a spanning tree in distinct colors.
From our insight developed during the current work, we conjecture such a slightly stronger bound.
\begin{conjecture}
\label{con:tight}
A connected graph $G$ has $\rc(G) \leq \f(G) - 1$.
\end{conjecture}
We expect that proving this conjecture requires further extensive case analysis. 
Further, we note that Lauri~\cite{Lauri-phd} proposed the following stronger version of the above conjecture, discovered using the automated conjecture-making software GraPHedron~\cite{Melot2008}.
\begin{conjecture}[\cite{Lauri-phd}]
\label{con:forest}
A connected graph $G$ has $\src(G) \leq \f(G) - 1$.
\end{conjecture}

%\paragraph{Acknowledgements} We thank Pierre Hauweele for his help in running the conjecture-making system GraPHedron.
\paragraph{{\bf Acknowledgement.}} We thank Kurt Mehlhorn for funding the research visit of Erik Jan van Leeuwen and Juho Lauri to Max Planck Institute for Informatics, Saarbr\"{u}cken, which led to this work.
\bibliographystyle{abbrv}
\bibliography{forest-number}
\end{document}